\documentclass{article}
\usepackage[letterpaper,top=2cm,bottom=2cm,left=3cm,right=3cm,marginparwidth=1.75cm]{geometry}
\usepackage{graphicx}
\usepackage{amssymb,latexsym}

\usepackage{setspace}

\usepackage{amsmath}

\usepackage{mathrsfs,amsfonts}

\usepackage{amsthm,amsxtra}

\usepackage[final]{showkeys}

\usepackage{stmaryrd}

\usepackage[ruled]{algorithm2e}
\usepackage{mathtools}

\newif\ifPDF
\ifx\pdfoutput\undefined
\PDFfalse
\else
\ifnum\pdfoutput > 0
\PDFtrue
\else
\PDFfalse
\fi
\fi

\ifPDF
\usepackage{pdftricks}
\begin{psinputs}
	\usepackage{pstricks}
	\usepackage{pstcol}
	\usepackage{pst-plot}
	\usepackage{pst-tree}
	\usepackage{pst-eps}
	\usepackage{multido}
	\usepackage{pst-node}
	\usepackage{pst-eps}
\end{psinputs}
\else
\usepackage{pstricks}
\fi

\ifPDF
\usepackage[debug,pdftex,colorlinks=true, 
linkcolor=blue, bookmarksopen=false,pdfpagelabels]{hyperref}
\else
\usepackage[dvips]{hyperref}
\fi

\usepackage[openbib]{currvita}

\usepackage{fancyhdr}

\newtheorem{theorem}{Theorem}[section]
\newtheorem{lemma}[theorem]{Lemma}
\newtheorem{definition}[theorem]{Definition}
\newtheorem{example}[theorem]{Example}

\newtheorem{proposition}[theorem]{Proposition} 
\newtheorem{remark}[theorem]{Remark}

\newcommand{\supp}{\operatorname{supp}}

\DeclareMathOperator*{\argmin}{\arg\min}

 \newcommand{\bbE}{\mathbb E}
 \newcommand{\bbN}{\mathbb N}

\newcommand{\bLambda}{\boldsymbol \Lambda}

\newcommand{\ba}{\mathbf a} \newcommand{\bb}{\mathbf b}
\newcommand{\bc}{\mathbf c}

\newcommand{\bi}{\mathbf i}

\newcommand{\by}{\mathbf y}  
\newcommand{\bA}{\mathbf A} 
  
 \newcommand{\bF}{\mathbf F}
\newcommand{\bG}{\mathbf G}

 \newcommand{\cD}{\mathcal D} 
 \newcommand{\cF}{\mathcal F}
 
\newcommand{\cI}{\mathcal I}

\newcommand{\cO}{\mathcal O}

\newcommand{\cU}{\mathcal U}

\newcommand{\abs}[1]{|#1|}
\newcommand{\norm}[1]{\|#1\|}

\makeatletter
\newcommand{\chapterauthor}[1]{%
	{\parindent0pt\vspace*{-25pt}%
		\linespread{1.1}\large\scshape#1%
		\par\nobreak\vspace*{35pt}}
	\@afterheading%
}
\makeatother

\usepackage{esint} 
\newcommand{\drift}{\operatorname{drift}}
\newcommand{\R}{\mathbb{R}}
\newcommand{\diffuse}{\operatorname{diffuse}}

\newtheorem{assumption}{Assumption}

\newcommand{\starbc}{\mathbf{c}^*}
\usepackage{booktabs} 
\usepackage{siunitx}  
\usepackage{multirow}
\title{Stoch-IDENT: New Method and Mathematical Analysis for Identifying  SPDEs from Data}
\author{Jianbo Cui, Roy Y. He}
\date{}

\begin{document}

\maketitle

\begin{abstract}
    In this paper, we propose Stoch-IDENT, a novel framework for identifying stochastic partial differential equations (SPDEs) from observational data. Our method can handle linear and nonlinear high-order SPDEs driven by time-dependent Wiener processes, accommodating both additive and multiplicative noise structures. To investigate the identifiability of SPDEs from trajectory data, we analyze the spectral properties of the solution's mean and covariance for linear SPDEs with constant coefficients, as well as the dimension of the solution space for parabolic and hyperbolic types, generalizing the identifiability theory for deterministic PDEs. Algorithmically, the drift term is identified via a sample-mean generalization of existing methods for PDE identification. For the diffusion term, we formulate a sparse regression problem with quadratic measurements induced from drift residuals and feature covariances. To address this challenging non-convex and non-smooth optimization, we develop a new greedy algorithm, Quadratic Subspace Pursuit (QSP), and prove that QSP enjoys stable support recovery under certain conditions. We validate Stoch-IDENT on various SPDEs, demonstrating its effectiveness through quantitative and qualitative evaluations.
    \end{abstract}
\section{Introduction}
Data-driven discovery of Partial Differential Equation (PDE) has been catching much attention since the introduction of methods such as PDE-FIND~\cite{rudy2017data} and IDENT~\cite{kang2021ident}. Generalizing the sparsity-inducing regression framework, Sparse identification of nonlinear dynamics (SINDy),  initiated by Brunton et al.~\cite{brunton2016discovering}, these methods find the differential operators that compose the models governing the dynamical behaviors of the observational data, yielding valuable techniques for model discovery in  biological sciences~\cite{ducos2025evaluating}  and geosciences~\cite{song2024towards}. 

Various advancements have been developed to address diverse application problems involving PDEs~\cite{messenger2021weak,berg2019data,zhang2021robust,xu2020dlga}. To mitigate noise amplification caused by numerical differentiation and overcome model redundancy, Robust-IDENT~\cite{he2022robust} was introduced. This framework employs an $\ell_0$-constraint formulation, efficiently solved using the Subspace Pursuit (SP) algorithm~\cite{dai2009subspace}. Due to its effectiveness, the framework has been widely adopted in recent PDE identification methods, including Weak-IDENT~\cite{tang2023weakident}  for  integral-form PDEs,  CaSLR~\cite{he2024much} for patch-based PDE identification, Fourier-IDENT~\cite{tang2023fourier} for identification in the frequency domain, GP-IDENT~\cite{he2025group} and WG-IDENT~\cite{tang2025wg} for PDEs with varying coefficients, and its integral-form variation. A comprehensive review  can be found in~\cite{he2025ident}. Additionally, deep learning-based approaches, such as those in~\cite{long2019pde,du2024discover,stephany2024weak}, offer alternative solutions.

However, a critical limitation of these deterministic frameworks lies in their inherent inability to fully capture the uncertainties and random fluctuations that are pervasive in real-world phenomena. The ubiquitous nature of stochasticity, whether intrinsic to physical systems or induced by external perturbations,  often necessitates a different modeling formulation~\cite{boninsegna2018sparse,wang2022data,tripura2023sparse,li2021data}. Meanwhile, there is increasing interest in data-driven identification of stochastic partial differential equations (SPDEs) to enable modeling random behaviors. 
To the best of our knowledge, the work  by Mathpati et al.~\cite{mathpati2024discovering} is the first study addressing data-driven identification of SPDEs. They extended the Kramers-Moyal expansion and applied the SS prior for sparsity while finding the model terms via a variational Bayesian framework using the Kullback–Leibler (KL) divergence.  More recently, Tripura et al.~\cite{tripura2024data} considered generalizing their framework with tailored dictionaries suitable for identifying Hamiltonian and Lagrangian of physical systems. Gerardos and Ronceray~\cite{gerardos2025principled} introduced  Parsimonious Stochastic Inference (PASTIS) as an alternative to Akaike's Information Criterion (AIC), which was adopted in SINDy-based methods. 

The aforementioned methods for SPDE identification are mainly based on the Kramers-Moyal expansion and focus on identifying the \textit{squared} diffusion part rather than the diffusion part itself. Indeed, recovering the latter, whether in a pathwise or an expectation sense, requires nonlinear regression. This is a more challenging procedure that has not been addressed in the previous work. Although these methods have been validated on various nonlinear and high-order SPDEs driven by additive noises, their applicability to SPDEs with multiplicative noises remains underexplored in the existing literature.

Another important yet unresolved issue is the theoretical understanding of when and why SPDEs can be uniquely identified from data. For deterministic PDEs, He et al.~\cite{he2024much} established a foundational theory, demonstrating that the identifiability of a PDE depends on the dimension of its solution space. Their work reveals why parabolic PDEs are inherently more challenging to identify than hyperbolic ones. In the specific case of linear PDEs with constant coefficients, they further proved that just two time snapshots of the solution trajectory suffice for unique identification, provided the initial condition contains sufficiently rich Fourier modes. However, the extension to SPDEs is non-trivial, as their inherent stochasticity introduces new complexities that necessitate a distinct analytical framework.

In this work, we address these challenges through both theoretical and algorithmic contributions. Specifically, we propose \textbf{Stoch-IDENT}, a novel framework  for identifying SPDEs from path sample data. Figure~\ref{fig_workflow} shows an overview of our method. Stoch-IDENT is capable of handling both linear and nonlinear high-order SPDEs driven by time-dependent Wiener processes with additive or multiplicative noise structures. 

On the theoretical side, 
we present a general identifiability theory for linear SPDEs with constant coefficients, which justifies the conditions under which SPDEs can be uniquely recovered from trajectory data. These results are independent of any specific algorithm. For the drift term, we show that identifiability holds when the initial data contains sufficiently many nontrivial Fourier modes. For the diffusion term, we exploit the covariance structure of stochastic integrals to show that the diffusion operator can be determined uniquely up to equivalence classes, providing the theoretical foundation for robust recovery  even in the presence of stochasticity. Furthermore, we analyze the solution space spanned by trajectories for both parabolic and hyperbolic SPDEs. For the parabolic case, we show that the solution space can be approximated to accuracy $\mathcal{O}(\epsilon)$ by a linear space of dimension $\mathcal{O}(|\log(\epsilon)|)^2$ in the averaged sense (Section~\ref{parabolic}). For the hyperbolic case, via stochastic characteristics, we obtain a similar polynomial-dimensional approximation result to that in the deterministic setting~\cite{he2024much} (Section~\ref{hyperbolic}).

On the algorithmic side, we exploit $\ell_0$-constrained sparse regression, in contrast to existing methods based on the Kramers--Moyal expansion or Bayesian frameworks. Our strategy follows a candidate generation and model validation pipeline~\cite{rudy2017data,kang2021ident,he2022robust}, for which we propose a new validation method based on time integration (Section~\ref{sec:model-selection}).
For the drift term (Section~\ref{sec_candidate_drift}), we reduce identification to deterministic PDE identification by taking path expectations, enabling us to leverage the  existing methods such as Robust-IDENT~\cite{he2022robust}. For the diffusion term (Section~\ref{sec_candidate_diffuse}), we formulate an $\ell_0$-constrained sparse regression problem with quadratic measurements~\cite{fan2018variable}. To address this combinatorially complex nonlinear problem, we develop \textbf{Quadratic Subspace Pursuit (QSP)} (Algorithm~\ref{QSPalgo}), a new greedy algorithm that extends the expansion--shrinkage paradigm~\cite{he2025group} of SP~\cite{dai2009subspace} to nonlinear regression. We further prove that QSP is guaranteed to maintain the true support throughout its iterations under certain conditions (Theorem~\ref{thm:main}). While we demonstrate its effectiveness in the context of SPDE identification, we highlight that QSP is also applicable to broader problems such as phase retrieval~\cite{balan2006signal}. As the identification of the diffusion term requires additional computation for the feature covariances, we design an efficient statistical test (Section~\ref{sec_stat_detect}) to detect the pure additive noise case, which reduces diffusion term identification to a single parameter estimation problem, bypassing the subsequent quadratic regression. 

To summarize, our contributions in this paper include:
\begin{enumerate}
    \item A novel framework, Stoch-IDENT, for identifying SPDEs driven by time-dependent Wiener processes with additive and multiplicative noise structures. The proposed framework exploits sample means and feature covariances to accurately recover the drift and diffusion terms  directly from sample trajectories.
    \item An in-depth identifiability analysis for linear SPDEs with constant coefficients, establishing conditions under which the drift and diffusion terms can be uniquely identified, independent of any specific algorithm. For parabolic and hyperbolic SPDEs, we analyze the dimensions of the solution trajectory spaces, generalizing the identifiability theory for deterministic PDEs in~\cite{he2024much}.

    \item A new greedy algorithm, Quadratic Subspace Pursuit (QSP), for sparse regression problems with quadratic measurements, developed for identifying the diffusion term. We prove that under certain conditions, QSP is guaranteed to maintain the true support throughout its iterations. The proposed algorithm is also applicable to broader problems such as phase retrieval.

\item Two algorithmic enhancements: (i) a statistical test for the pure additive noise case, which bypasses the quadratic regression upon detection; and (ii) a time-integration-based model selection method for identifying the best candidate SPDE.
    
    \item Numerical experiments validating Stoch-IDENT across various SPDEs with additive and multiplicative noise.
\end{enumerate}

This paper is organized as follows. In Section~\ref{sec_proposed_method}, we present the framework of Stoch-IDENT for identifying SPDEs from data. 
In Section~\ref{sec_identifiability}, we analyze the identifiability conditions for linear parabolic and hyperbolic SPDEs with constant coefficients.  In Section~\ref{sec_algorithm}, we describe the algorithmic components of Stoch-IDENT, including Subspace Pursuit (SP) for drift candidate generation, the statistical test for additive noise, and the proposed Quadratic Subspace Pursuit (QSP) for diffusion identification. In Section~\ref{sec_numerical}, we present numerical results to validate our method. We conclude this paper in Section~\ref{sec_conclusion}.

\begin{figure}[t]
\centering
\includegraphics[width=0.8\textwidth]{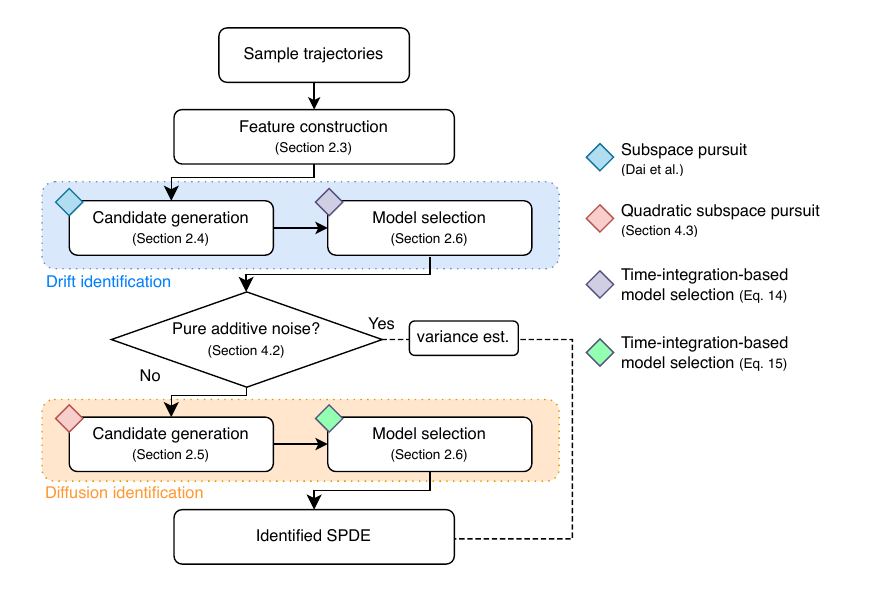}
\caption{Workflow of the proposed Stoch-IDENT for identifying SPDEs from trajectory data.}\label{fig_workflow}
\end{figure}
\section{Stoch-IDENT Framework}\label{sec_proposed_method}
 We propose \textbf{Stoch-IDENT}  to find an SPDE whose strong solutions closely approximate the observed trajectories. Figure~\ref{fig_workflow} illustrates our proposed workflow.  
We provide a high-level overview of the framework here, while deferring details to Section~\ref{sec_algorithm}.

\subsection{Mathematical formulation of the problem}\label{sec_math_formulation}

Let $\mathcal{U}=\{U_n: \Gamma\to\mathbb{R}\}_{n=1}^N$ be a collection of observed trajectories, where $\Gamma$ is a grid on the observation time-space domain $[0,T]\times \mathcal{D}\subset\mathbb{R}\times\mathbb{R}^d$ (with $d\geq 1$ denoting the space dimension), and $N$ is the number of sample trajectories. Assume that each $U_n$ can be approximated by a strong solution of an unknown SPDE in It\^o's sense driven by a time-dependent Wiener process $W$, of the form
\begin{equation}
du = \sum_{k=1}^Ka_kF_k\,dt + \sum_{j=1}^J b_jG_jdW(t)\;,\label{eq_SPDE_template}
\end{equation}
subject to appropriate boundary conditions.
{For the well-posedness results of SPDEs, we refer to \cite{MR1207136} and references therein.}
In the generic form~\eqref{eq_SPDE_template}, the drift term is a linear combination of $K$ \textit{candidate drift features} $\{F_k: [0,T]\times \mathcal{D}\to\mathbb{R}\}_{k=1}^K$, and the diffusion term is a linear combination of $J$ \textit{candidate diffusion features} $\{G_j: [0,T]\times \mathcal{D}\to\mathbb{R}\}_{j=1}^J$. Each feature is assumed to be a linear or nonlinear transformation of $u$ and its spatial derivatives, e.g., $u_x$, $uu_x$, or $\sin(u)$. {The choice of dictionary can be as general as polynomial types~\cite{kang2021ident,he2022robust,tang2023weakident}, or incorporate additional terms based on prior knowledge~\cite{brunton2016discovering,messenger2021weak,tang2026priorident}.} For $k=1,2,\dots,K$, we call $a_k\in\mathbb{R}$ the feature coefficient of the $k$-th drift feature; and for $j=1,2,\dots, J$, we call $b_j\in\mathbb{R}$ the feature coefficient of the $j$-th diffusion feature. When $a_k=0$ for some $k$, the $k$-th drift feature is inactive; otherwise, it is active. The same applies to the diffusion features.

The approximation of the observed trajectories by the solution is understood in the \textit{strong sense}. More precisely, let $(\Omega, \mathcal{F}, (\mathcal{F})_{t\geq 0},\mathbb{P})$ be the completely filtered probability space where~\eqref{eq_SPDE_template} is defined, and let $\Gamma = \{(t_i,x_m): i=0,1,\dots,I,m=0,1,\dots,M\}$ be the set of grid points. Let $\{\omega_n\}_{n=1}^N$ be $N$ samples from $\Omega$, and let $u_n:=u(\omega_n): [0,T]\times\mathcal{D}$ be a strong solution of~\eqref{eq_SPDE_template} driven by the sample path $W(t, \omega_n)$. For $n=1,2,\dots,N$, the  trajectory $U_n$ is obtained by sampling the solution:
\begin{equation*}
U_n(t_i,x_m) = u_n(t_i,x_m)\;,\quad   ~i=0,1,\dots,I; m=0,1,\dots,M.
\end{equation*}
 Our objective is to  identify the active drift and diffusion features and recover the corresponding feature coefficients based on the observed data. 
 This framework can be readily adapted to the identification of SDEs or more general stochastic systems. For further details on possible extensions,  we refer the readers to~\cite{kang2021ident,he2022robust, tang2023weakident}. 

 It should be noted that we focus on \textit{strong solutions} primarily to justify pathwise evaluation at grid points and sample‑path regression. We fix the probability space and filtration both to avoid technical complications and because, for simplicity, the data are generated from a fixed simulator. This setting also serves as a starting point for developing an identification theory and algorithms for SPDEs. One may instead formulate the identification problem in terms of stochastically weak or variational solutions (see \cite{MR3410409} for the relevant definitions).

\subsection{Overview of the proposed Stoch-IDENT framework}\label{sec_overview_SDE_ident}
Suppose the underlying SPDE can be represented by the generic form~\eqref{eq_SPDE_template}. We propose to identify the active features by the following major steps:\\[3pt]
\noindent \textbf{Step 1. Identify  drift features from sample means.} By leveraging the martingale property of the stochastic integral in the It\^o's sense, the diffusion part of~\eqref{eq_SPDE_template} vanishes when taking expectations. Thus, we first search for the active drift features and estimate the associated coefficients by exploiting the following relation:
\begin{equation}\label{eq_drift_identification_relation}
d\mathbb{E}[u] = \sum_{k=1}^K a_k\mathbb{E}[F_k]\,dt\;.
\end{equation}
To ensure interpretability and focus on the most relevant features, we require that only a few features are active, i.e., $(a_1,\dots,a_K)$ is sparse. This task can be addressed by any existing  identification methods~\cite{kang2021ident,he2022robust,rudy2017data}. We explain the details in Section~\ref{sec_candidate_drift} and Section~\ref{sec_algorithm}.\\[2pt]
\noindent\textbf{Step 2. Identify diffusion features from squared drift residuals.} To find the diffusion features and the associated coefficients $b_j$, we  consider the following relation:
{\small 
\begin{align}\label{eq_diffuse_identification_relation}
\mathbb E\left[\left(D_t^+u(t_i,x)-\int_{t_i}^{t_{i+1}}\sum_{k=1}^K a_k F_k(t,x)\,dt\right)^2\right]=\mathbb E \left[\left(\int_{t_i}^{t_{i+1}}\sum_{j=1}^J b_j G_j(t,x)\,dW(t)\right)^2 \right]. 
\end{align}}for any $x\in\cD$ and $i=1,\dots,I$, where $t_1<t_2<\cdots<t_I$ is a sequence of times in $[0,T]$, and $D_t^+u(t_i,x):=u(t_{i+1}, x)-u(t_i,x)$ is the forward time difference. If the estimated coefficients $\{\widehat{a}_k\}_{k=1}^K$  from Step 1 are close to the true values, we expect that~\eqref{eq_diffuse_identification_relation} holds approximately after replacing $a_k$ with $\widehat{a}_k$.  Note that~\eqref{eq_drift_identification_relation} is only linear in $\{a_k\}_{k=1}^K$, where as~\eqref{eq_diffuse_identification_relation} is quadratic in  $\{b_j\}_{j=1}^J$. This is challenging especially when sparsity is desired, as presented in Section~\ref{sec_candidate_diffuse}. For this, we propose a novel algorithm  in Section~\ref{sec:QSP}.

After obtaining  $\widehat{\ba}=(\widehat{a}_1,\dots,\widehat{a}_K)\in\mathbb{R}^K$ from~\eqref{eq_drift_identification_relation} and  $\widehat{\bb}=(\widehat{b}_1,\dots,\widehat{b}_J)\in\mathbb{R}^J$ from~\eqref{eq_diffuse_identification_relation}, we express the identified SPDE as: 
\begin{equation}
du = \sum_{k=1}^K\widehat{a}_kF_k\,dt + \sum_{j=1}^J \widehat{b}_jG_jdW(t)\;.\label{eq_SPDE_identified}
\end{equation}
If the identified model~\eqref{eq_SPDE_identified} is well-posed,  by~\eqref{eq_drift_identification_relation} and~\eqref{eq_diffuse_identification_relation}, we expect that if $\widehat{\ba}$ and $\widehat{\bb}$ are close to the true coefficients $\ba=(a_1,\dots,a_K)$ and  $\bb=(b_1,\dots,b_J)$, respectively, then increments of solutions of~\eqref{eq_SPDE_template}  and~\eqref{eq_SPDE_identified} will have similar means and variances.  In Section~\ref{sec_identifiability}, We will theoretically justify the feasibility of identification based on the relations~\eqref{eq_drift_identification_relation} and~\eqref{eq_diffuse_identification_relation}, under conditions on the richness of the dynamics exhibited by the observed trajectories.

In general, the identification algorithm itself does not currently enforce well‑posedness constraints. One may use simple post‑hoc checks (e.g. coefficient nonnegativity or bounds on high‑order derivatives) that can be used to discard obviously ill‑posed models, and we mark this as an important direction for future work.

\subsection{Drift and diffusion feature systems}\label{sec_feature_system}
We mainly focus on the case where the data $U_n:\Gamma\to\mathbb{R}$ is assumed to be sampled from a strong solution $u_n:=u(\omega_n): [0,T]\times\cD\to\mathbb{R}$ of~\eqref{eq_SPDE_template} with $\cD\subset\mathbb{R}^d$ for some integer $d\geq 1$ and $\omega_n\in \Omega$ for $n=1,\dots,N$. The extension to multi-dimensional outputs or stochastic systems is straightforward.  The algorithmic procedures and implementation details are described in Section~\ref{sec_algorithm}. Consider the It\^{o}'s integral form of~\eqref{eq_SPDE_template} evaluated at  equidistant time points $0=t_0<t_1<\dots<t_I=T$ as follows
 \begin{equation}\label{eq_SPDE_integral_form} 
 D_t^{-}u(t_i,x)= \sum_{k=1}^K
 a_k\int_{t_{i-1}}^{t_i}F_k(t,x)\,dt +\sum_{j=1}^J b_j\int_{t_{i-1}}^{t_i} G_j(t,x)\,dW(t)\;,
 \end{equation}
where  $D_t^{-}u(t_i,x):=u(t_i, x) - u(t_{i-1},x)$  is the backward time difference, $i=1,\dots,I$, and $x\in\cD$. Applying the left Riemann sum and Euler-Maruyama scheme to approximate the deterministic and stochastic integrals in~\eqref{eq_SPDE_integral_form}, respectively, and  then formally taking the expectation    yields: 
\begin{equation}\label{eq_deterministic_part}
\bbE[u(t_{i},x)]-\bbE[u(t_{i-1},x)] =\Delta t\sum_{k=1}^Ka_k \bbE[F_k(t_{i-1},x)]+{\cO((\Delta t)^{2}})\;.
\end{equation}
For a set of points $\{x_m\}_{m=1}^M\subset\cD$, we define the \textit{drift feature matrix} as
\begin{equation}\label{eq_drift_feature_matrix}
\bF := \Delta t\cdot \begin{pmatrix}
  \bbE[F_1(t_1,x_1)]&\bbE[F_2(t_1,x_1)]&\cdots& \bbE[F_K(t_1,x_1)]\\
  \bbE[F_1(t_1,x_2)]&\bbE[F_2(t_1,x_2)]&\cdots&\bbE[F_K(t_1,x_2)]\\
  \vdots&\vdots&\ddots&\vdots\\
  \bbE[F_1(t_{I},x_M)]&\bbE[F_2(t_{I},x_M)]&\cdots&\bbE[F_K(t_{I},x_M)]\\
\end{pmatrix}\in\mathbb{R}^{IM\times K}\;,
\end{equation}
where each column is obtained by flattening the expected values of candidate drift  features  evaluated at discrete time-space points.  Define the \textit{drift response vector} as
\begin{equation}\label{eq_drift_feature_response}
\by:=\begin{pmatrix}
\bbE[u(t_{1},x_1)]-\bbE[u(t_{0},x_1)]\\
\bbE[u(t_{1},x_2)]-\bbE[u(t_{0},x_2)]\\
\vdots\\
\bbE[u(t_{I},x_M)]-\bbE[u(t_{I-1},x_M)]\\
\end{pmatrix}\in\mathbb{R}^{IM}\;.
\end{equation}
The relation~\eqref{eq_deterministic_part} can thus be compactly  expressed  as $\by = \bF \ba + {\cO((\Delta t)^{2}})$, where $\ba=(a_1,\dots,a_K)^\top\in\mathbb{R}^K$ is the unknown feature coefficient vector.

According to~\eqref{eq_diffuse_identification_relation}, the expectation of the squared drift residual 
\begin{equation*}
r^2(t_i,x,\ba):=\left(D_t^{-}u(t_i,x)-\Delta t \sum_{k=1}^K a_kF_k(t_{i-1},x)\right)^2
\end{equation*}
 should be close to

\begin{equation}\label{eq:covariance-structure}
\mathbb E \left(\sum_{j=1}^J b_jG_j(t_{i-1},x)(W(t_i)-W(t_{i-1}))\right)^2 = \Delta t \sum_{j=1}^J\sum_{s=1}^Jb_jb_s\bbE\left[ G_j(t_{i-1},x)\cdot G_s(t_{i-1},x)\right]\;
\end{equation}
with the reminder term $\mathcal O((\Delta t)^{2})$ (this is due to mean square error for approximating the stochastic integral, see e.g. \cite{MR4241457}),
for any $i=1,\dots,I$ and $ x\in\cD$. 
We define the $i$-th \textit{diffuse feature matrix} as
{\small
\begin{equation}\label{eq_diffuse_matrix}
\bG_{i} := \Delta t\cdot \begin{pmatrix}\bbE\left[\mathlarger{\fint}_{\cD}G_1(t_{i-1},y)\cdot G_1(t_{i-1},y)\,dy\right]&\cdots& \bbE\left[\mathlarger{\fint}_{\cD}G_1(t_{i-1},y)\cdot G_J(t_{i-1},y)\,dy\right]\\
  \vdots&\ddots&\vdots\\
  \bbE\left[\mathlarger{\fint}_{\cD}G_J(t_{i-1},y)\cdot G_1(t_{i-1},y)\,dy\right]&\cdots& \bbE\left[\mathlarger{\fint}_{\cD}G_J(t_{i-1},y)\cdot G_J(t_{i-1},y)\,dy\right]
\end{pmatrix}\;,
\end{equation}
}where $\mathlarger{\fint}_{\cD}g(y)\,dy:=|\cD|^{-1}\cdot\int_{\cD}g(y)\,dy$ denotes the average integral and $|\cD|$ is the Lebesgue measure of $\cD$. Moreover, we define the $i$-th \textit{diffuse response} as
\begin{equation}\label{eq_diffuse_response}
\zeta_i(\ba) = \bbE\left[\mathlarger{\fint}_{\cD} r^2(t_i,y,\ba)\,dy\right]\in\mathbb{R}\;, i=1,\dots,I\;.
\end{equation}
Since only discrete observations of the sampled strong solutions are available, the feature matrices and responses above must be approximated from data. We call the drift feature matrix together with the drift feature response a \textit{drift feature system}, and the diffusion feature matrices with the diffusion responses a \textit{diffusion feature system}; their sample approximations are discussed in Section~\ref{sec_mean_approximation}.

\subsection{Generation of candidate drift models}\label{sec_candidate_drift}
Instead of directly searching for the optimal drift model, we propose generating a sequence of \textit{candidate models} and then selecting the optimal one among them via the validation method introduced in Section~\ref{sec:model-selection}. To produce a pool of candidates, we consider the problem
\vspace{-1mm}
\begin{equation}\label{eq_L0_drift_select_greedy}
\begin{aligned}
\min_{\bc\in\mathbb{R}^{K}} &\|\bF\bc - \by\|_2^2 \\
\text{s.t.}&\|\bc\|_0 =  k
\end{aligned}
\end{equation}
for $k = 1,2,\dots,K$, where $\|\bc\|_0:=|\operatorname{supp}(\bc)|$ counts the number of non-zero entries of a vector $\bc$. A solution of~\eqref{eq_L0_drift_select_greedy} yields a candidate drift model with exactly $k$ active features; thus, by solving it for each $k = 1, \ldots, K$, we obtain $K$ candidate models from which to select. This strategy of sequentially generating candidate models by sweeping a \textit{discrete} sparsity parameter was initiated in~\cite{he2022robust} for identifying constant-coefficient PDEs. 
{Other frameworks, such as the $\ell_1$-regularization of IDENT~\cite{kang2021ident} and the $\ell_0$-regularization of SINDy~\cite{brunton2016discovering}, can also be used to generate candidate drift models by sweeping a \textit{continuous} regularization parameter. Our choice is motivated by the ease of producing exactly $K$ distinct candidates without redundancy. For numerical evidence on their performance comparisons, we refer the readers to~\cite{he2022robust}.
}

\subsection{Generation of candidate diffusion models}\label{sec_candidate_diffuse}
According to~\eqref{eq_diffuse_identification_relation}, the diffuse response~\eqref{eq_diffuse_response} should remain close to a quadratic form determined by the diffuse feature matrix~\eqref{eq_diffuse_matrix}. To enforce sparsity of the diffusion coefficients in the same spirit as for the drift model, we propose to generate candidate diffusion models by considering the following $\ell_0$-constrained sparse regression problem with quadratic measurements:
\vspace{-1mm}
\begin{equation}\label{eq_L0_diffuse_select}
\begin{aligned}
\min_{\bc\in\mathbb{R}^J} &\sum_{i=1}^I\left(\bc^\top \bG_i \bc - \zeta_i(\widehat{\ba})\right)^2\\
\text{s.t.}&\|\bc\|_0=j
\end{aligned}\;,
\end{equation}
for $j=1,2,\dots, J$, where $\zeta_i(\widehat{\ba})$ is defined in~\eqref{eq_diffuse_response} with the drift coefficients replaced by the optimal $\widehat{\ba}$ selected from the candidates generated by~\eqref{eq_L0_drift_select_greedy} (see Section~\ref{sec:model-selection} for the selection criterion). By sweeping $j$ from $1$ to $J$, we obtain a list of candidate diffusion models with different sparsity levels, from which the optimal one is selected according to a method introduced in Section~\ref{sec:model-selection}.

{Our formulation differs from existing works such as~\cite{mathpati2024discovering}, which concentrate on the \textit{squared} diffusion part, i.e., the variance of the noise term, thereby reducing the problem to linear regression. In this case, existing frameworks for identifying ODEs/PDEs such as STLS~\cite{brunton2016discovering} and STRidge~\cite{rudy2017data} with $\ell_0$ regularization, and SP with $\ell_0$-constraint~\cite{he2022robust}, can be directly applied. By contrast, our quadratic formulation directly works on the diffusion part, enabling identification of more general noise structures.}
We also note that~\eqref{eq_L0_diffuse_select} is a sparse regression problem with quadratic measurements, related to the problem of phase retrieval~\cite{balan2006signal,balan2009painless,fan2018variable,chen2025solving}. Most existing methods induce sparsity through $\ell_1$-regularization. In this work, we propose to address~\eqref{eq_L0_diffuse_select} by  Quadratic Subspace Pursuit (QSP), a new greedy algorithm introduced in Section~\ref{sec:QSP}.

\subsection{Model selection via time integration}\label{sec:model-selection}
{To select the optimal models from the candidates, we need a validation method. Existing approaches include information-theoretic criteria such as AIC and BIC, as adopted in SINDy~\cite{brunton2016discovering} and PDE-FIND~\cite{rudy2017data}, and the Time Evolution Error (TEE) proposed in Robust-IDENT~\cite{he2022robust}. In our stochastic setting, AIC/BIC are unreliable as they do not account for the accumulated dynamics over time, while TEE becomes computationally prohibitive since it requires numerical time integration for each candidate. To exploit the benefits of evolution error accumulation while remaining computationally efficient, we propose to select the optimal candidates by integrating the residuals over the 
full time interval $[0,T]$, providing a global assessment of how well a candidate model reproduces the observed dynamics. Specifically, given any set of drift coefficient vectors $\{\ba_1,\ba_2,\dots, \ba_M\}\subset\mathbb{R}^K$, we evaluate 
the following drift validation score}:
\begin{equation}\label{eq_drift_score}
S_{\drift}(\ba_m) := I^{-1}\cdot\sum_{i=1}^I\int_{\cD}\left(\sum_{k=1}^K a_{m,k}\int_{0}^{t_i}\mathbb{E}\left[F_k(s,x)\right]\,ds+\mathbb E[ u(0,x)]-\bbE[u(t_i,x)]\right)^2\,dx,
\end{equation}
for $m=1,\dots,M$, where $a_{m,k}$ denotes the $k$-th entry of $\ba_m$.   Given an estimation $\widehat{\ba}:=\arg\min\limits_{m=1,\dots,M}S_{\drift}(\ba_m)$ for the drift coefficients and any set of diffusion coefficient vectors $\{\bb_1,\dots,\bb_{L}\}\subset\mathbb{R}^J$, we define 
\begin{equation}\label{eq_S_diffuse}
S_{\diffuse}(\bb_l|\widehat{\ba}) := I^{-1}\cdot\sum_{i=1}^I\int_\cD\left(\mathbb{E}\left[R^2_i(x,\widehat{\ba})\right]-\mathbb{E}\int_{0}^{t_i}\left(\sum_{j=1}^J b_{l,j}G_j(s,x)\right)^2\,ds\right)^2\,dx,
\end{equation}
for $l=1,\dots,L$, where  $R^2_i(x,\widehat{\ba}):=\left(\sum_{k=1}^K\widehat{a}_k\int_{0}^{t_i} F_k(s,x)\,ds+u(0,x)-u(t_i,x)\right)^2$ and $b_{l,j}$ is the $j$-th entry of $\bb_l$.  Notice that~\eqref{eq_S_diffuse} quantifies the integral form of~\eqref{eq_diffuse_identification_relation} for a pair of candidate diffusion and drift  coefficients. Similarly to the choice of drift coefficients,  we  set the optimal $\widehat{\bb}:=\arg\min_{l=1,\dots,L}S_{\diffuse}(\bb_l|\widehat{\ba})$ as the optimal diffusion coefficient vector associated with a drift coefficient vector $\widehat{\ba}$. Consequently, the SPDE with these optimal coefficient vectors as expressed in~\eqref{eq_SPDE_identified} is the identified model.

By the Cauchy--Schwarz inequality and the assumptions on the moment boundedness of features, we can show that if the estimated drift and diffusion coefficients are close to the true ones, then the corresponding values of~\eqref{eq_drift_score} {\color{red}and \eqref{eq_S_diffuse}} are low (see Appendix \ref{proof_lemma_S_diffuse} for the proof). 
\begin{proposition}\label{lemma_S_drift} Suppose $\ba^*=(a^*_1,\cdots, a^*_K)$ and $\bb^*=(b^*_1,\cdots, b^*_J)$ are the true drift  and diffusion coefficient vectors of the underlying SPDE, respectively. Assume that $\sup_{(t,x)\in [0,T]\times \cD} \mathbb E [|F_k(t,x)|^2]<+\infty$ and $\sup_{(t,x)\in [0,T]\times \cD} \mathbb E[ |G_j(t,x)|^2]<+\infty$  for every $k=1,\dots,K$ and $j=1,\dots, J$.  Then there exist constants $C_1,C_2>0$  such that for any $\ba\in\mathbb{R}^K$ and $\bb\in\mathbb{R}^J$, the following inequalities hold:
\begin{equation*}
S_{\drift}(\ba) \leq C_1\|\ba^*-\ba\|_2^2\;,
\end{equation*}
\begin{equation*}
S_{\diffuse}(\bb|\ba)\leq C_2\left((\| \ba\| +\|\ba^*\|)^2(\| \ba-\ba^*\|^2)
    +(\|\bb\| +\|\bb^* \|)^2\|\bb-\bb^*\|^2\right)\;.
\end{equation*}
\end{proposition}
Hence, by comparing the values of~\eqref{eq_drift_score} and \eqref{eq_S_diffuse} associated with different candidates, the smallest  values  indicate the optimal choices.

\section{Identifiability Analysis of Stoch-IDENT}\label{sec_identifiability}

In this section, we establish a general theoretical framework for the identifiability of SPDEs from trajectory data, that is, whether the drift and diffusion operators can in principle be uniquely determined (up to equivalence classes) from the mean and covariance of solution trajectories via the relations~\eqref{eq_drift_identification_relation} and~\eqref{eq_diffuse_identification_relation}. This analysis is theoretically general  and is independent of any specific estimation algorithm or numerical procedure. To remain accessible, We focus on linear SPDEs with constant coefficients, for which sharp and verifiable identifiability conditions can be established. We further assume the driving Wiener process $W$ to be time dependent, though the framework extends naturally to space time dependent processes.

Let  $\mathcal D= \mathbb T^d$ for simplicity, and denote by $\widehat f(\xi),\xi\in \mathbb Z^d$, the Fourier transformation of $f\in L^2(\mathcal D)$: 
\begin{align*}
\widehat f(\xi)=(2\pi)^{-\frac d 2}\int_{\mathcal D} e^{-\bi \xi\cdot x }f(x)dx\;,
\end{align*} 
where $\bi^2=-1$. Applying the Fourier transform to both sides of \eqref{eq_SPDE_template}, we obtain:
\begin{align*}
d \widehat u(\xi)=\sum_{k=1}^Ka_k\widehat F_k(\xi) dt +\sum_{j=1}^Jb_j \widehat{G_j}(\xi) d W (t)\;.
\end{align*}

We assume that the strong solution of \eqref{eq_SPDE_template} exists uniquely on the interval $[0,T]$ for some $T>0$ (see, e.g., \cite[Appendix G]{MR3410409}). That is, there exists a unique $\mathcal F_t$-adapted stochastic process $u$ such that 
\begin{align*}
u(t)=u(0)+\int_{0}^t \sum_{k=1}^Ka_k F_k ds +    \int_{0}^t \sum_{j=1}^J b_j G_j dW(s), \; \text{a.s.}\;,
\end{align*}
where 
the right-hand side is finite for any $t\in [0,T]$. 
This  also requires that  $u(t)$ belongs to the domain of every deterministic  and stochastic feature,  implying that 
\begin{align*}
\mathbb E \Big\|\int_0^T  \sum_{j=1}^J a_j F_j dt\Big\| +\int_0^T \mathbb E \sum_{j=1}^J b_j^2 \|G_j \|^2dt<+\infty\;,
\end{align*}
where $\|\cdot\|$ is the $L^2(\mathcal D)$-norm on the physical domain. Alternatively, we may  consider SPDEs in the Stratonovich sense
\begin{equation}
du = \sum_{k=1}^Ka_kF_k\,dt + \sum_{j=1}^J b_jG_j\circ dW(t)\;, \label{eq_SPDE_template_S}
\end{equation}
which is equivalent to a stochastic equation in the It\^o sense (see e.g., \cite[Chapter 7]{MR4241457}). The advantage of  using \eqref{eq_SPDE_template_S} is that it automatically admits the chain rule. 

\subsection{Linear SPDE identification with constant coefficients}
\label{newsec-linear-spde-identification}

We consider the identification problem related to linear SPDEs with constant coefficients,
and show that the underlying differential operators of SPDEs can be identified if the initial data or diffusion coefficient contains sufficiently rich Fourier modes.

Consider the following linear SPDE, 
\begin{align}\label{linear-spde} 
du=\mathcal L u dt +\mathcal G R \circ dW(t)\;,
\end{align}
where $\mathcal{L} = \sum_{|\alpha|=0}^{p_1} a_\alpha \partial_x^{\alpha}$ is a linear operator with constant coefficients and the total number of summands is $K$. The diffusion operator 
 $\mathcal G=\sum_{|\beta|=0}^{p_2} b_{\beta} \partial_x^{\beta}$ with the number of summands $J$. 
Suppose that either  $R$ is a given real-valued function  (additive noise case), 
or $R=u$ (multiplicative noise case). Here  $\alpha,\beta\in\bbN^d$,  $a_{\alpha}, b_{\beta}\in \mathbb R$, and $p_1,p_2\in \mathbb N$.

Applying the Fourier transform yields the system in frequency space:
\begin{align}\label{fourier-ode}
  d  \widehat u(t, \xi)=(2\pi)^{-\frac d2}\sum_{|\alpha|=0}^{p_1} a_{\alpha}(\bi \xi)^{\alpha} \widehat u(t,\xi) dt+(2\pi)^{-\frac d2}\sum_{|\beta|=0}^{p_2} b_{\beta}(\bi \xi)^{\beta} \widehat R(\xi) \circ d W(t)\;.
\end{align}
When $R$ is independent of $u$, the Stratonovich integral is equivalent to the It\^o's integral \cite{MR4628103}. 
By the Duhamel principle, for any $t_2\ge t_1\ge 0,$
\begin{align}\label{mild-fourier}
\widehat u(t_2, \xi)
&=\widehat u(t_1,\xi) \exp\Big((2\pi)^{-\frac d2}\sum_{|\alpha|=0}^{p_1} a_{\alpha} (\bi \xi)^{\alpha}(t_2-t_1)\Big)\nonumber\\
&+(2\pi)^{-\frac d2} \sum_{|\beta|=0}^{p_2}b_{\beta}(\bi\xi)^{\beta}\int_{t_1}^{t_2}\exp\Big((2\pi)^{-\frac d2}\sum_{|\alpha|=0}^{p_1} a_{\alpha} (\bi \xi)^{\alpha}(t_2-s)\Big) \widehat R(\xi) \circ dW(s)\;.
\end{align}
In the multiplicative noise case, by the chain rule, 
\begin{align}\label{mild-fourier-mul}
 {\widehat u(t_2, \xi)}&=\widehat u(t_1, \xi)     \exp\Big((2\pi)^{-\frac d2}\sum_{|\alpha|=0}^{p_1}a_{\alpha} (\bi \xi)^{\alpha}(t_2-t_1)\Big) \exp\Big((2\pi)^{-\frac d2}\sum_{|\beta|=0}^{p_2} b_{\beta} (\bi \xi)^{\beta}(W(t_2)-W(t_1))\Big)\;.
\end{align}

Our idea of identification is to firstly study the drift terms via the average of sample trajectories. Then we study the diffusion identification through the covariance information of the stochastic integral.

\subsubsection{Drift identification}
As shown in~\eqref{eq_drift_identification_relation}, the identification of the drift terms is based on properties of conditional expectation and polynomial hypersurfaces. The main difference between the multiplicative and additive noise cases is the order in which expectations are taken:   for multiplicative noise, we begin by taking the expectation of \eqref{mild-fourier-mul}, whereas for additive noise, we first take the expectation of \eqref{mild-fourier}. 
Since the proof closely follows \cite[Theorem 3.2]{he2024much}, we do not elaborate on the details here, and put the proof in the Appendix \ref{sec3-appendix}. 
\begin{proposition}\label{prop-drift-identification}
Let $\mathcal Q=\{\xi\in \mathbb Z^d: \widehat u_0(\xi)\neq 0\}$  
and suppose that  
{\small
 $|\mathcal Q|\ge \max \Bigg(\sum_{k=0}^{\lfloor\frac{p_1}{2} \rfloor} \binom{2k+d-1}{d-1} $ $,\sum_{k=0}^{\lfloor\frac {p_1-1}2 \rfloor} \binom{2k+d}{d} \Bigg), \; a.s.
 $} Suppose that $\mathcal Q$ is not located in an algebraic polynomial hypersurface of degree $\le p_1$ consisting of only
even- or only odd-order terms a.s. Then the parameters    $a_{\alpha}, |\alpha|\le p_1$, in \eqref{fourier-ode} are uniquely determined by the  solution at two time points  $u(t_1, \cdot), u(t_2, \cdot)$ when $|t_2-t_1|$ is sufficiently small. 
\end{proposition}

\subsubsection{Diffusion   identification}
In~\eqref{eq_diffuse_identification_relation}, the identification of the diffusion terms is based on the covariance information of the corresponding stochastic integral.  Since we do not assume any pathwise information of the stochastic convolution, the diffusion identification is unique only up to an equivalence class  (see the proof of Theorem \ref{diffusion-identification-covariances}). If pathwise information were available, unique identification in the classical sense would be possible. However, the accuracy of the identified coefficients from a single sample trajectory may be low, even if the correct terms in the dictionary are identified.

\begin{theorem}\label{diffusion-identification-covariances} 
Let $a_{\alpha}\in \mathbb R, |\alpha|\le p_1$, be given.
Let $\mathcal Q_1=\{\xi\in\mathbb Z^d: \widehat{R}(\xi)\neq 0\}$ in the additive noise case, and $\mathcal Q_1=\{\xi\in\mathbb Z^d: \mathbb E \widehat u_0(\xi)\neq 0\}$ in the multiplicative noise case. Assume that $|\mathcal Q_1|$ is sufficiently large and that $\mathcal Q_1$ is not located on an algebraic polynomial hypersurface of degree $\le 2p_1$. Then 
 the parameters $b_{\beta}, |\beta|\le p_2$, are uniquely determined, up to an equivalent class, by two instants $u(t_2, \cdot),u(t_1, \cdot)$ with $|t_2-t_1|>0$ sufficiently small.
\end{theorem}
\begin{proof}
The main difference between the multiplicative and additive noise cases lies in the approach to obtaining the covariance information. In the multiplicative noise case, we apply It\^o's isometry and exponential moment estimate to \eqref{mild-fourier-mul}, whereas in the additive noise case, we directly use It\^o’s isometry to analyze the covariance of the stochastic integral appearing in 
\eqref{mild-fourier}.
Since the remaining steps are similar, we present the details only for the multiplicative noise case. For completeness, we refer to Appendix \ref{sec3-appendix-add-diffusion} for the proof in the additive noise case.
 
By applying the property of conditional expectation to \eqref{mild-fourier-mul}, and  utilizing the independent increments of Brownian motion,  we obtain
\vspace{-1mm}
\begin{align}\nonumber
 \frac {\mathbb E \widehat u(t_2, \xi)}
{\mathbb E \widehat u(t_1, \xi) }
&=\exp\Big((2\pi)^{-\frac d2}\sum_{|\alpha|=0}^{p_1} a_{\alpha} (\bi \xi)^{\alpha}(t_2-t_1)\Big)\\\nonumber
&\quad \mathbb E \Bigg[\exp\Big((2\pi)^{-\frac d2}\sum_{|\beta|=0}^{p_2} b_{\beta} (\bi \xi)^{\beta}(W(t_2)-W(t_1))\Big)\Bigg]\\\label{use-exp1}
&=\exp\Big((2\pi)^{-\frac d2}\sum_{|\alpha|=0}^{p_1} a_{\alpha} (\bi \xi)^{\alpha}(t_2-t_1)\Big)\\\nonumber
&\quad \exp\Big(\frac 12 (2\pi)^{-d} \sum_{|\beta|,|\widehat \beta|=0}^{p_2} b_{\beta} b_{\widehat\beta} \bi^{|\beta|+|\widehat \beta|} \xi^{\beta+\widehat \beta}
(t_2-t_1)\Big)\;,\nonumber
\end{align}
where we have used the exponential moment of Brownian motion:
\begin{align}\label{prop-exp}
    \mathbb E[e^{cW(t)}]=e^{\frac {c^2 t}2},\; \forall \; c\in \mathbb C,\;
 \end{align}
  in the last step.
  
Considering  polar coordinates of \eqref{use-exp1}, we obtain 
{\small
\begin{align*}
  (2\pi)^{\frac d2} \log  {\Big|}\frac {\mathbb E \widehat u(t_2, \xi)}
{\mathbb E \widehat u(t_1, \xi) }{\Big|}  -\sum_{|\alpha|\; \text{even}} a_{\alpha}(\bi\xi)^{\alpha}(t_2-t_1)&=\frac 12 \sum_{|\beta|+|\widehat \beta|~\text{even}}^{p_2} \bi^{|\beta|+|\widehat \beta|}\xi^{\beta+\widehat \beta} b_{\beta} b_{\widehat \beta}(t_2-t_1)\;,\\
(2\pi)^{\frac d2} \operatorname{Arg} \frac {\mathbb E \widehat u(t_2, \xi)}
{\mathbb E \widehat u(t_1, \xi) } -\sum_{|\alpha|\; \text{odd}} a_{\alpha}(\bi\xi)^{\alpha}\bi^{-1}(t_2-t_1)&= \frac 12 \sum_{|\beta|+|\widehat \beta| ~\text{odd}}^{p_2} \bi^{|\beta|+|\widehat \beta|-1}\xi^{\beta+\widehat \beta}b_{\beta} b_{\widehat \beta} (t_2-t_1)\;.
\end{align*}}By choosing the Fourier modes $\xi_k\in \mathcal Q_1$ with $k=1,\cdots,\widetilde K\ge |\mathcal Q_1|,$ and under the assumption on $\mathcal Q_1$, there exists a unique solution for 
\begin{align*}
\by_{even}=\bA_{even}\bc_{even}, \;\by_{odd}=\bA_{odd}\bc_{odd}\;,
\end{align*}
\vspace{-1mm}
where 
\vspace{-2mm}
\begin{align*}
   (\by_{even})_k= \frac {2(2\pi)^{\frac d2}}{(t_2-t_1)}\log \frac {\mathbb E \widehat u(t_2, \xi_k)}
{\mathbb E \widehat u(t_1, \xi_k) }  -\sum_{|\alpha|\; \text{even}} a_{\alpha}(\bi\xi_k)^{\alpha}, k\le \widetilde K\;,\\
(\by_{odd})_k=\frac {2(2\pi)^{\frac d2}}{(t_2-t_1)} \log \frac {\mathbb E \widehat u(t_2, \xi_k)}
{\mathbb E \widehat u(t_1, \xi_k) }  -\sum_{|\alpha|\; \text{even}} a_{\alpha}(\bi\xi_k)^{\alpha}, k\le \widetilde K\;.
\end{align*}
\vspace{-1mm}
Here $(\bA_{even})_{k\gamma}=\xi^{\gamma}$ and
\vspace{-1mm}
\begin{align}\label{multiplicity-1}
(\bc_{even})_{\gamma}=\sum_{\beta+\widehat \beta =\gamma} \bi^{|\beta|+|\widehat \beta|}b_{\beta}b_{\widehat \beta}  
\end{align}
for even $|\gamma|\le 2p_2$,
and similarly for the odd case. 

Although the vectors $\bc_{even}$ and $\bc_{odd}$ are unique, the vector $b_{\beta}$ is uniquely determined up to an equivalent class of the solutions of \eqref{multiplicity-1} and the corresponding equation for odd indices.  
\end{proof}

Proposition~\ref{prop-drift-identification} and Theorem~\ref{diffusion-identification-covariances} justify the feasibility of the Stoch-IDENT identification framework based on~\eqref{eq_drift_identification_relation} and~\eqref{eq_diffuse_identification_relation}, and characterize the necessary conditions on the observed trajectories for the SPDE to be identifiable. We note that our analysis also provides the theoretical foundation for the data fidelity terms in the sparse regression formulations~\eqref{eq_L0_drift_select_greedy} and~\eqref{eq_L0_diffuse_select}. However, a general identifiability analysis incorporating the additional sparsity constraints would require further assumptions on the spectral properties of the feature matrices, which are difficult to validate in practice and are often not satisfied in the PDE and SPDE setting.

\subsection{Data space spanned by the solution trajectory of  linear SPDE}\label{sec-data-space}

We present interesting findings on the data spaces associated with two types of linear SPDEs: parabolic and hyperbolic equations. To facilitate the analysis, we impose additional conditions on the drift and diffusion operators.
{In this part, we denote $H:=L^2(\mathcal D)$.} It should be mentioned that results for the stochastic Schr\"odinger equation \cite{MR3826675} can also be derived, as it exhibits properties of both parabolic and hyperbolic equations.

\subsubsection{Parabolic SPDE}\label{parabolic}

Assume that $-\mathcal L:H\to H$ is a linear, densely defined, closed operator and is also admissible (see  Definition 2.1 of  \cite{he2024much}), i.e.,
$$\|(z+\mathcal L)^{-1}\|_{H\to H}\le \frac {C}{1+|z|},\; \text{for 
 all} \; z\in \mathbb C/\Sigma_{\delta}\;,$$
where the constant $C>0$ and the sector region
$\Sigma_{\delta}\subset \mathbb C$ is defined by 
$$\Sigma_{\delta}=\{z\in \mathbb C:|\operatorname{Arg}(z)|\le \delta\}, \; \delta\in (0,\frac \pi 2)\;.$$
By \cite[Theorem 1]{MR2231867},  there exists 
an operator $\mathcal A_L$ defined by $\mathcal A_L=\sum_{k=-L}^{L}c_ke^{-z_kt} $ $(z_k+\mathcal L)^{-1}$, $c_k,z_k\in \mathbb C,$ such that 
$
\|e^{\mathcal Lt}-\mathcal A_L(t)\|_{H\to H}=\mathcal O(e^{-cL})
$
uniformly in the time interval $[t_0,\Lambda t_0]$  with $0<c=\mathcal O(1/\log(\Lambda)).$ Here $t_0>0$, $L\in \mathbb N^+$, and $\Lambda>1$ { which is a fixed number independent of $\mathcal L.$}
In particular, for $t\in [t_0,T]$ that $t_0=\epsilon^{\kappa},\kappa>0,$ with  $ \epsilon \in(0,1)$ sufficiently small, one can take $L=C_{\mathcal L}(\kappa)|\log(\epsilon)|^2$ such that
\begin{align}\label{exp-decaying}
\|e^{\mathcal Lt}-\mathcal A_L(t)\|_{H\to H}\lesssim  \epsilon
\end{align}
for some constant $C_{\mathcal L}(\kappa)>0$ (see \cite[Corollary 2.3]{he2024much}).

Suppose that there exists $\mu>0$ such that $-\mathcal L_{\mu}=-\mathcal L+\mu$ is admissible. 
Let $u_0(x)=\sum_{k=1}^\infty c_k\phi_k(x)\in H$ with $c_k\in \mathbb R$ and $\phi_k$ being the orthonormal basis corresponding to the dominant operator $-\mathcal L_{\mu}$. Here $(\lambda_k,\phi_k(x))_{k\ge 1}$ are eigenpairs of $-\mathcal L_{\mu}$ sorted by the real part $\Re\lambda_k$ in an ascending order, including the multiplicity. By the Wely's law, the growing speed satisfies  $\Re\lambda_k=\mathcal O(k^{\frac {p_1}{d}}).$ 
For simplicity,  assume that $\mathcal G$ is a linear closed operator of order $p_2$, and commutes with $\mathcal L$. Denote the eigenvalues of $\mathcal G$ by $(q_k)_{k\ge 1}$, sorted by $\Re q_k$ in an ascending order, including the multiplicity.

Now we present the analysis for  
the data space spanned by the stochastic parabolic equation \eqref{linear-spde}. It can be seen that the solution of \eqref{linear-spde} satisfies \cite{MR1207136}:
\begin{align*}
 u(t, x)=e^{\mu t}\sum_{k=1}^{\infty} c_ke^{-\lambda_k t}\phi_k(x)+ \sum_{k=1}^{\infty} 
 \int_0^t e^{\mu (t-s)} e^{-\lambda_k (t-s)} q_k R_k \phi_k(x) \circ dW(s), \; a.s.
\end{align*} 
Here $R(x)=\sum_{k=1}^{\infty} R_k \phi_k(x)$ and $\mathcal G \phi_k =q_k \phi_k, q_k\in\mathbb C.$ 
In particular, in the multiplicative noise case, we have that 
\begin{align*}
 u(t, x)=e^{\mu t}\sum_{k=1}^{\infty} c_ke^{-\lambda_k t}e^{q_k W(t)}\phi_k(x),~{a.s.}
\end{align*}

\begin{theorem}\label{theorem-diffusion-identification}
Let $\epsilon\in (0,1).$
Suppose that $-\mathcal L+\mu$ is admissible {for some $\mu>0$,} and that $u_0$ is $\mathcal F_0$-measurable  satisfying $|c_k|_{L^2(\Omega;\mathbb R)}\le \theta k^{-\gamma}$ for some $\theta>0,\gamma>\frac 12.$
\begin{itemize}
    \item (Multiplicative noise) Assume that $R=u$ and    $\lim\limits_{k\to\infty} \Re(\lambda_k)-2\Re(q_k)^2>0.$
For any $t\in [0,T],$ there exists a linear space $V\subset H$ of dimension $C_{\mathcal L}|\log(\epsilon)|^2$ such that \begin{align}\label{add-over-err}
   \|u(t)-P_{V}u(t)\|_{L^2(\Omega;H)}\lesssim  \epsilon (1+\|u_0\|_{L^2(\Omega;H)})\;. 
   \end{align}
  
    \item (Additive noise) Assume that $R$ is a given function and  
$\sum_{k=1}^{\infty} \frac {|q_k|^2|R_k|^2}{\Re({\lambda_k})^{1-\theta_1}}<\infty$ for some $\theta_1\in (0,1].$ 
  For any $t\in [0,T]$, there exists a linear space $V\subset H$ of dimension $C_{\mathcal L}|\log(\epsilon)|^2$ such that \eqref{add-over-err} holds.
\end{itemize}
 Here $P_V$ is the projection operator onto $V$.
\end{theorem}    

\begin{proof}

We provide detailed derivations for \eqref{add-over-err}  in the case of multiplicative noise; the additive noise case is similar and thus omitted. For completeness, we put the proof in the additive noise case in Appendix \ref{sec-appendix-add-data}.

Let $u_{\widetilde M}(t, x)=e^{\mu t}\sum_{k=1}^{\widetilde M}c_ke^{-\lambda_k t}e^{q_kW(t)}\phi_k(x)$ denote the Galerkin approximation  with $\widetilde M\in \mathbb N^+.$
Then we have that 
\begin{align*}
&\|u_{\widetilde M}(t, \cdot)-u(t, \cdot)\|_{H}^2
=  e^{2\mu t} \sum_{k=\widetilde M+1}^\infty c_k^2 e^{-2\Re(\lambda_{k})t}e^{2\Re(q_{k})W(t)}.
\end{align*}
  Since $\lim\limits_{k\to\infty} [\Re(\lambda_k)-2\Re(q_k)^2]>0,$ there exists $\widetilde K\in \mathbb N^+$ such that $\mu- \Re(\lambda_{k})+ \Re(q_{k})^2<0$ for $k\ge \widetilde K$.  
  This, together with the property of conditional expectation and \eqref{prop-exp}, yield that for $\widetilde M\ge \widetilde K,$
 \begin{align}\label{spe-trun-error}
 \mathbb E \|u_{\widetilde M}(t, \cdot)-u(t, \cdot)\|_{H}^2
 &= \sum_{k=\widetilde M+1}^\infty \mathbb E [c_k^2] e^{2\mu t -2\Re(\lambda_{k})t+2\Re(q_{k})^2t}\le \theta^2 \frac {\widetilde M^{1-2\gamma}}{2\gamma-1}.
 \end{align}

Next, for any $\epsilon \in (0,1),$ let $M_{\epsilon}\ge \widetilde K,L_{\epsilon}>0$ to be determined later,  and define 
$$w_{\epsilon}
=\sum_{k=1}^{M_{\epsilon}}c_k\sum_{l=0}^{L_{\epsilon}} (-1)^l \frac {(\lambda_kt-\mu t)^l}{l!}e^{q_k W(t)}\phi_k(x).$$
Then for each $t$, $w_{\epsilon}$ sits in the linear space 
$$V_1=\text{span} \Big\{\sum_{k=1}^{M_{\epsilon}}c_k(-1)^l\frac {(\lambda_kt-\mu t)^l}{l!}e^{q_k W(t)}\phi_k(x): l=0,1\cdots,L_{\epsilon}\Big\}.$$
Now we are in a position to deal with the case that $\sup\limits_{k\le M_{\epsilon}}[|\lambda_k-\mu|+2\Re(q_k)^2]t\le 1.$ 
By Minkowski's inequality and the independent increment property of $W(\cdot)$, we have 
\begin{align*}
    \|u_{M_\epsilon}(t, \cdot)-w_{\epsilon}(t, \cdot)\|_{L^2(\Omega;H)}^2
    &\le \sum_{k=1}^{M_{\epsilon}}\mathbb E [|c_k|^2]\|e^{q_k W(t)}\|_{L^2(\Omega;\mathbb R)}^2
\Big|\sum_{l=L_{\epsilon}+1}^{\infty}(-1)^l\frac {(\lambda_kt-\mu t)^l}{l!}\Big|^2\\
    &\le \sum_{k=1}^{M_{\epsilon}}\mathbb E[|c_k|^2]\|e^{q_kW(t)}\|_{L^2(\Omega;\mathbb R)}^2\frac 1{[(L_{\epsilon}+1)!]^2}.
    \end{align*}
 According to \eqref{prop-exp}, for $t\in [0,\inf\limits_{k\le M_{\epsilon}}\frac 1{{|\lambda_{k}}-\mu|+2\Re(q_{k})^2}]$, we have
\begin{align*}
     \|u_{M_\epsilon}(t, \cdot)-w_{\epsilon}(t, \cdot)\|_{L^2(\Omega;H)}^2
     &\le \sum_{k=1}^{M_{\epsilon}}\mathbb E [|c_k|^2]e^{2\Re(q_k)^2t}\frac 1{[(L_{\epsilon}+1)!]^2}\le \theta^2 \frac {M_{\epsilon}^{1-2\gamma}}{2\gamma-1} e^{-2L_\epsilon+1}. 
\end{align*}
Letting $L_{\epsilon}=\frac 12|\log(\epsilon)|$ and $M_{\epsilon}=\widetilde K+\epsilon^{\frac 2{1-2\gamma}}$ with $\epsilon\in (0,1)$, and using \eqref{spe-trun-error}, we have  
  \begin{align*}
     \|u(t, \cdot)-w_{\epsilon}(t, \cdot)\|_{L^2(\Omega;H)}&\le \|u(t, \cdot)-u_{M_\epsilon}(t, \cdot)\|_{L^2(\Omega;H)}+ \|u_{M_\epsilon}(t, \cdot)-w_{\epsilon}(t, \cdot)\|_{L^2(\Omega;H)}\\
     &\lesssim  \epsilon (1+\|u_0\|_{L^2(\Omega;H)}). 
    \end{align*}

For  $t\in [t_0,T]$ with $t_0$ being $\inf\limits_{k\le M_{\epsilon}}\frac 1{|\lambda_{k}-\mu|+2\Re(q_{k})^2}$, by \eqref{exp-decaying}, there exists 
$\mathcal A_L$ approximating $e^{\frac {\mathcal L}2 t}$ such that $\|e^{-\frac {\mathcal L t}2}-\mathcal A_L\|_{H\to H}\lesssim \epsilon$, where $t\in [t_0,T].$ Here $t_0=\epsilon^{\kappa}$,  $L=C_{\frac {\mathcal {L}}2}(\kappa)|\log \epsilon|^2$ and 
 $\kappa=|\sup\limits_{k\le M_{\epsilon}} \log(|\lambda_{k}-\mu|+2\Re(q_{k})^2)|/|\log(\epsilon)|.$
Consequently, by~\eqref{prop-exp} and using the assumption that $\lim\limits_{k\to\infty}[\Re(\lambda_k)-2\Re(q_k)^2]>0$, we have  
\begin{align*}
   \|(e^{-\frac {\mathcal L} 2 t }- \mathcal A_L)e^{\frac {\mathcal L}2 t+\mathcal G W(t)}u_0\|^2_{L^2(\Omega;H)} 
   &\lesssim \epsilon^2 \mathbb E \|e^{\frac {\mathcal L}2 t+\mathcal G W(t)}u_0\|^2_{H}\\
   &\lesssim \epsilon^2
   \sum_{k} e^{-[\Re(\lambda_{k})-\mu]t}e^{2 \Re(q_{k})^2 t} \mathbb E [c_k^2] 
   \lesssim \epsilon^2\mathbb E [\|u_0\|_H^2].
\end{align*}
Therefore, for $t\in  [t_0,$ $T]$, there exists a linear space $V_2$ of dimension $L=C_{\frac {\mathcal L}2}(\kappa)|\log \epsilon|^2$
such that 
$$\|u(t, \cdot)-P_{V_2}u(t, \cdot)\|_{L^2(\Omega;H)}\lesssim \epsilon (1+\|u_0\|_{L^2(\Omega;H)}).$$
Taking $V$ as the linear space containing $V_1\cup V_2$, we complete the proof for  $t\in [0,T]$.
\end{proof}

\subsubsection{Hyperbolic  SPDE}\label{hyperbolic}
In this part, we study the behavior of solution trajectories for 
hyperbolic SPDEs, which is intrinsically different from the parabolic case.
Consider the following stochastic transport equation in the Stratonovich sense:
\begin{align}\label{hyper}
  &d u(t, x)+c(x)\cdot \nabla u(t,x) \circ dW(t)=0,\\\nonumber 
  &u(0, x)=u_0(x),
\end{align}
where $x\in \mathcal D$, $t\in [0,T],$ and $c(\cdot)$ is vector-valued.  Its equivalent It\^o's form is: 
\begin{align*}
 d u(t,x)+c(x)\cdot \nabla u(t,x) dB(t) 
  &=\frac 12  c(x)\cdot \nabla (c(x)\cdot \nabla u(t,x))dt.
\end{align*}
Thanks to  the corresponding particle formulation: 
\begin{align}\label{particle-sys}
 d X(t)=-c(X(t))\circ dW(t), X(0)=x,
\end{align}
the solution of \eqref{hyper} can be understood via the characteristic line method: 
$$u(X_{0,t}(x),t)=u_0(x).$$
Here $X_{0,t}(x)$ denotes the solution of \eqref{particle-sys} starting at time $0$ with initial state $x$ and ending at time $t$. 
If the characteristic lines do not intersect,  $u(y,t)=u_0(X_{t,0}(y)))$.

In this case, we can define the following two correlation functions in space and time in the average sense:
\begin{align*}
 K(x,y)=\int_0^T \mathbb E \left[u(s,x)u(s,y)\right]ds,\;G(s,t)=\int_{\mathcal D} \mathbb E \left[u(t,x)u(s,x)\right]dx,
\end{align*}
where $x,y\in \mathcal D, s,t\in [0,T].$ Note that $K(\cdot,\cdot)$ and $G(\cdot,\cdot)$ are the kernels of symmetric semi-positive compact integral
operators on $H$ and $L^2([0,T]),$ respectively.

\begin{proposition}\label{hyper-dimension}
 Let $c(\cdot)\in \mathcal C^{p+1}(\mathcal D;\mathbb R^d)$  ($p\in \mathbb N$) be a velocity field and $u_0\in \mathcal C^{p}(\mathcal D).$ Then there exists a subspace $V$ of dimension $o(\epsilon^{-\frac 2{p}})$ such that 
 \begin{align*}
    \|P_V u-u\|_{L^2(\Omega;L^2([0,T];H))}\lesssim \epsilon.
 \end{align*}
\end{proposition}
\begin{proof}

Notice that 
the Jacobian matrix $Y(t)=\frac {\partial X(t)}{\partial x}$  satisfies 
\begin{align*}
 Y(t)=I-  \int_0^t \frac {\partial}{\partial x} c(X(t)) Y(t) \circ d W(t),
\end{align*}
and that its inverse matrix $Z(t)$ satisfies 
\begin{align*}
 Z(t)=I+  \int_0^t  Z(u) \frac {\partial}{\partial x} c(X(t))  \circ d W(t).
\end{align*}
Thus, by the chain rule, $G(s,t)$ is differentiable with respect to $x$.
Since $c(\cdot)\in \mathcal C^{p+1}(\mathcal D;\mathbb R^d)$ and $\mathcal D$ is a compact set, by \cite[Corollary 4.6.7]{MR1070361} (see also \cite[Chapter 1, section 1.1]{MR4628103}), we have that 
$X_{s,t}(\cdot)\in \mathcal C^{p}(\mathcal D;\mathbb R^d)$ for any $s\le t\in [0,T].$ This also implies that $G(s,t)\in $  $\mathcal C^{p}([0,T]^2)$
, since $c^{(k)}(\cdot)$ for $0\leq k\leq p$ are bounded thanks to the compactness of $\mathcal D$.

Notice that the eigenvalues of $K$ and $G$ are non-negative satisfying $\lambda_1\ge \lambda_2\ge \cdots\ge \lambda_j\ge \cdots$ with $\lambda_j\to 0$ as $j\to \infty.$
 The differentiability of $G(s,t)$, i.e., $G(s,t)\in \mathcal C^{p}([0,T]^2)$ leads to 
$\lambda_j= o(j^{-(p+1)})$ (see, e.g., \cite{MR728689}). Define $V_{K}^k$ and $V_{G}^k$ ($k\in \mathbb N^+$) as the linear spaces spanned by the first $k$ leading eigenfunctions of $K(x,y)$ and $G(s,t)$, respectively. It holds that 
\begin{align*}
&\int_0^T \mathbb E \|u(t, \cdot)-P_{V_K^k}u(t, \cdot)\|_{L^2(\mathcal D)}^2dt=\int_{\mathcal D}\mathbb E \|u(\cdot,x)-P_{V_G^k}u(\cdot,x)\|_{L^2([0,T])}^2dx=\sum_{j=k+1}^{\infty}\lambda_j.
\end{align*}
This, together with $\lambda_j=o(j^{-(p+1)}),$ yields that 
\begin{align*}
\sum_{j=k+1}^{\infty}\lambda_j=o(k^{-p}),
\end{align*}
which completes the proof.
\end{proof}

Theorem~\ref{theorem-diffusion-identification} and Proposition~\ref{hyper-dimension} quantify the effective dimension of the space spanned by solution trajectories of parabolic versus hyperbolic SPDEs. Parabolic equations lead to solution trajectories lying in a low-dimensional, highly compressed subspace (dimension $\mathcal{O}(|\log(\epsilon)|^2)$ at accuracy $\epsilon$), while hyperbolic SPDEs span a richer subspace with polynomial-type dimension. See Section~\ref{sec_parabolic_vs_hyperbolic} for a numerical illustration. 
These properties are intrinsic to the SPDE solution trajectories and is independent of any particular identification algorithm. These results generalize the deterministic versions analyzed in~\cite{he2024much}.

\section{Proposed Algorithms and Implementation Details for Stoch-IDENT}\label{sec_algorithm}
We now describe the algorithmic components of Stoch-IDENT. Algorithm~\ref{alg:stoch-ident} presents a high-level overview of the full pipeline. As discussed in Section~\ref{sec_candidate_drift}, the candidate generation for drift terms follows the same approach as identifying constant-coefficient PDEs. Although~\eqref{eq_L0_drift_select_greedy} is NP-hard~\cite{nguyen2019np}, Subspace Pursuit (SP)~\cite{dai2009subspace} was found to be effective in finding the true supports~\cite{he2022robust,tang2023weakident}. {Moreover, under certain structural conditions on the feature matrix $\bF$, SP is shown to achieve exact sparse recovery and is robust when the data are perturbed~\cite[Theorems~1~and~9]{dai2009subspace}.} We adopt the same algorithm in this work for generating candidate drift models and refer the readers to the aforementioned works for details. In the following, we focus on the algorithmic components specific to SPDE identification.

\begin{algorithm}[t]
\caption{{Stoch-IDENT algorithm}}
\label{alg:stoch-ident}
\KwIn{Dataset $\{U_n\}_{n=1}^N$, drift dictionary $\mathcal{F}$ with $K$ features, diffusion dictionary $\mathcal{G}$ with $J$ features, significance threshold $p^*$}
\KwOut{Identified drift model $\widehat{\ba}$, diffusion model $\widehat{\bb}$}

\textbf{Step 1: Drift Identification}\\
\For{$k = 1, \dots, K$}{
    Solve~\eqref{eq_L0_drift_select_greedy} with sparsity level $k$ to obtain candidate drift model with coefficients $\widehat{\ba}^{(k)}$\;
}
Select the optimal candidate $\widehat{\ba}$ from $\{\widehat{\ba}^{(k)}\}_{k=1}^K$ which minimizes~\eqref{eq_drift_score}\;
Compute drift residuals $\mathcal{R}^n := \{\rho_i^n(\widehat{\ba}),\, i = 1,\dots,I\}$ for $n = 1,\dots,N$\;

\textbf{Step 2: Pure Additive Noise Detection}\\
Test each $\mathcal{R}^n$ for Gaussianity and aggregate $p$-values via Stouffer's method\;
\If{combined $p$-value $> p^*$}{
    Estimate $\widehat{\sigma}$ via sample standard deviation of $\mathcal{R}^n$\;
    Set $\widehat{\bb} = \widehat{\sigma}$ and \Return\;
}

\textbf{Step 3: General Diffusion Identification}\\
Compute diffuse response $\{\zeta_i(\widehat{\ba})\}_{i=1}^I$ from~\eqref{eq_diffuse_response}\;
\For{$j = 1, \dots, J$}{
    Solve~\eqref{eq_L0_diffuse_select} to obtain candidate diffusion model $\widehat{\bb}^{(j)}$\;
}
Select the optimal candidate $\widehat{\bb}$ from $\{\widehat{\bb}^{(j)}\}_{j=1}^J$ which minimizes~\eqref{eq_S_diffuse}\;
\end{algorithm}

\subsection{Problem reformulation by sample average approximation}\label{sec_mean_approximation}

As the mean values in the feature systems introduced in Section~\ref{sec_feature_system} are not available, we approximate them by their respective sample averages:\vspace{-1mm}
\begin{equation}\label{eq_sample_mean_drift}
\bbE[F_k(t,x)]\approx \frac{1}{N}\sum_{n=1}^N F_k^n(t,x)\;,~\text{and}~\bbE[u(t,x)] \approx \frac{1}{N}\sum_{n=1}^NU_n(t,x)\;,
\end{equation}
for $k=1,\dots, K$ and any $(t,x)\in [0,T]\times\cD$, where $F_k^n(t,x):=F_k(t,x,\omega_n)$  for $\omega_n\in\Omega$. Define $\bF^N$ and $\by^N$ as the $N$ sample mean approximations of  $\bF$~\eqref{eq_drift_feature_matrix} and $\by$~\eqref{eq_drift_feature_response} respectively by replacing their entries with the corresponding estimators~\eqref{eq_sample_mean_drift}. The expectations in the diffusion feature matrix~\eqref{eq_diffuse_matrix} and the diffusion responses~\eqref{eq_diffuse_response} are also approximated by sample averages. Denoting $G_j^n(t,x):={G_j}(t,x,\omega_n)$, we use
\begin{equation*}
\bbE\left[\mathlarger{\fint}_{\cD}G_s(t_i,y)\cdot G_j(t_i,y)\,dy\right]\approx \frac{1}{N}\sum_{n=1}^N \mathlarger{\fint}_{\cD}G^n_s(t_i,y)\cdot G^n_j(t_i,y)\,dy\;,
\end{equation*}
for any $s,j\in\{1,\dots,J\}$ in~\eqref{eq_diffuse_matrix} and denote  $\bG_{i}^N$ as the resulting matrix for $i=1,\dots,I$.  Given   an estimated drift coefficient vector $\widehat{\ba}=(\widehat{a}_1,\dots,\widehat{a}_K)$, we use
\begin{equation*}
\zeta_i\approx\zeta_i^N(\widehat{\ba}):=\frac{1}{N}\sum_{n=1}^N \mathlarger{\fint}_{\cD}\left(r_i^n(y, \widehat{\ba})\right)^2\,dy\;,
\end{equation*}
where  $r^n_i(x,\widehat{\ba}) := U_n(t_{i},x)-U_n(t_{i-1},x)-\Delta t \sum_{k=1}^K \widehat{a}_k F^n_k(t_{i-1},x)$, and $\Delta t>0$ is the interval size of the grid in the time dimension.  For~\eqref{eq_L0_drift_select_greedy},  we obtain 
a surrogate:
\vspace{-1mm}
\begin{equation}\label{eq_L0_drift_select_sample}
\begin{aligned}
\min_{\bc\in\mathbb{R}^K}&\|\bF^N\bc - \by^N\|_2^2\\
\text{s.t.}&~\|\bc\|_0=k
\end{aligned}\;,
\end{equation}
for $k=1,\dots,K$. Suppose $\widehat{\ba}^N$ is an estimated diffusion coefficient vector, then we consider the following surrogate of~\eqref{eq_L0_diffuse_select} 
\vspace{-1mm}
\begin{equation}\label{eq_L0_diffuse_select_sample}
\begin{aligned}
\min_{\bc\in\mathbb{R}^J}&\sum_{i=1}^I\left(\bc^\top\bG^N_i\bc - \zeta_i^N(\widehat{\ba}^N)\right)^2\\
\text{s.t.}&~\|\bc\|_0=j
\end{aligned}\;,
\end{equation}
for $j=1,\dots,J$. As the feature systems are approximated via sample means, the identified drift and diffusion coefficients are random variables. 

The following result shows  that as the number of sample paths $N\to\infty$, any convergent sequence of  minimizers of the sample average approximation~\eqref{eq_L0_drift_select_sample} converges to a minimizer of~\eqref{eq_L0_drift_select_greedy} almost surely.

\begin{theorem}\label{eq_convergence_theorem}Denote $f_O(\bc):=\|\bF\bc-\by\|_2^2$ and $f_N(\bc):=\|\bF^N\bc-\by^N\|_2^2$. Assume that $\bc^*\in\mathbb{R}^K$ with $\|\bc^*\|_0=k$ is a  local minimizer of~\eqref{eq_L0_drift_select_greedy}, i.e., there exists some $\varepsilon>0$ such that $f_O(\bc^*)\leq f_O(\bc)$ for any $\bc$ with $\|\bc-\bc^*\|_2<\varepsilon$. Then the problem~\eqref{eq_L0_drift_select_sample} converges to~\eqref{eq_L0_drift_select_greedy} almost surely in the following sense:
\begin{enumerate}
\item  The optimal value of~\eqref{eq_L0_drift_select_greedy} converges almost surely to that of~\eqref{eq_L0_drift_select_sample}.
\item Let $\Psi_N:=\arg\min_{\bc\in\mathbb{R}^K}\{f_N(\bc): \|\bc\|_0=k\}$ and $\Psi_O:=\arg\min_{\bc\in\mathbb{R}^K}\{f_O(\bc): \|\bc\|_0=k\}$, then
\vspace{-1mm}
\begin{equation}\label{eq_Psi_def}
\limsup_{N\to\infty} \Psi_N \subset \Psi_O~a.s.
\end{equation}
\end{enumerate}
where $\limsup$ is understood in the Kuratowski-Mosco sense~\cite{salinetti1981convergence}:
\vspace{-2mm}
\begin{equation*}
\limsup_{N\to\infty} X_N:=\{x\in\mathbb{R}^K:\exists (x_{N_k})_{k=1}^\infty~\text{with}~\lim_{k\to\infty}x_{N_k}=x~\text{and}~x_{N_k}\in X_{N_k}\;,\forall k\in\mathbb{N}\}\;.
\end{equation*}
\end{theorem}
\begin{proof} This   follows from a series of results in~\cite{vogel1994stochastic}. From the strong law of large numbers, $f_N(\bc)$ converges to $f_O(\bc)$ almost surely for any $\bc\in\mathbb{R}^K$. Since $\bc^*$  is a local minimizer of $f_O$, by Theorem 5.1 of~\cite{vogel1994stochastic}, $f_N$ is lower semi-continuously convergent almost surely  to $f_O$ on the set $C_k := \{\bc:\|\bc\|=k\}$ (See Definition 2.6~\cite{vogel1994stochastic}). Hence, the conditions in Theorem 4.1 of~\cite{vogel1994stochastic} are satisfied, which implies the conclusions.
\end{proof}
Following~\cite{vogel1994stochastic}, the analogous results with convergence in probability also hold. Based on the almost surely convergent subsequence of minimizers specified in~\eqref{eq_Psi_def}, the asymptotic behavior of the surrogate~\eqref{eq_L0_diffuse_select_sample} can be similarly deduced.

{\begin{remark}\label{remark_asymp_condition}
 The condition that there exists a $k$-sparse solution to~\eqref{eq_L0_drift_select_greedy} which is also a local minimizer can be guaranteed by sufficient conditions involving Spark~\cite{donoho2003optimally} or Restricted  Isometry Property (RIP)~\cite{candes2005decoding} of   $\bF$; see~\cite{beck2013sparsity} also. It should be noted that Theorem \ref{eq_convergence_theorem} is about the asymptotic convergence of sample‑average approximations of the objective. 
While  covariance estimation in general requires  sample size scaled quadratically with the dimension, our setting  exploits sparsity which significantly reduces the required sample size. This is consistent with the sparse recovery literature~\cite{candes2005decoding,donoho2003optimally}.
\end{remark}
}

\subsection{Statistical tests for detecting pure additive noise}\label{sec_stat_detect}

We note that if the underlying SPDE has only additive noise, the identification of the diffusion model reduces to a single parameter estimation problem once the drift model is selected, and no further model selection is needed. It would thus be computationally beneficial to detect whether a SPDE model with additive noise alone indeed fits the observed dynamics. For this purpose, we devise a statistical testing strategy.

Let $\widehat{\ba}=(\widehat{a}_1,\dots,\widehat{a}_K)$ be the estimated drift coefficient vector obtained from Section~\ref{sec_candidate_drift}. We define the space averaged residual error for the $n$-th trajectory
\vspace{-2mm}
\begin{equation*}
 \rho^n_i(\widehat{\ba}) := \mathlarger{\fint}_{\cD}\left(U_n(t_{i},x)-U_n(t_{i-1},x)-\Delta t \sum_{k=1}^K \widehat{a}_k F^n_k(t_{i-1},x)\right)\,dx\;,
\end{equation*}
for $i=1,\dots,I$ and $n=1,\dots,N$. For an SPDE with additive noise $\sigma dW$, if $\widehat{\ba}$ is accurate, the residual data $\mathcal{R}^n:=\{\rho_i^n(\widehat{\ba}),\, i=1,\dots,I\}$ should be approximately distributed as $\mathcal{N}(0, \sigma^2\Delta t)$ for $n=1,\dots,N$. We test each $\mathcal{R}^n$ for Gaussianity using the D'Agostino--Pearson test~\cite{d1973tests}, and the resulting $N$ $p$-values are then aggregated via Stouffer's method~\cite{stouffer1949american} to control the false positive rate. If the combined $p$-value is smaller than some threshold $p^*>0$, we deem the noise to be multiplicative and proceed with the diffusion identification in Section~\ref{sec:QSP}. Otherwise, we treat it as additive, estimate $\sigma$ via the sample standard deviation $\widehat{\sigma}$, and identify the model as $du = \sum_{k=1}^K\widehat{a}_k\cF_k(u)\,dt + \widehat{\sigma}\,dW(t)$. This bypasses the subsequent diffusion identification procedure.

\subsection{New Quadratic Subspace Pursuit (QSP) for candidate diffusion models}\label{sec:QSP}
If the noise is not purely additive, we proceed to generate candidate diffuse terms by considering~\eqref{eq_L0_diffuse_select}. For this, we propose a new greedy algorithm, Quadratic Subspace Pursuit (QSP) described in Section~\ref{sec:qsp-algorithm} and show a conditional stability of support recovery in Section~\ref{sec:qsp-convergence}
\subsubsection{QSP algorithm}\label{sec:qsp-algorithm}
\begin{algorithm}[t]
	\KwIn{Diffuse feature system $\bG_i\in\mathbb{R}^{J\times J}$~\eqref{eq_diffuse_matrix}, $\zeta_i\in\mathbb{R}$~\eqref{eq_diffuse_response}, for $i=1,\dots,I$, and some integer $j\in\{1,\dots,J\}$.}
	
	\textbf{Initialization:} $\ell=0$;\\
     Compute $q_s^{(0)} = \min_{c\in\mathbb{R}} \sum_{i=1}^I(\bG^{s,s}_i c-\xi_i)^2$ for $s=1,\dots,J$\\
    Set $\cI^0 = \{$Indices corresponding to the $j$ smallest  $|q_s^{(0)}|\}$;

Compute $\widehat{\mathbf{c}}^{(0)} \in\underset{\substack{\mathbf{c}\in\mathbb{R}^{J}\\
        [\mathbf{c}]_{(\mathcal{I}^{0})^\complement}=\mathbf{0}}}{\argmin} \sum_{i=1}^I (\bc^\top\bG_i\bc - \zeta_i)^2$;

    Compute $\eta_i^{(0)}= [\widehat{\mathbf{c}}^{(0)}]^{\top}_{\mathcal{I}^{0}}[\bG_i]_{\mathcal{I}^{0}} [\widehat{\mathbf{c}}^{(0)}]_{\mathcal{I}^{0}} - \zeta_i$ for $i=1,2,\dots,I$;

	\While{True}{
		\textbf{Step 1.} 
        Compute $q_s^{(\ell+1)} = \min_c \sum_{i=1}^I(\bG^{s,s}_i c-\eta_i^{(\ell)})^2$ for  $s\not\in \cI^\ell$\\
Set $\widetilde{\mathcal{I}}^{\ell+1}=\mathcal{I}^{\ell}\cup\{$Indices corresponding to the $j$ smallest  $|q_s^{(\ell+1)}|\}$;
		
		\textbf{Step 2.} Compute $\overline{\mathbf{c}}^{(\ell+1)} \in\underset{\substack{\mathbf{c}\in\mathbb{R}^{J}\\
        [\mathbf{c}]_{(\widetilde{\mathcal{I}}^{\ell+1})^\complement}=\mathbf{0}}}{\argmin} \sum_{i=1}^I (\bc^\top\bG_i\bc - \zeta_i)^2$ \\
        
        Set $\mathcal{I}^{\ell+1}=\{$Indices of the entries of $\overline{\mathbf{c}}^{(\ell+1)}$ with the $j$ largest absolute values$\}$\;
		
		\textbf{Step 3.} Compute $\widehat{\mathbf{c}}^{(\ell+1)} \in\underset{\substack{\mathbf{c}\in\mathbb{R}^{J}\\
        [\mathbf{c}]_{(\mathcal{I}^{\ell+1})^\complement}=\mathbf{0}}}{\argmin} \sum_{i=1}^I (\bc^\top\bG_i\bc - \zeta_i)^2$ 
	\\
		
		\textbf{Step 4.} Compute $\eta_i^{(\ell+1)}= [\widehat{\mathbf{c}}^{(\ell+1)}]^{\top}_{\mathcal{I}^{\ell+1}}[\bG_i]_{\mathcal{I}^{\ell+1}} [\widehat{\mathbf{c}}^{(\ell+1)}]_{\mathcal{I}^{\ell+1}} - \zeta_i$ for $i=1,2,\dots,I$;
        
        \If{$\sum_{i=1}^I(\eta_i^{(\ell+1)})^2>\sum_{i=1}^I(\eta^{(\ell)}_i)^2$} {Set $\mathcal{I}^{*} = \mathcal{I}^\ell$ and $\widehat{\bc}^{*} = \widehat{\bc}^{(\ell)}$, then break;}
        
        \Else{Set $\ell\gets \ell+1$ and continue;
        }

	}
	
	\KwOut{ Indices of chosen features $\mathcal{I}^*$ with $|\mathcal{I}^*|=j$ and  reconstructed coefficients $\widehat{\bc}^*\in\mathbb{R}^{J}$ with $\text{supp}(\widehat{\bc}^*) = \cI^*$.
	}
	\caption{{\bf Proposed Quadratic Subspace Pursuit (QSP)}  \label{QSPalgo}}	
\end{algorithm}

 Algorithm~\ref{QSPalgo} shows the pseudo-code of QSP. The proposed QSP  searches for a $j$-sparse vector mainly by iterating two operations stated in Algorithm~\ref{QSPalgo} (\textbf{Step 1}) \textit{Expanding}: include unselected variables that have the strongest potential in reducing the regression residuals; and (\textbf{Step 2}) \textit{Shrinking}: discard selected variables  with  regression coefficients of small magnitudes. 

Specifically, we define $\eta_i^{(0)} = \zeta_i$ for $i=1,\dots,I$ and set $\cI^0$ as described in Algorithm~\ref{QSPalgo}. In the $\ell$-th iteration of \textbf{Step 1}, we use the strategy of  coordinate descent  and examine the squared sum of the individual regression errors:
\begin{equation*}
q_s^{(\ell+1)} = \min_c \sum_{i=1}^I(\bG^{s,s}_i c-\eta_i^{(\ell)})^2\;, s=1,\dots, J\;.
\end{equation*}
Here $\bG_i^{s,s}$ is the $s$-th diagonal element of the $i$-th diffusion feature matrix;  $\eta_i^{(\ell)}$ is the quadratic regression error associated with the $i$-th measurement (\textbf{Step 4}) using the current candidate variables; and $s$ is any index of variables not selected in the previous iteration.  Then we include the variables with the smallest $j$ errors, yielding a set of candidate indices $\widetilde{\cI}^{\ell+1}$ with size  at most $2j$.

In the $\ell$-th iteration of \textbf{Step 2}, we solve the nonlinear regression problem \vspace{-2mm}
\begin{equation}\label{eq_regression_quadratic}
\overline{\mathbf{c}}^{(\ell+1)} \in\underset{\substack{\mathbf{c}\in\mathbb{R}^{J}\\
        [\mathbf{c}]_{(\widetilde{\mathcal{I}}^{\ell+1})^\complement}=\mathbf{0}}}{\argmin} \sum_{i=1}^I (\bc^\top\bG_i\bc - \zeta_i(\widehat{\ba}))^2
\end{equation}
only using the selected variables indexed by $\widetilde{\cI}^{\ell+1}$. For~\eqref{eq_regression_quadratic}, we tested with nonlinear conjugate gradient (CG) descent using various CG updating parameters~\cite{hager2006survey}, and find that the one proposed by~\cite{hager2005new} performs the best in our experiments.  Then we keep the indices of the entries of $\overline{\mathbf{c}}^{(\ell+1)} $ with the $j$ largest absolute values to be the  set $\cI^{\ell+1}$. In \textbf{Step 3}, we compute the regression coefficients using the updated set of variables. The algorithm terminates when the regression error does not improve.

\begin{remark}
While QSP addresses a sparse regression problem with quadratic measurements and SP focuses on linear ones, QSP shares the same algorithmic structure as SP~\cite{dai2009subspace}: both iterate between expansion and shrinkage. We also note that STLS and STRidge, as considered in SINDy~\cite{brunton2016discovering,rudy2017data}, employ certain thresholding strategies; it would therefore be interesting to explore the quadratic counterparts of these algorithms for sparse quadratic regression in future work.
\end{remark}
\subsubsection{Stability of support recovery of QSP}\label{sec:qsp-convergence}
Although a full algorithmic analysis of the proposed QSP algorithm, e.g., convergence,  is beyond the scope of this paper, we present here sufficient conditions so  that QSP always include the true support during the iteration.

Let $\starbc \in \R^J$ be the true $j$-sparse diffusion coefficient vector
with support $S^* = \supp(\starbc)$, $|S^*| = j$.
The diffusion feature matrices $\bG_i \in \R^{J \times J}$,
$i = 1, \ldots, I$, are symmetric positive semi-definite. Suppose
the drift term identification is sufficiently accurate so that the diffuse responses satisfy
\begin{equation}\label{eq:noise_model}
    \zeta_i \;=\; (\starbc)^\top \bG_i \starbc \;+\; \epsilon_i,
    \qquad |\epsilon_i| \leq \epsilon \text{ uniformly.}
\end{equation}

We state  assumptions needed to establish the stability of support recovery property of QSP. 
\begin{assumption}[Bounded feature matrices]
\label{ass:bounded}
There exists a constant $M > 0$ such that
\[
    \max_{i = 1,\ldots,I} \norm{\bG_i}_2 \;\leq\; M.
\]
\end{assumption}

\begin{definition}[Cross-feature coherence]
\label{def:coherence}
We define the \emph{cross-feature coherence} of the diffusion
feature matrices $\{\bG_i\}_{i=1}^I$  as:
\[
    \mu_{\bG}
    \;:=\;
    \max_{s \neq t}\;
    \frac{1}{I}
    \sum_{i=1}^{I}
    \frac{\abs{\bG_i^{s,t}}^2}
         {\bG_i^{s,s}\cdot \bG_i^{t,t}}.
\]
\end{definition}

\begin{assumption}[Signal Strength]
\label{ass:signal}
With $\mu_{\bG}$ as in Definition~\ref{def:coherence}
and 
\begin{equation}
    D_s := \frac{1}{I}\sum_{i=1}^{I}(\bG_i^{s,s})^2,
\end{equation}
the true coefficient vector $\starbc$ and the diffusion
feature matrices $\{\bG_i\}$ satisfy:
\[
    \min_{s \in S^*}
    \frac{[\starbc]_s^2\,\sqrt{D_s}}{\norm{\starbc}^2}
    \;>\;
    2\,|S^*|\,M\,\sqrt{\mu_{\bG}},
\]
where $[\starbc]_s$ denotes the $s$-th entry of $\starbc$,  $|S^*|$ is the true number of active diffusion
features.
\end{assumption}

\begin{theorem}[Stability of support recovery of QSP]
\label{thm:main}
Under Assumptions~\ref{ass:bounded} and~\ref{ass:signal}, suppose $j \geq |S^*|$. There exists some $\epsilon_2^*$ such that whenever the perturbation bound  $\epsilon$ in~\eqref{eq:noise_model} satisfies $\epsilon<\epsilon_2^*$,  we have:
\begin{enumerate}
    \item $S^* \subseteq \mathcal{I}^0$.
    \item For any $\ell \geq 0$, if $S^* \subseteq \mathcal{I}^\ell$ and:
        \begin{equation}\label{eq:approx_cond}
    \frac{\norm{\overline{\mathbf{c}}^{(\ell+1)}
          - \starbc}^2}{\norm{\starbc}^2}
    \;<\;
    \frac{1}{4}\cdot\frac{\min_{s \in S^*}\sqrt{D_s}}
              {\,\max_{s \in S^*}\sqrt{D_s}}
    \cdot
    \min_{s \in S^*}
    \frac{[\starbc]_s^2}{\norm{\starbc}^2},
\end{equation}
    then $S^* \subseteq \mathcal{I}^{\ell+1}$.
\end{enumerate}
\end{theorem}
See Appendix~\ref{sec:proof-main} for the proof and the expression of $\epsilon_2^*$.  Theorem~\ref{thm:main} guarantees that if the perturbation $\{\epsilon_i\}_i$ coming from the drift residuals or sample approximation variation is sufficiently small, QSP with $j\geq |S^*|$ is guaranteed to include the true support in the initialization. In particular, when $|S^*|=j$, this will exactly recover. Moreover, when the relative coefficient recovery of $\overline{\bc}^{(\ell+1)}$ with more non-zero entries than $\bc^*$ is bounded as~\eqref{eq:approx_cond}, the true support will remain included if it was already from the previous iteration. Combining these, we can conclude that QSP has conditional stability of support recovery. 

We highlight that the condition \eqref{eq:approx_cond} has a natural
interpretation. Define:
\begin{itemize}
    \item The \emph{relative approximation error}:
    $\displaystyle\delta^{(\ell+1)} :=
    \frac{\norm{\overline{\mathbf{c}}^{(\ell+1)} -
    \starbc}}{\norm{\starbc}}$,
    \item The \emph{minimum relative signal amplitude}:
$\displaystyle\alpha
    \;:=\;
    \min_{s \in S^*}
    \frac{|[\starbc]_s|}{\norm{\starbc}}
    \;\in\; (0, 1],
$    \item The \emph{feature energy condition number}:
    $\displaystyle\kappa_D := \frac{\max_{s \in S^*}\sqrt{D_s}}
    {\min_{s \in S^*}\sqrt{D_s}} \geq 1$.
\end{itemize}
Then \eqref{eq:approx_cond} reads:
\[
   \delta^{(\ell+1)}
    \;<\;
  \frac{\alpha}{2\,\sqrt{\kappa_D}}.
\]
 This threshold is
easy to satisfy when:
\begin{enumerate}
    \item[(i)] The true coefficients are well-balanced
    ($\alpha$ close to $\frac{1}{|S^*|}$), and
    \item[(ii)] The feature energies $D_s$ are
    homogeneous across true features ($\kappa_D$
    close to $1$).
\end{enumerate}
Conversely, support recovery becomes harder when one
true feature is much weaker than the others (small
$\alpha$) or when the feature energies are highly
heterogeneous (such as  $\kappa_D$ is large).

\begin{remark}
Empirically, we observe that QSP terminates in about 5--10 iterations. However, its convergence analysis is non-trivial as we need to characterize the conditions under which  $\sum_{i=1}^I(\eta^{(\ell)}_i)^2$ (Line 12 of Algorithm~\ref{QSPalgo}) is non-increasing in $\ell$; this requires careful control of the score $q_s^{(\ell+1)}$ (Line 7 of Algorithm~\ref{QSPalgo}) for $s\notin S^*$; see~\cite{dai2009subspace,he2025group} for examples. This, however, is highly nontrivial as it involves $\eta_i^{(\ell)}$, which admits a closed form expression only for $\ell=0$. 
Exploration of its applicability to more general settings, as well as further analysis of its convergence and stability, are left to a different work. 
\end{remark}

\subsection{Normalization and trimming}\label{sec_trimming}
The identification of the drift terms is analogous to the identification of PDEs, and the candidate generation~\eqref{eq_L0_drift_select_greedy} is addressed using the SP algorithm~\cite{dai2009subspace}. To avoid the effects of scaling during feature selection, we normalize the columns of $\bF$ so that each column has unit norm. As proposed in~\cite{tang2023weakident}, the trimming technique is effective at removing redundant features that fail to be removed during the greedy search. Specifically, for a candidate drift coefficient vector $\widehat{\ba}=(\widehat{a}_1,\dots,\widehat{a}_K)$, we compute the relative contribution of the $k$-th term as
\begin{equation}
\rho_k(\widehat{\ba}):=\frac{|\widehat{a}_k|}{\max_{s=1,\dots,K}|\widehat{a}_s|}
\end{equation}
for $k=1,\dots,K$, and for a specified threshold parameter $\tau_{d}>0$, we set the $k$-th coefficient to zero if $\rho_k(\widehat{\ba})<\tau_d$ and re-estimate the other non-zero coefficients via least square fitting.

For the identification of the diffusion terms, we propose analogous normalization and trimming techniques adapted to the quadratic structure of~\eqref{eq_L0_diffuse_select}.
In particular, we define $\overline{\bG}_i:= \bLambda^{-1}\bG_i\bLambda^{-1}$, 
for $i=1,\dots,I$, where $\bLambda\in\mathbb{R}^{J\times J}$ is a diagonal matrix whose $j$-th diagonal element is $\sqrt{J^{-1}\cdot\sum_{i=1}^I\bG_i^{j,j}}$, for $j=1,\dots,J$. When running QSP (Algorithm~\ref{QSPalgo}), we substitute $\bG_i$ with its normalized version $\overline{\bG}_i$, and the resulting coefficient estimate $\widehat{\bb}$ is transformed back to the original basis via $\widehat{\bb}\bLambda^{-1}$. For trimming, given a candidate diffusion coefficient vector $\widehat{\bb}=(\widehat{b}_1,\dots,\widehat{b}_J)$, we define $\theta_j(\widehat{\bb}) :=|\widehat{b}_j|/\max_{s=1,\dots,J}|\widehat{b}_s|$ for $j=1,\dots,J$, and set the $j$-th coefficient to zero if $\theta_j(\widehat{\bb}) < \tau_f$ for some threshold $\tau_f>0$. The remaining non-zero coefficients are then re-estimated via the nonlinear regression~\eqref{eq_regression_quadratic}.

\section{Numerical Experiments}\label{sec_numerical}

This section presents numerical experiments to validate Stoch-IDENT on various SPDEs
{involving genuinely nontrivial features (mixed additive/multiplicative noise, nonlinear diffusion structure).} We simulate observed trajectories by solving the SPDEs numerically using the Euler-Maruyama scheme for time discretization and appropriate methods for spatial variables. All examples are Cauchy problems with periodic boundary conditions. 

In this work, we consider \textbf{dictionaries of type $\mathbf{(p,q)}$} for integers $p\geq0$ and $q\geq 1$, which means that we include  features with spatial derivatives up to order $p$ and multiplications of up to $q$ terms. These dictionary  parameters  can be different for the dictionaries for the drift and diffusion parts.   We estimate the spatial differential features using the classical 7-point finite difference scheme~\cite{fornberg1988generation} with periodic boundary conditions.  For identifying the diffusion part, the maximal  number of iterations of the nonlinear CG for address~\eqref{eq_regression_quadratic} is set to be $1000$; the initial guesses for the non-zero  entries are fixed at $10$; and  the iteration terminates if  the gradient of the loss function has a magnitude smaller than $1\times 10^{-14}$.  The maximal number of iterations for the QSP (Algorithm~\ref{QSPalgo}) is set to $100$, although in practice, we observe that it converges around $5\sim 10$ iterations. For both drift and diffusion identification, we apply a trimming threshold (Section~\ref{sec_trimming}) of $\tau_d=\tau_f=0.3$. For the pure additive noise detection (line 8 in Algorithm~\ref{alg:stoch-ident}), the $p$-value threshold is fixed as $p^*=0.02$.

\begin{table}
\centering
\caption{Evaluation for identification. Here $\text{TP}=|\operatorname{supp}(\bc^*)\cap \operatorname{supp}(\widehat{\bc})|$, $\text{TN}=|\operatorname{supp}(\bc^*)^\complement\cap \operatorname{supp}(\widehat{\bc})^\complement|$, $\text{FP}=|\operatorname{supp}(\bc^*)^\complement\cap \operatorname{supp}(\widehat{\bc})|$, and $\text{FN}=|\operatorname{supp}(\bc^*)\cap \operatorname{supp}(\widehat{\bc})^\complement|$ between  an estimated drift (diffusion) coefficient vector $\widehat{\bc}$ and the true  drift (diffusion) vector $\bc^*$, respectively.  }\label{tab_metrics}
\begin{tabular}{c|c|c|c}
\toprule\toprule
Precision (Prec)& Accuracy (Acc)& Recall& F1-score\\
\midrule
$\frac{\text{TP}}{\text{TP}+ \text{FP}}$&
$\frac{\text{TP}+\text{TN}}{\text{TP}+ \text{FP}+\text{TN}+ \text{FN}}$&
$\frac{\text{TP}}{\text{TP}+ \text{FN}}$&
$2\cdot \frac{\text{Prec}\times \text{Recall}}{\text{Prec}+ \text{Recall}}$\\
   \bottomrule
\end{tabular}
\end{table}

For performance evaluation for both drift and diffusion parts, we  compare the support of the estimated coefficient vector $\bc=(c_1,\dots, c_K)$ with the support of the ground truth coefficient vector $\bc^*=(c_1^*,\dots, c^*_K)$ using metrics in Table~\ref{tab_metrics}.
These metrics are all bounded between $0$ and $1$, with $1$ being the best. In addition, we evaluate the coefficient errors using the following metrics:
\begin{itemize}
\item  Relative in-coefficient error
\begin{equation*}
E_{\text{in}}(\bc, \bc^*)=\frac{\sqrt{\sum_{i\in\operatorname{supp}(\bc^*)}(c_i-c^*_i)^2}}{\|\bc^*\|_2}\times 100\%\;.
\end{equation*}
\item Relative out-coefficient error
\begin{equation*}
E_{\text{out}}(\bc, \bc^*)=\frac{\sqrt{\sum_{i\not\in\operatorname{supp}(\bc^*)}c_i^2}}{\|\bc\|_2}\times 100\%\;.
\end{equation*}
\end{itemize}
The relative in-coefficient error measures the deviation of the estimated coefficients for the true features, while the relative out-coefficient error measures the magnitude of the coefficients for the wrongly identified features.

\subsection{General performances of Stoch-IDENT}\label{sec_general}
\def\fw{0.25}
\begin{figure}
\centering
	\begin{tabular}{ccc}
		(a)&(b)&(c)\\
		\includegraphics[width=\fw\textwidth]{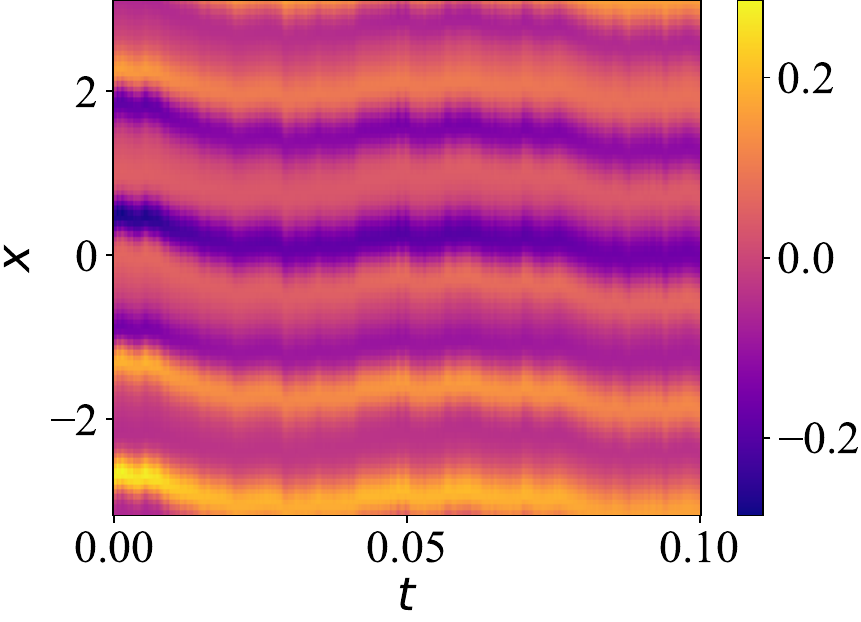}&
		\includegraphics[width=\fw\textwidth]{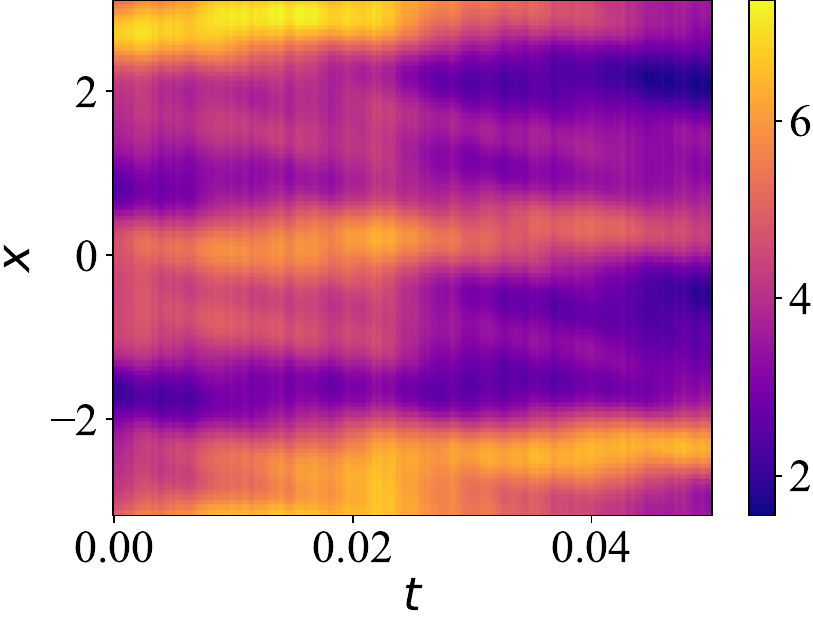}&
		\includegraphics[width=\fw\textwidth]{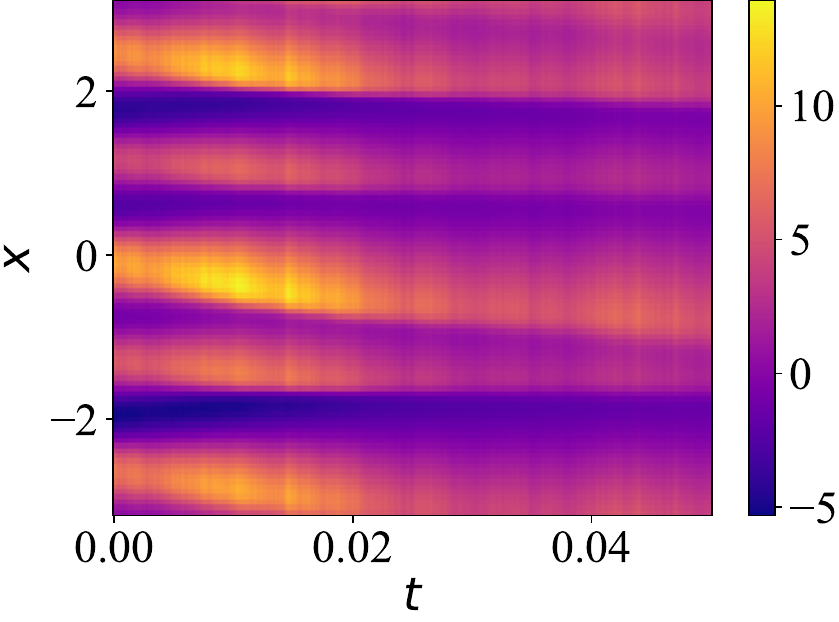}\\\hline
		\multicolumn{3}{c}{Precision}\\\hline
		\includegraphics[width=\fw\textwidth]{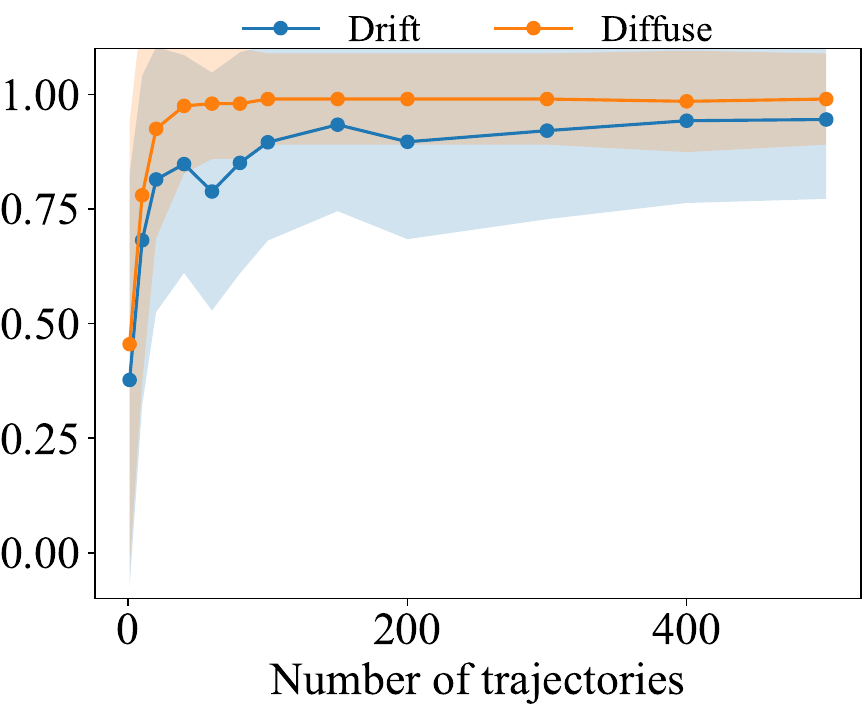}&
		\includegraphics[width=\fw\textwidth]{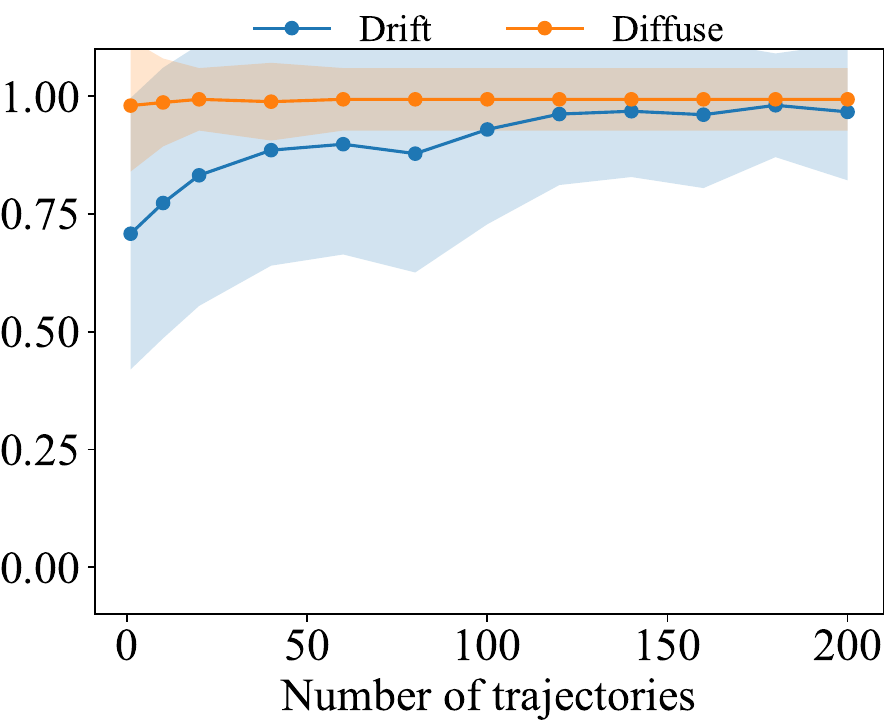}&
		\includegraphics[width=\fw\textwidth]{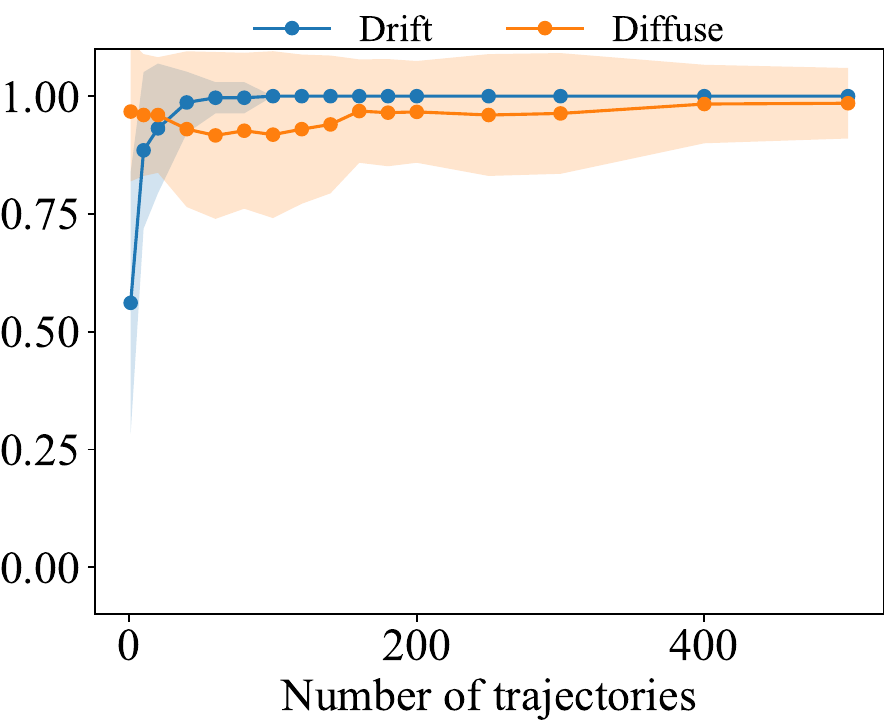}\\\hline
		\multicolumn{3}{c}{Recall}\\\hline
		\includegraphics[width=\fw\textwidth]{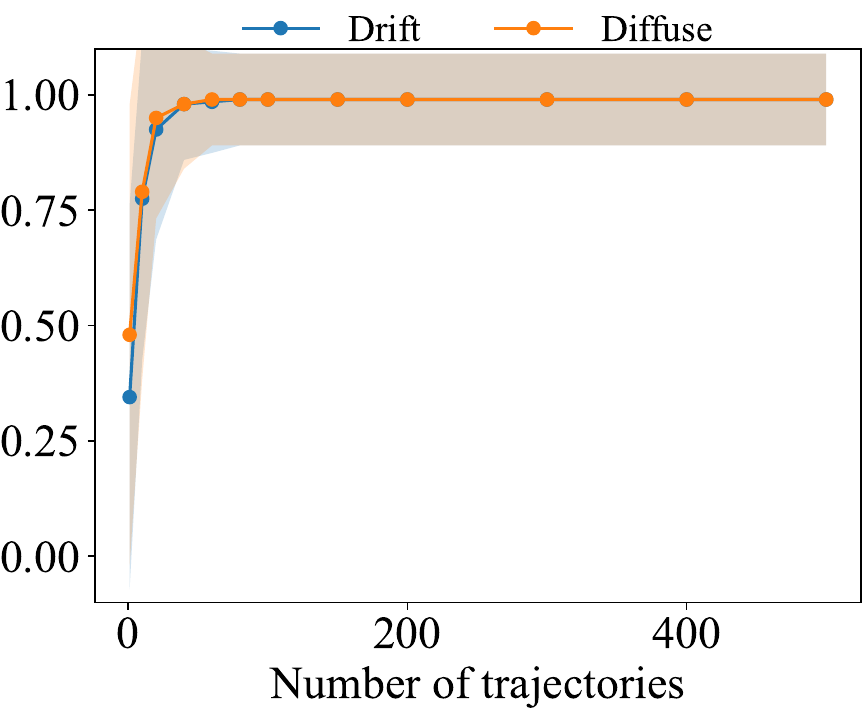}&
		\includegraphics[width=\fw\textwidth]{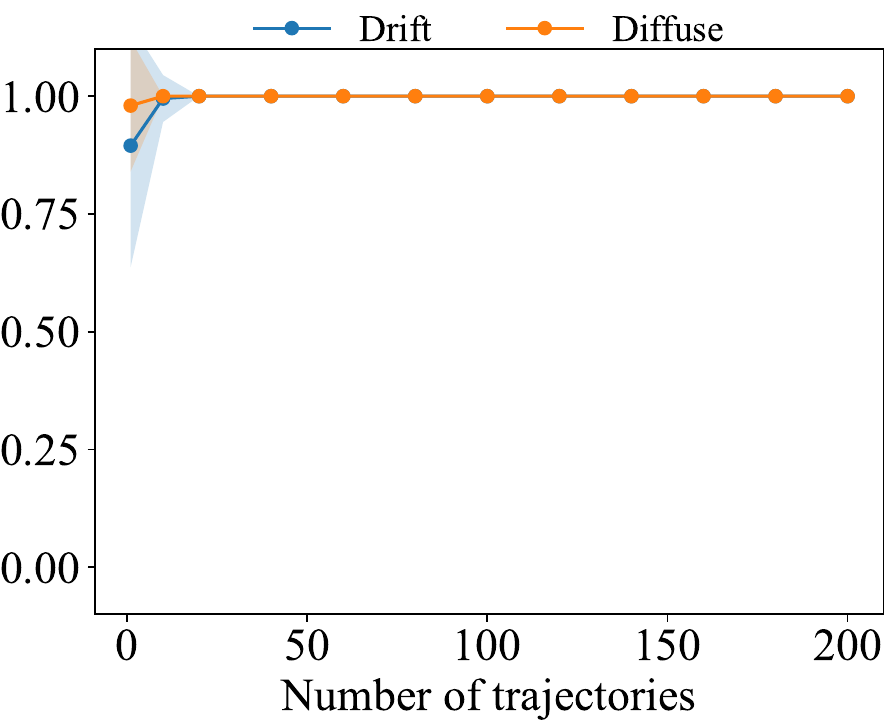}&
		\includegraphics[width=\fw\textwidth]{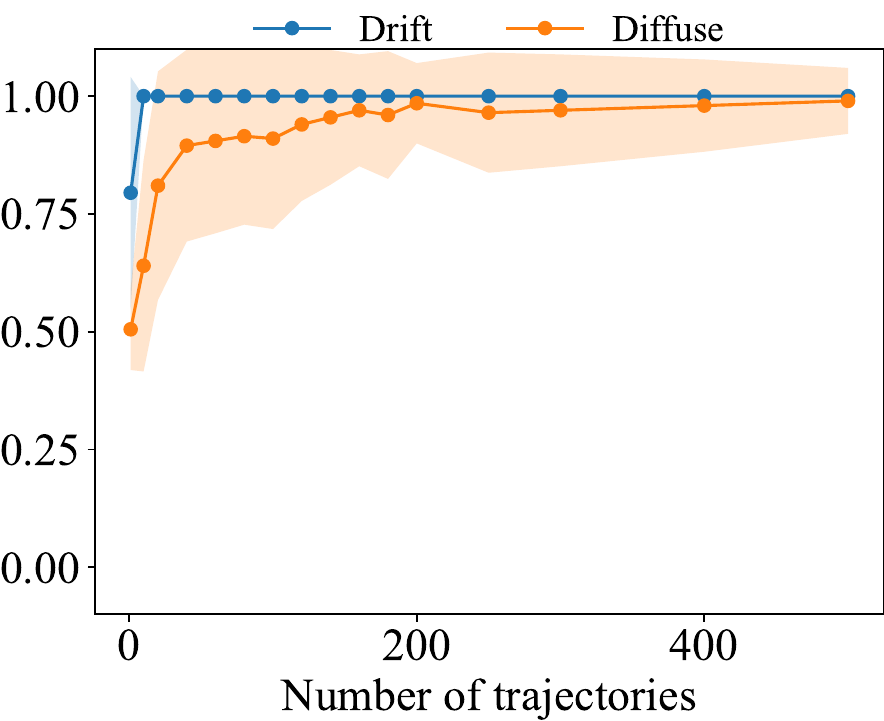}\\\hline
		\multicolumn{3}{c}{Coef-In}\\\hline
		\includegraphics[width=\fw\textwidth]{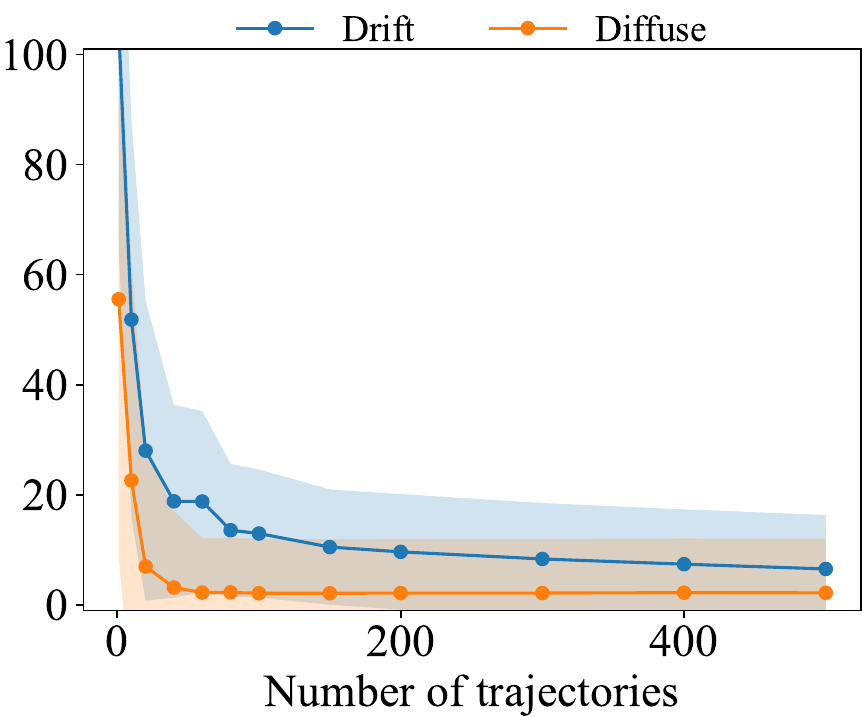}&
		\includegraphics[width=\fw\textwidth]{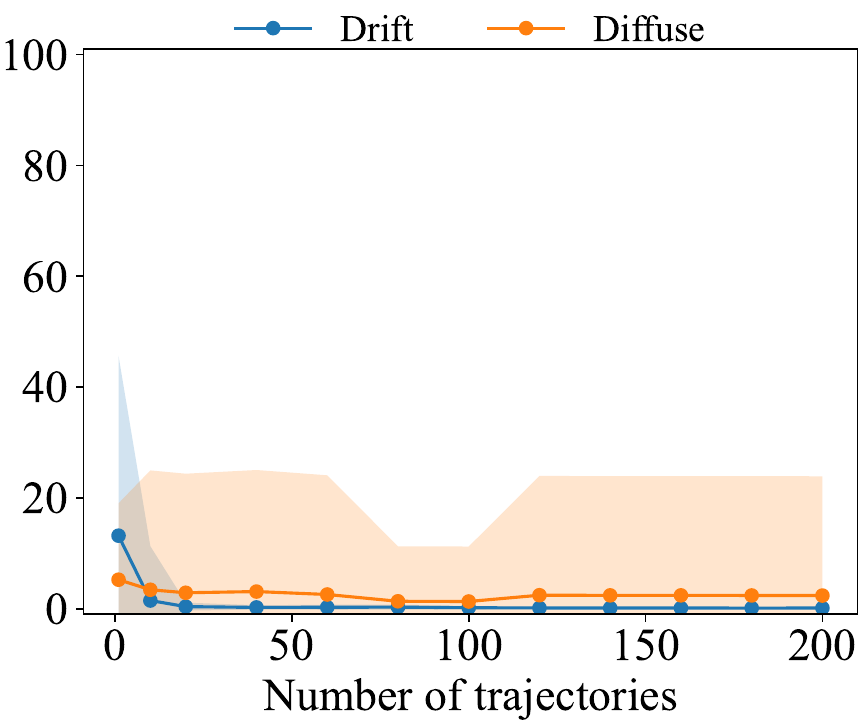}&
		\includegraphics[width=\fw\textwidth]{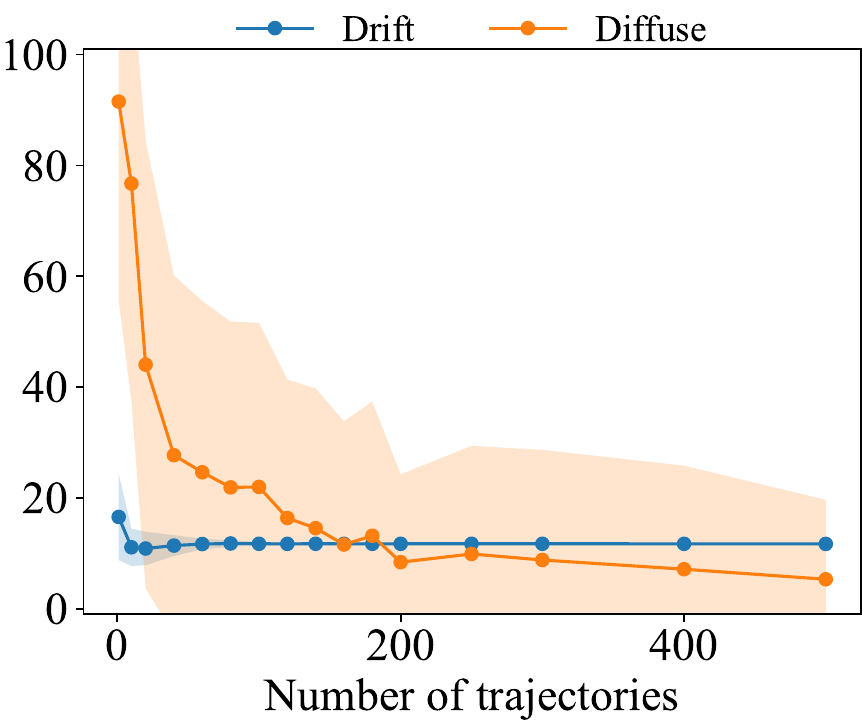}\\\hline
		\multicolumn{3}{c}{Coef-Out}\\\hline
		\includegraphics[width=\fw\textwidth]{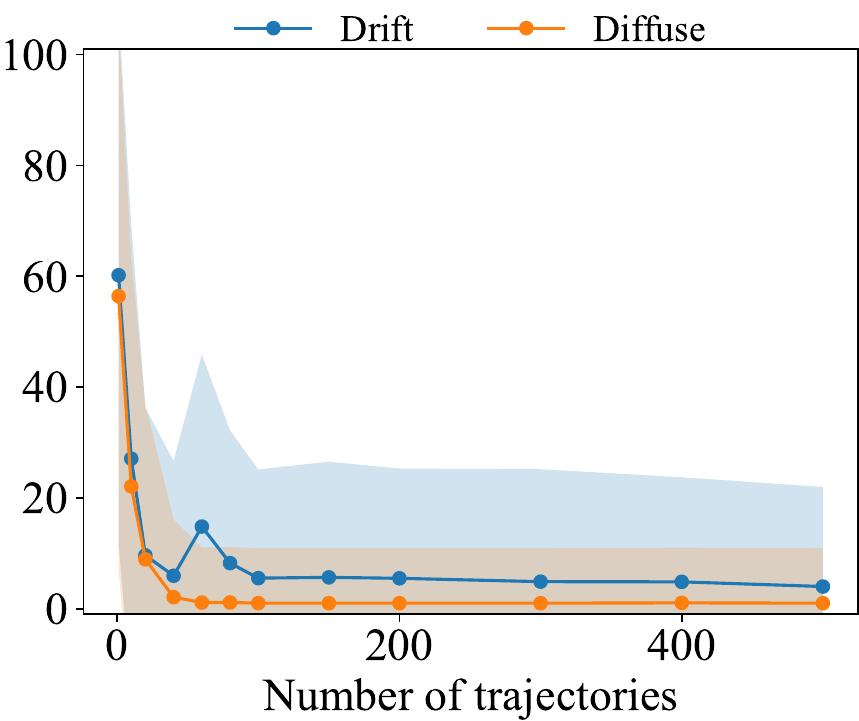}&
		\includegraphics[width=\fw\textwidth]{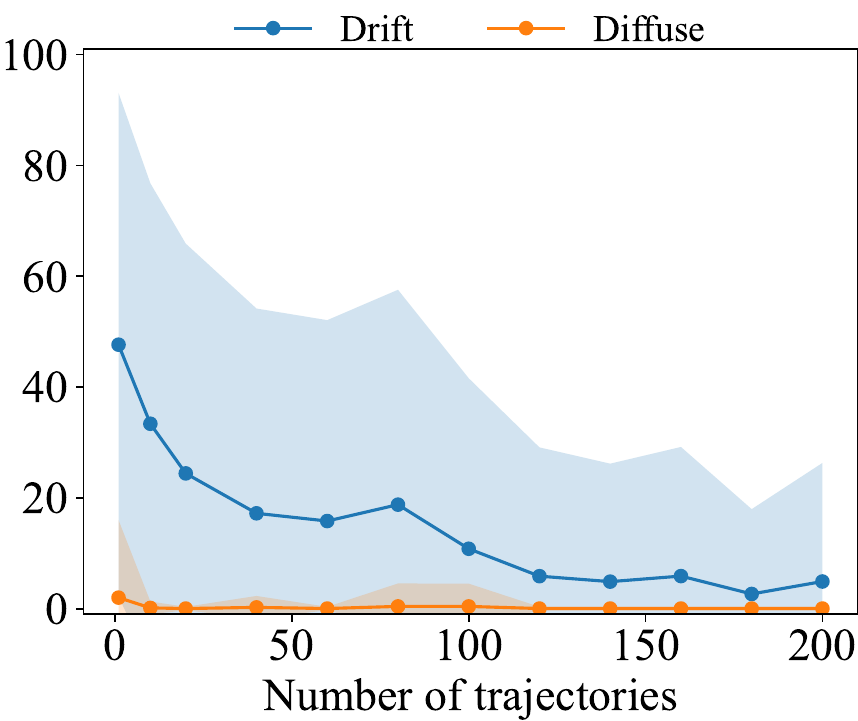}&
		\includegraphics[width=\fw\textwidth]{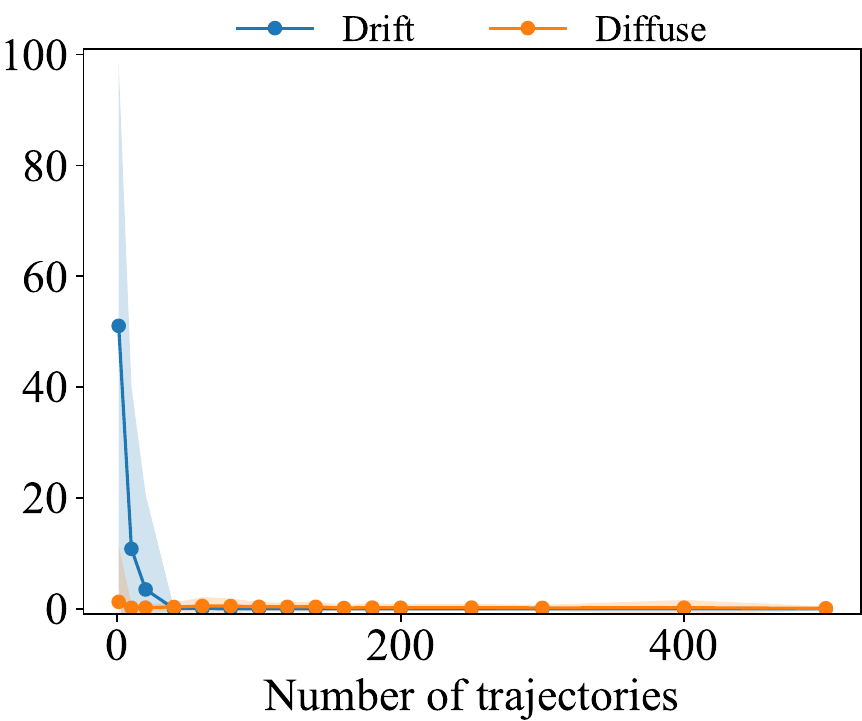}
	\end{tabular}
	\caption{Identification performance versus the number of trajectories for (a) the stochastic transport equation~\eqref{eq_example_transport}, (b) the stochastic KdV equation~\eqref{eq_example_kdv}, and (c) the stochastic Burgers equation~\eqref{eq_example_burgers}. Performance is evaluated using precision, recall, and relative in-sample and out-of-sample coefficient errors. For each choice of number of trajectories, results are averaged over 100 independent experiments with shaded regions indicating one standard deviation.}\label{fig_general_performance}
\end{figure}

We test Stoch-IDENT on $N$ independent trajectories generated from the following SPDEs.\\

\noindent\textbf{Stochastic transport equation}  with multiplicative noise~\cite{MR2593276}: 
\begin{equation}\label{eq_example_transport}
 d u= (3u_x + 0.5u_{xx})dt + u_xdW(t)\;,\quad\text{for}~x\in [-\pi, \pi)~\text{and}~t\in(0,0.1)
\end{equation}
with the initial condition $ u(0,x) = 0.1\exp(\sin(4x-0.2))\cdot\cos(5x+0.8)$.\\

\noindent\textbf{Stochastic Korteweg-de Vries (KdV) equation}  with additive noise~\cite{MR1711301}: 
\begin{equation}\label{eq_example_kdv}
du = (-6uu_x - u_{xxx})dt + 7dW(t)\;,\quad\text{for}~x\in [-\pi, \pi)~\text{and}~t\in(0,0.05)
\end{equation}
with the initial condition $ u(0,x) = \exp(\sin(3x-0.2))\cdot \cos(2x+0.8) + 4.0$.\\

\noindent\textbf{Stochastic Burgers equation}  with both additive and multiplicative noise~\cite{MR3033029}: 
\begin{equation}\label{eq_example_burgers}
du = (3uu_x + 0.5u_{xx})dt + (5 + 2u)dW(t)\;,\quad\text{for}~x\in [-\pi, \pi)~\text{and}~t\in(0,0.05)
\end{equation}
with the initial condition $u(0,x) = 3\sin^2(x-1)+ 2\cos(2x)  + 5\sin(5x+2.0)+ 1.0$.\\

For all the SPDEs above,  we employ periodic boundary conditions, and  the solution datasets are collected on a uniform grid with $300$ points in time and $100$ points in space. The grids are up-sampled in time by $50$ times when generating the data through numerical evolution. We use dictionaries of type $(4,3)$ for the drift part, and type $(2,2)$ for the diffusion part. {In accordance with Proposition~\ref{prop-drift-identification} and Theorem~\ref{diffusion-identification-covariances}, unique identification requires the initial data to activate sufficiently many nontrivial Fourier modes. The initial conditions in the experiments here and afterwards are therefore chosen to be smooth but spectrally rich functions: combinations of exponentials, sines, and cosines with incommensurate frequencies}. For each choice of sample size $N$, we conduct $100$ independent experiments and examine the statistics of the evaluation metrics. 

In Figure~\ref{fig_general_performance}, we present the identification results as follows: the first row displays sample trajectories; the second row shows the relationship between $N$ and precision (see Table~\ref{tab_metrics}); the third row reports recall (see Table~\ref{tab_metrics}); the fourth row presents the relative in-coefficient error; and the last row illustrates the relative out-coefficient error. We make the following observations: (1) As $N$ increases, identification accuracy measured by precision and recall improves and coefficient error decreases across all SPDEs.
(2)  Identifying the diffusion term is most challenging with a mixture of additive and multiplicative noise. Performance is best with purely additive noise (e.g., KdV equation), worse with purely multiplicative noise (e.g., transport equation), and worst with mixed noise (e.g., Burgers equation). Notably, as $N$ grows, the identified diffusion part for the Burgers example often contains only additive noise. This happens because the drift identification method measures the residual's lack of fit using a squared $\ell_2$ norm, which implicitly treats the residual as a homoscedastic normal vector. If the underlying multiplicative noise is strong, it can cause the residual to deviate from normality, facilitating correct diffusion identification; however, stronger noise also requires more sample paths to accurately approximate the covariance structure.

\subsection{Identification of Stochastic nonlinear Schr\"odinger equation}

\begin{table}
\centering
\caption{The most frequently identified model from $100$ independent experiments of identifying the  stochastic nonlinear Schr\"{o}dinger equation~\eqref{eq_NLS} using varying numbers of trajectories ($N$). For each identified feature, we also report the mean value of the associated coefficients $\pm$ the standard deviation.  }\label{tab_NLS} 
{\small
\begin{tabular}{c l}
\toprule\toprule
\multicolumn{1}{l}{$N$}& \multicolumn{1}{c}{Most frequently identified model}
\\
\midrule
1& $\displaystyle\begin{cases}
du &= \left(9.760_{\pm 0.559} u +4.435_{\pm 0.061} u_{xx}+ 0.888_{\pm 0.024} u^3\right)dt+ 2.832_{\pm 0.027}dW(t)\\
dv &= \left(-1.012_{\pm 0.041} u^2v -4.772_{\pm 0.082} v_{xx} -0.950_{\pm 0.040} v^3\right)dt+0.987_{\pm 0.003} udW(t)
\end{cases}$\\
\midrule
10& $\displaystyle\begin{cases}
du &= \left(4.942_{\pm0.092} u_{xx} +0.994_{\pm 0.068} u^3+ 1.002_{\pm 0.101} uv^2\right)dt +1.000_{\pm 0.014} vdW(t) \\
dv &= \left(-1.004_{\pm 0.103} u^2v -4.960_{\pm 0.010} v_{xx} -1.009_{\pm 0.084} v^3\right)dt +1.001_{\pm 0.018} u dW(t)
\end{cases}$\\
\midrule
40& $\displaystyle\begin{cases}
du &= \left(4.935_{\pm 0.049} u_{xx} + 0.989_{\pm 0.037} u^3 + 0.995_{\pm 0.054} uv^2\right)dt + 0.998_{\pm 0.010} v dW(t)\\
dv&=\left(-1.010_{\pm 0.059} u^2v -4.950_{\pm 0.048} v_{xx} -0.997_{\pm 0.040} v^3\right)dt +1.004_{\pm 0.009} u dW(t)
\end{cases}$\\
\bottomrule\toprule
& \multicolumn{1}{c}{Reference model}\\
\midrule
& $\displaystyle\begin{cases}
du &= \left(5u_{xx} + u^3 + uv^2\right)dt \pm v dW(t)\\
dv&=\left(-u^2v -5v_{xx} -v^3\right)dt \pm  u dW(t)
\end{cases}$\\
\bottomrule
\end{tabular}
}
\end{table}

To demonstrate the versatility of the proposed Stoch-IDENT, we test it with the  \textbf{stochastic nonlinear Schr\"{o}dinger (NLS)} with multiplicative noise (see e.g. \cite{MR3670034,MR3826675}):
\begin{equation}\label{eq_NLS}
d \rho= 5\bi  \rho_{xx} dt+ \bi |\rho|^2\rho dt+\bi  \rho dW(t)\;,\quad\text{for}~x\in [-\pi, \pi)~\text{and}~t\in(0,0.2)
\end{equation}
where $\rho = u+\bi v$ is a complex function, and  $W(t)$ is a real-valued Wiener process. The initial condition is $u(0,x) = \exp(\sin(2x+1))+1$ and $v(0,x) = \exp(\cos(3x+1))+1$. The setup for the grid and dictionaries is identical to the previous experiments. 
For  $N=1, 10$ and $40$ sample paths, we conduct $100$ independent experiments. 

In Table~\ref{tab_NLS}, we report the most frequently identified models when $N=1,10$ and $40$. For the coefficients, we show the sample means and standard deviations of the estimated values. We observe that \textbf{(1)}  Although the correct features can be frequently identified, a single path ($N=1$) is insufficient to yield the correct model.  \textbf{(2)}  In this NLS example, a few more sample paths ($N=10$) are enough to frequently identify the correct model, and the variability of the estimated coefficients is reduced when more paths are available.   \textbf{(3)} In the original model~\eqref{eq_NLS}, the imaginary diffusion part is symbolically $-udW(t)$, which is equivalent to $udW(t)$  in the sense discussed in section~\ref{newsec-linear-spde-identification}. Since we use positive values for the nonzero entries of the initial guesses for the nonlinear CG iterations, the estimations  tend to converge to the  positive values. From~\eqref{eq_regression_quadratic}, this sign difference is  indistinguishable; thus, the coefficient errors are measured in absolute values in this case.

\subsection{Parabolic versus hyperbolic identification}\label{sec_parabolic_vs_hyperbolic}

\begin{figure}
\centering
\includegraphics[width=0.9\textwidth]{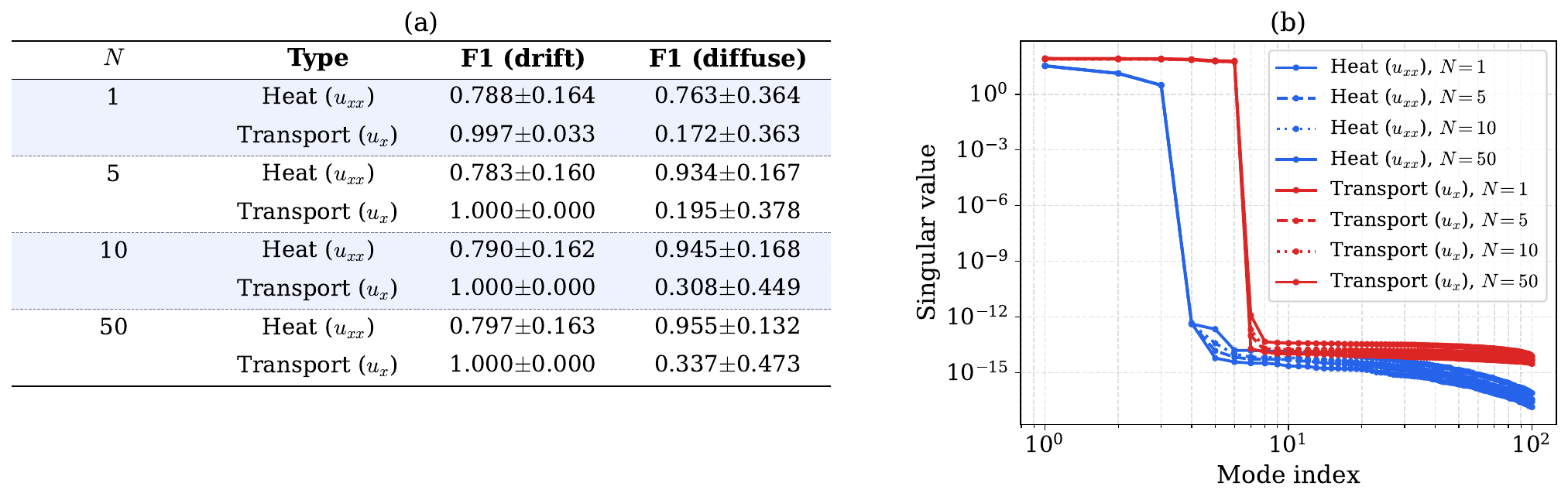}
\caption{Numerical verification of the identifiability theory developed in Section~\ref{sec_identifiability}. According to Theorem~\ref{theorem-diffusion-identification} and Proposition~\ref{hyper-dimension}, the solution dimension of the stochastic heat equation is generally lower than that of the stochastic transport equation. Using~\eqref{eq:ident-transport} and~\eqref{eq:ident-heat} as test models, (a) shows the F1 scores of the drift and diffuse identification performance as the number of trajectories increases, and (b) shows the spectrum of the averaged trajectory. For the stochastic transport equation, drift recovery is much easier than for the stochastic heat equation, which is further supported by the richer spectral content of the transport solution compared to the heat solution. On the other hand, diffuse identification is considerably harder for the transport equation than for the heat equation, due to the fact that the transport equation accumulates noise during evolution.}\label{fig_identifiability}
\end{figure}

To illustrate the theoretical findings of Section~\ref{sec-data-space}, we conduct a controlled experiment comparing the identification of a parabolic and a hyperbolic SPDE with identical diffusion structures. Specifically, we consider comparing a stochastic transport equation:
\begin{equation}\label{eq:ident-transport}
du = 5u_{xx}\,dt + u\,dW(t), 0<t<1.0, -\pi<x\leq \pi
\end{equation}
with a stochastic heat equation:
\begin{equation}\label{eq:ident-heat}
du = 5u_x\,dt + u\,dW(t), 0<t<1.0, -\pi<x\leq \pi,
\end{equation}
with periodic boundary condition in space. Both equations share the same diffusion term and are solved on the same spatial grid with the same initial condition $\cos(x- 0.8) + \cos(3x+ 0.8)- \cos(5x + 1)$. Since our theoretical analysis in Section~\ref{sec_identifiability} focuses on linear SPDEs for accessibility, here and only here we use type $(5,1)$ dictionaries for both drift and diffuse terms. We apply Stoch-IDENT to $N$ trajectories of~\eqref{eq:ident-transport} and~\eqref{eq:ident-heat} for $N=1,5,10,$ and $50$, repeating each experiment $100$ times, and report the F1 score (see Table~\ref{tab_metrics}) for both the drift and diffuse parts.

Figure~\ref{fig_identifiability} (a) shows the comparison results. We observe that as the number of trajectories $N$ increases, the identification accuracy for both drift and diffuse parts of both equations improve. For the stochastic transport equation, we see that the drift identification is better than the stochastic heat equation: it always recovers exactly the drift part after $N\geq 5$, yet the F1 score of  the heat equation achieves only around $0.8$ when $N=50$. This is reflected by the spectrum of the mean trajectory data shown in (b), where the spectral content of the transposition case is much more than the heat case, confirming our claims made in  Propositions~\ref{theorem-diffusion-identification} and~\ref{hyper-dimension}. 

Interestingly, contrary to the case for drift identification, for diffuse identification,  For the stochastic heat equation, the diffusion operator smooths and dissipates energy over time. The corresponding stochastic convolution $\int_0^t e^{\mathcal{L}(t-s)} \mathcal{G} \, dW(s)$ decays exponentially in each Fourier mode (except zero mode), meaning the contribution of noise injected at earlier times is strongly suppressed by time $t$. As a result, the stochastic part of the solution remains well-controlled and does not grow. For the stochastic transport equation, however, there is no such dissipation. The noise accumulates along characteristics without decay. Concretely, the  stochastic integral in the transport case grows like $\sqrt{t}$ in $L^2$, so the signal-to-noise ratio for the diffuse identification degrades over time.

Moreover, the diffuse identification relies on computing the diffusion feature matrix $\mathbf{G}_i$ in~\eqref{eq_diffuse_matrix}. As implied in~\cite{cosse2024sparse}, when the off-diagonal entries of $\mathbf{G}_i$ are small relative to the diagonal, different features are more distinguishable and sparse recovery becomes more tractable. In the transport case, the accumulated noise enters $u$ itself, so each feature $G_j(t_{i-1}, x)$ is contaminated by the full history of the Wiener process up to time $t_i$. This causes inflated off-diagonal correlations in $\mathbf{G}_i$, making it prone to ill-conditioning. Furthermore, the diffuse response $\zeta_i(\widehat{\ba})$ as defined in~\eqref{eq_diffuse_response} suffers from stronger temporal correlation for~\eqref{eq:ident-transport}, which can inflate estimation variance. These claims and solutions deserve more careful investigation and we leave them to future work.

\subsection{Comparison study}

\begin{table}
\centering
\caption{We compare Stoch-IDENT with e-SINDy and VB~\cite{mathpati2024discovering} using  identical drift and diffuse feature dictionaries on stochastic heat equations with multiplicative noise: (I) $0.3u_xdW(t)$ and (II) $(2u+0.5u_x)dW(t)$. For Stoch-IDENT, the default parameters are used; for e-SINDy and VB, the parameters are selected via grid search (see the main text for details). Stoch-IDENT achieves robust identification and accurate coefficient estimation. Importantly, while e-SINDy and VB identify the squared diffusion term, Stoch-IDENT recovers the original diffusion term directly (up to equivalence).}\label{tab_compare}
{\small
\begin{tabular}{c|cc|cc|cc|cc}
\toprule
\toprule
\multicolumn{9}{c}{(I) Equation~\eqref{eq_stoch_heat_mult}}\\\midrule
\multirow{2}{*}{Method}& \multicolumn{2}{c}{Prec ($\uparrow$)}& \multicolumn{2}{c}{Recall ($\uparrow$)}& \multicolumn{2}{c}{$E_{\text{in}}$ ($\downarrow$)}&\multicolumn{2}{c}{$E_{\text{out}}$ ($\downarrow$)}
\\
&Drift&Diff.&Drift&Diff.&Drift&Diff.&Drift&Diff.\\
\midrule
Proposed&0.9358&0.9883&1.0000&1.0000&1.92&2.79&5.17&0.49\\
e-SINDy &0.6217&1.0000&1.0000&1.0000&1.13&3.34&60.68&0.00\\
VB      &0.9350&1.0000&1.0000&1.0000&1.07&2.33&12.68&0.00\\
\midrule
\multicolumn{9}{c}{Most frequently identified model}\\\midrule
Proposed&\multicolumn{8}{l}{$du = 0.982_{\pm0.007}u_{xx}\,dt+ 0.298_{\pm0.004} u_{x}\,dW(t)$}\\\midrule
e-SINDy&\multicolumn{8}{l}{$du = \Big(0.990_{\pm0.008}u_{xx}-1.334_{\pm2.714}u^2u_{x}\Big)\,dt+\Big(0.093_{\pm0.002} u_x^2\Big)^{1/2}\,dW(t)$}\\\midrule
VB&\multicolumn{8}{l}{$du = 0.992_{\pm0.008}u_{xx}\,dt+ \left(0.091_{\pm0.002}u_{x}^2\right)^{1/2}\,dW(t)$}\vspace{2pt}\\
\bottomrule\toprule
\multicolumn{9}{c}{(II) Equation~\eqref{eq_stoch_heat_mix}}\\
\midrule
\multirow{2}{*}{Method}& \multicolumn{2}{c}{Prec ($\uparrow$)}& \multicolumn{2}{c}{Recall ($\uparrow$)}& \multicolumn{2}{c}{$E_{\text{in}}$ ($\downarrow$)}&\multicolumn{2}{c}{$E_{\text{out}}$ ($\downarrow$)}
\\
&Drift&Diff.&Drift&Diff.&Drift&Diff.&Drift&Diff.\\
\midrule
Proposed&0.8749&0.9100&1.0000&0.9050&2.98&18.50&12.61&9.29\\
e-SINDy &0.3008&1.0000&1.0000&1.0000&2.32&19.22&94.36&0.00\\
VB      &0.7883&0.9975&1.0000&1.0000&2.76&20.66&28.60&0.00\\
\midrule
\multicolumn{9}{c}{Most frequently identified model}\\\midrule
Proposed&\multicolumn{8}{l}{$du = 0.996_{\pm0.033} u_{xx}\,dt + \left(2.029_{\pm0.087} u + 0.415_{\pm 0.034} u_x \right)\,dW(t)
$}\\\midrule
e-SINDy&\multicolumn{8}{l}{$\begin{aligned}du =& \left(0.997_{\pm0.027}u_{xx}-1.597_{\pm2.069}u^2u_x-1.907_{\pm5.338}u^3)\right)\,dt\\
&+\Big(0.258_{\pm0.011} u_x^2+1.849_{\pm0.057} uu_x+3.208_{\pm0.050} u^2\Big)^{1/2}\,dW(t)\end{aligned}$}\\\midrule
VB&\multicolumn{8}{l}{$\begin{aligned}du =& 1.003_{\pm0.033}u_{xx}\,dt+\\
&\Big(0.253_{\pm0.009}u_x^2+1.824_{\pm0.042}uu_x+3.104_{\pm0.054}u^2\Big)^{1/2}\,dW(t)\end{aligned}$}\\
\bottomrule
\end{tabular}}
\end{table}
We compare Stoch-IDENT with e-SINDy and the Variational Bayesian (VB) method from~\cite{mathpati2024discovering}, highlighting key differences. First, e-SINDy and VB find the square of the diffusion term. This forces the use of higher-order terms for simple features (e.g., $u_x dW(t)$ requires $u_x^2$), requires a dictionary up to $M(M+1)/2$ terms for identifying $M$ features, and dampens small coefficients, making them hard to detect. Second, their dictionary was limited to terms like $u^q\partial_x^p u$ with integers $p$ and $q$, and their experiments mainly focused on additive noise.

Numerically, we test Stoch-IDENT, e-SINDy, and VB on identifying the following stochastic heat equations with multiplicative noise:
\begin{equation}\label{eq_stoch_heat_mult}
	du= u_{xx} dt + 0.3u_xdW(t)\;,\quad\text{for}~x\in [-\pi, \pi)~\text{and}~t\in(0,0.1)
\end{equation}
and
\begin{equation}\label{eq_stoch_heat_mix}
	du= u_{xx} dt + \left(2u+0.5u_x\right)dW(t)\;,\quad\text{for}~x\in [-\pi, \pi)~\text{and}~t\in(0,0.1)
\end{equation}
both with the initial condition $u(0,x) = 0.2\exp(\sin(3x-0.2))\cdot\cos(4x+0.8)$. Using the grid from Section~\ref{sec_general}, we simulate $N=50$ trajectories. We note that in e-SINDy and VB, the diffuse model to be identified differs from ours. Instead of $0.3u_x$ in~\eqref{eq_stoch_heat_mult} and $2u+0.5u_x$ in~\eqref{eq_stoch_heat_mix}, they seek to identify
$$0.09u_x^2\qquad\text{and}\qquad 4u^2+2uu_x+0.25u_x^2,$$ 
respectively. Hence, the evaluation performances are measured against different references.

Nevertheless, to ensure fairness,  we compare all methods using exactly the same drift dictionary of type $(4,3)$, i.e., partial derivatives with respect to $x$ up to order $4$ and products of at most $3$ terms, yielding $56$ candidate features each; and diffuse dictionary of type $(2,2)$. Stoch-IDENT uses default parameters. For both e-SINDy and VB, there are three tunable parameters: the regularization parameters for the drift $\lambda'_{\mathrm{drift}}$ and for the diffusion $\lambda'_{\mathrm{diffuse}}$, and the truncation parameter $\tau'$, which discards features whose coefficients fall below $\tau'$. We perform a grid search over $\lambda'_{\mathrm{drift}} \in \{0.2,\, 0.5,\, 0.8,\, 1.0\}$, $\lambda'_{\mathrm{diffuse}} \in \{0.01,\, 0.05,\, 0.1,\, 0.5\}$, and $\tau' \in \{1\times 10^{-3},\, 1\times 10^{-2}\}$, and report the combination yielding the highest drift and diffusion precision. For Equation~\eqref{eq_stoch_heat_mult}, the optimal parameters are $\lambda'_{\mathrm{drift}} = 0.8$ and $\lambda'_{\mathrm{diffuse}} = 0.05$ for both e-SINDy and VB, and for e-SINDy, the results are insensitive to $\tau'$, but for VB, the optimal is $\tau'=1\times10^{-2}$. For Equation~\eqref{eq_stoch_heat_mix}, the optimal parameters are $\lambda'_{\mathrm{drift}} = 0.8$ and $\lambda'_{\mathrm{diffuse}} = 0.1$, with no sensitivity to $\tau'$, for e-SINDy; and $\lambda'_{\mathrm{drift}} = 0.5$, $\lambda'_{\mathrm{diffuse}} = 0.1$, and $\tau' = 1\times 10^{-2}$ for VB.
Table~\ref{tab_compare} summarizes results of $100$ experiments.

For identifying~\eqref{eq_stoch_heat_mult}, Stoch-IDENT achieves the highest precision and recall for the drift model and remains comparable with e-SINDy and VB in terms of diffuse identification. We also report the most frequently identified models along with the mean and standard deviation of the reconstructed coefficients, and find that both VB and the proposed method recover the correct model. We highlight that our identification process for the diffuse terms differs from that of e-SINDy and VB, as it involves addressing a more challenging sparse regression problem~\eqref{eq_L0_diffuse_select} with quadratic measurements, which is harder than its linear counterpart~\eqref{eq_L0_drift_select_greedy}. The benefit of tackling this more challenging problem is that we directly recover the diffuse terms  rather than their squares.

Stoch-IDENT continues to perform satisfactorily in identifying~\eqref{eq_stoch_heat_mix}. Notably, Stoch-IDENT still achieves the highest drift precision in this example. As for coefficient recovery, e-SINDy achieves the lowest error for the true drift coefficients, while Stoch-IDENT yields the smallest error for the true diffuse coefficients. The slight underperformance of Stoch-IDENT in identifying the diffuse terms was expected, as it faces the challenge of a difficult sparse nonlinear regression problem~\eqref{eq_L0_diffuse_select}, whereas e-SINDy and VB rely on more tractable linear regression. According to the most frequent models reported in Table~\ref{tab_compare}(II), Stoch-IDENT correctly identifies the exact SPDE form with low estimation error. In contrast, e-SINDy often selects excessive terms, and VB's coefficients do not clearly reveal the true diffusion structure, as they are unlikely to conform to the required binomial form. Furthermore, we highlight that the parameters of Stoch-IDENT remain the same across both examples, whereas e-SINDy and VB require fine-tuning to achieve their best performance.

\subsection{Simulations from the identified model}

\begin{figure}
\begin{tabular}{ccc}
(a)&(b)&(c)\\
\includegraphics[width=0.3\textwidth]{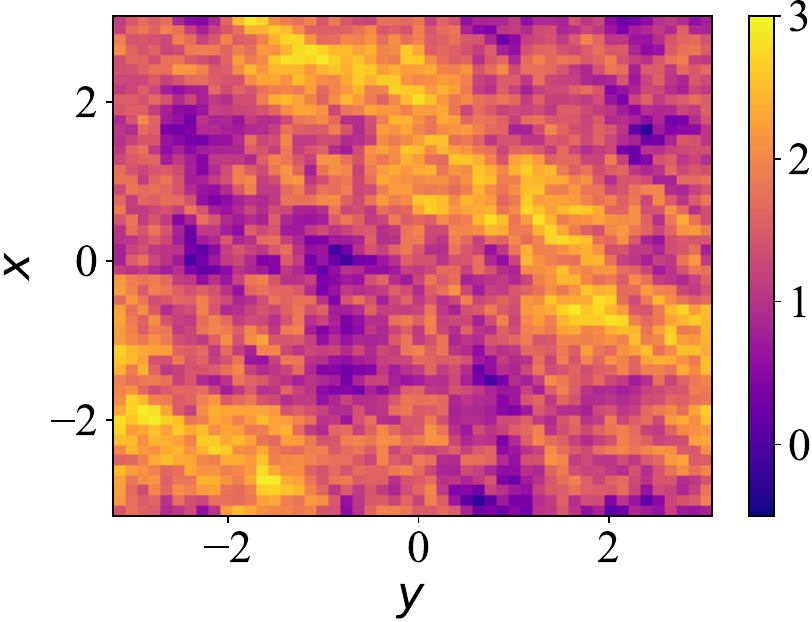}&
\includegraphics[width=0.3\textwidth]{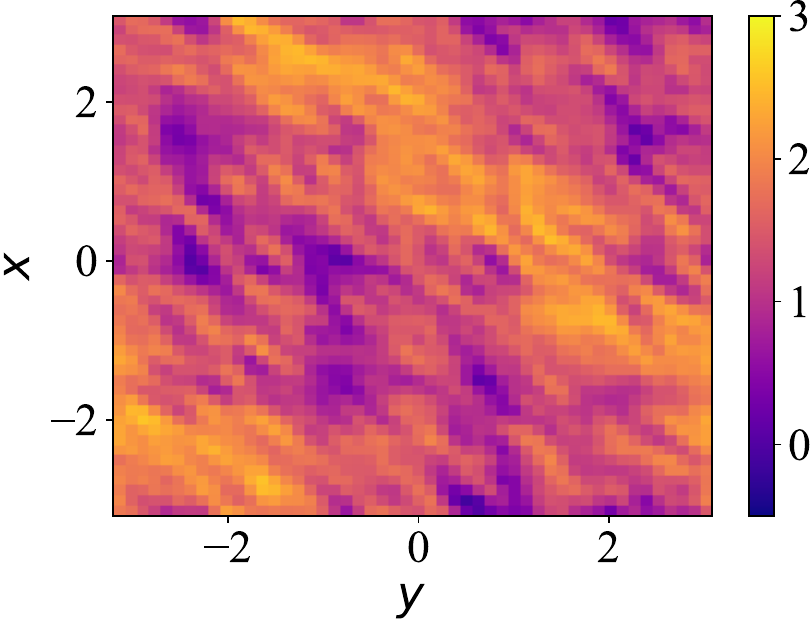}&
\includegraphics[width=0.3\textwidth]{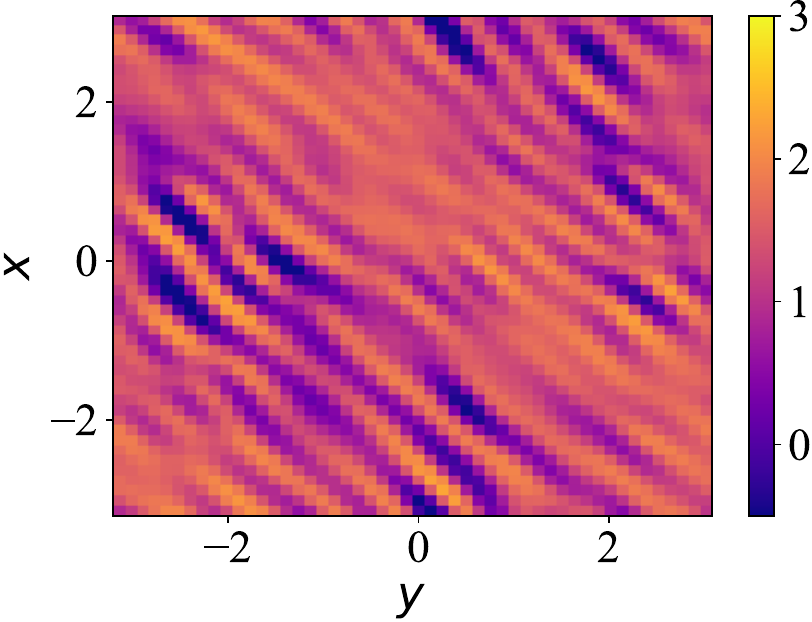}\\
(d) &(e)&(f)\\
\includegraphics[width=0.3\textwidth]{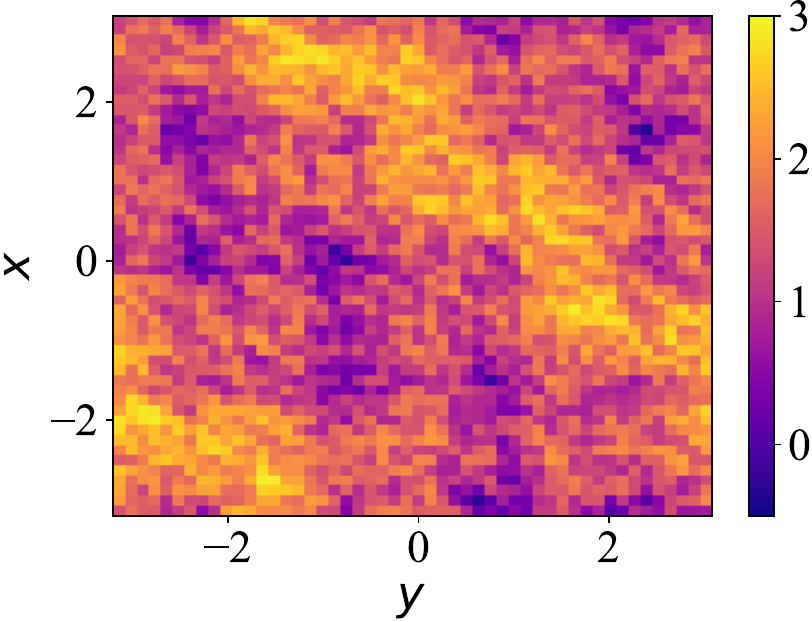}&
\includegraphics[width=0.3\textwidth]{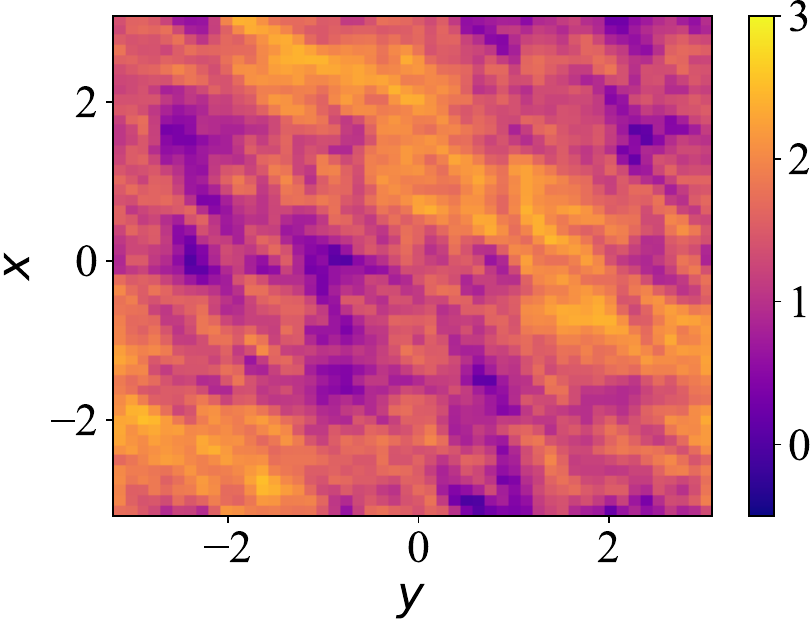}&
\includegraphics[width=0.3\textwidth]{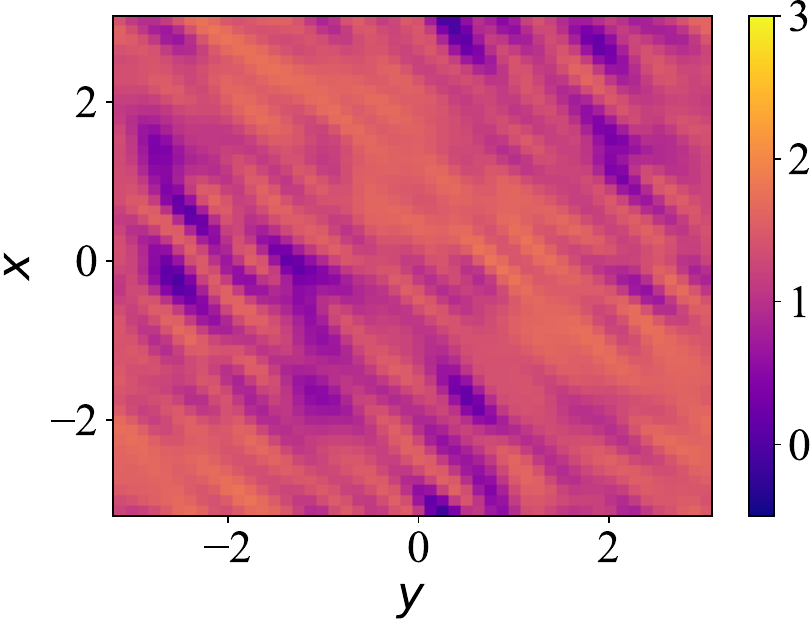}\\
\end{tabular}
\caption{Comparison between a solution path of stochastic Allen-Cahn~\eqref{eq_allen_cahn} at (a) $t=8\times 10^{-3}$ (b) $t=2\times 10^{-2}$, and (c) $t=8\times 10^{-2}$ with  the simulation by the identified SPDE~\eqref{eq_allen_cahn_simulation} at (d)-(f) the same time points. The identification is based on $20$ trajectories sampled on a coarse grid, and the simulated dynamics exhibit a similar pattern formation to that of the true model.}\label{fig_allen_cahn}
\end{figure}

We demonstrate the applicability of Stoch-IDENT by comparing the observed trajectory with the simulation from the identified model. For this experiment, we consider the  \textbf{stochastic Allen-Cahn equation} with multiplicative noise (see e.g. \cite{MR4019051,MR3812224}):
\begin{equation}\label{eq_allen_cahn}
d u=\left(0.5 \Delta u-2(u^3-u)\right)dt+ (u_x + u_y) dW(t)\;,\quad\text{for}~x\in [-\pi, \pi)~\text{and}~t\in(0,0.08)
\end{equation}
with the initial condition  $u(0,x,y)=\sqrt{\sin^2(2x)+\cos^2(2y)} + 0.5\exp(\sin(x+y)) + \zeta(x,y)$ where $\zeta(x,y)$ is a fixed random field independently sampled from a uniform distribution $\cU(-1,1)$ .  The  grid for the data contains equidistant $100$ points in time and $50$ points in both dimensions of the space. When solving~\eqref{eq_allen_cahn}, we up-sample the time by $50$ times.  The first row of Figure~\ref{fig_allen_cahn} shows  sample paths at $t=8\times 10^{-3}$, $t=2\times 10^{-2}$, and $t=8\times 10^{-2}$, where  stripe patterns  are formed. We use dictionaries of type $(3,3)$ for the drift part, and type $(2,2)$ for the diffusion part. 

With $N=20$ trajectories, we can identify the following model
\begin{equation}\label{eq_allen_cahn_simulation}
du=\left(0.333u_{xx} + 0.337u_{yy}-2.059u^3 + 2.084u\right)dt + \left(0.726u_x + 0.850 u_y\right) dW(t)\;,
\end{equation}
which contains the correct terms as in~\eqref{eq_allen_cahn}.  The coefficients for the Laplacian and diffusion terms show lower accuracy compared to the reaction terms. This stems from numerical errors in approximating differential features when using coarse grids. As seen in the second row of Figure~\ref{fig_allen_cahn}, when we simulate the path using the same Wiener process that drove the original dynamics, we observe a similar pattern formation.
In practice, the underlying Wiener process is typically unknown. Consequently,  evaluation metrics like TEE~\cite{kang2021ident} would be time-consuming~\cite{mathpati2024discovering}.

\subsection{Time complexity of Stoch-IDENT}

\begin{table}
\centering
\caption{Average runtime (in seconds) of $10$ runs of Stoch-IDENT when (I) the number of input trajectories is $N$, the dictionary for drift is of type $(4,3)$, and the dictionary for diffusion is of type $(2,2)$; (II) the dictionary for drift is of type $(p_{\text{drift}},2)$, dictionary for diffusion is of type $(2,2)$, and $N=20$; (III) the dictionary for diffusion is of type $(p_{\text{diffusion}},2)$, dictionary for drift is of type $(3,3)$, and $N=10$. The underlying SPDE is fixed as the stochastic transport equation~\eqref{eq_example_transport}.}\label{tab_time}

\begin{tabular}{ccc|ccc|ccc}
\toprule
\multicolumn{3}{c}{(I)}&\multicolumn{3}{c}{(II)}&\multicolumn{3}{c}{(III)}\\
\toprule
\multirow{2}{*}{$N$}& \multicolumn{2}{c|}{Time (sec.)} & \multirow{2}{*}{$p_{\text{drift}}$}&  \multicolumn{2}{c|}{Time (sec.)}&\multirow{2}{*}{$p_{\text{diffusion}}$}& \multicolumn{2}{c}{Time (sec.)}\\
&Drift&Diffusion&&Drift&Diffusion&&Drift&Diffusion\\
\midrule
1&2.32&11.31&2&2.51&13.29&2& 4.27& 12.98\\
10&4.04&12.23&3&2.97&13.37&3&4.23& 20.96\\
100&14.33&21.54&4&3.48&13.38&4&4.25& 21.93\\
200&26.10&28.24&5&4.10&13.41&5&4.26&21.25\\
400&48.11&44.34&6&4.73&13.37&6&4.27&26.62\\
\bottomrule
\end{tabular}
\end{table}

We study the time complexity of Stoch-IDENT\footnote{In this work, the  data generation and identification algorithms are implemented in \texttt{C++} and the construction of feature systems is realized in parallel with \texttt{OpenMP}. All the experiments are conducted in a MacBook Pro with an Apple M3 Max chip ($16$ threads) and 128 GB memory. The reported times include both constructions of the feature systems and the identification processes.}. The runtime of our method is primarily governed by three factors: the number of input trajectories ($N$), the size of the dictionary for the drift terms, and the size of the dictionary for the diffusion terms. 

In Table~\ref{tab_time}, we report the average runtime over 10 runs of Stoch-IDENT for identifying the stochastic transport equation~\eqref{eq_example_transport} as these factors are varied. (I) With the dictionaries for the drift and diffusion features fixed to types $(4,3)$ and $(2,2)$ respectively, we vary $N$ from $1$ to $400$. We observe that the time required to identify both the drift and diffusion parts grows linearly with $N$. The time for the diffusion part is greater than that for the drift part by approximately a constant factor. (II) With $N=10$ fixed, we increase the maximum order of partial derivatives in the dictionary for drift features. We see that this only affects the identification time of the drift part, and this relationship is linear. (III) When we increase the maximum order of partial derivatives in the dictionary for diffusion features,  the identification time for the diffusion part increases linearly.

\section{Conclusion}\label{sec_conclusion}
In this paper, we introduce Stoch-IDENT, a novel framework for identifying SPDEs driven by time-dependent Wiener processes with both additive and multiplicative noise structures. We provide  the theoretical foundation for the identifiability of linear SPDEs with constant coefficients from trajectory data, establishing  conditions under which the drift and diffusion operators can in principle be uniquely determined, independently of any specific algorithm. Specifically, we show that the Fourier modes of the initial data are essential for ensuring the uniqueness of the identified SPDEs. Furthermore, by analyzing the effective dimension of solution trajectories, we demonstrate the intrinsic difficulty of identifying parabolic equations compared to hyperbolic ones.  Algorithmically, we generalize Robust-IDENT~\cite{he2022robust} to recover drift terms and design a new greedy algorithm, QSP, to identify diffusion terms. Notably, QSP is applicable to general sparse regression problems with quadratic measurements, such as phase retrieval and distance-based localization~\cite{fan2018variable}. Through a series of experiments on linear and nonlinear high-order SPDEs, we verify Stoch-IDENT's effectiveness and illustrate its behavior.  

Several interesting directions remain for future work. While the current study focuses on time-dependent Wiener processes and controlled synthetic examples, these include extending Stoch-IDENT to space-time Wiener processes; applying it to real experimental data such as turbulence, materials science, and biophysics; establishing convergence and stability guarantees for QSP; and incorporating uncertainty quantification strategies such as stability selection~\cite{meinshausen2010stability} into the identification pipeline.

\appendix
\section{Some Proofs}
\label{sec3-appendix}
\subsection{Proof of Proposition \ref{prop-drift-identification} in additive noise case}\label{sec3-appendix-prop1}

\begin{proof}
  We only show the proof in the multiplicative noise case, since the one in the additive noise case is similar. The main difference between the proofs in these two cases is that, for multiplicative noise, we begin by taking the expectation of \eqref{mild-fourier-mul}. In contrast, for additive noise, we first take the expectation of \eqref{mild-fourier}.   

For simplicity, we assume that $t_2>t_1=0.$ 
In the multiplicative noise case,  by taking the polar coordinates on \eqref{mild-fourier-mul}, for any $t_1\neq t_2,$ and $\widehat u(t_1, \xi)\neq 0,$ it holds that
\vspace{-2mm}
{\small
\begin{align*}
    (2\pi)^{\frac d2}\log\Big(\Big|\frac {\widehat u(t_2, \xi)}{\widehat u(t_1, \xi)}\Big|\Big)&=\sum_{|\alpha|\le p_1, \alpha \; \text{even}} a_{\alpha}(\bi \xi)^{\alpha}(t_2-t_1)
    +\sum_{|\beta|\le p_2, \beta \; \text{even}} q_{\beta}(\bi \xi)^{\beta}(W(t_2)-W(t_1))\;,\\
       (2\pi)^{\frac d2}\operatorname{Arg} \Big(\Big|\frac {\widehat u(t_2, \xi)}{\widehat u(t_1, \xi)}\Big|\Big)&=\sum_{|\alpha|\le p_1, \alpha \; \text{odd}} a_{\alpha}(\bi \xi)^{\alpha}\bi^{-1}(t_2-t_1)
    +\sum_{|\beta|\le p_2, \beta \; \text{odd}} q_{\beta}(\bi \xi)^{\beta}\bi^{-1}(W(t_2)-W(t_1))\;.
\end{align*}}
By taking expectation on above equalities, we have that 
\vspace{-2mm}
\begin{align}\label{dri-mul-ide1}
    \frac {(2\pi)^{\frac d2}}{t_2-t_1} \mathbb E \Big[\log\Big(\Big|\frac {\widehat u(t_2, \xi)}{\widehat u(t_1, \xi)}\Big|\Big)\Big]&=\sum_{|\alpha|\le p_1, \alpha \; \text{even}} a_{\alpha}(\bi \xi)^{\alpha},\\\label{dri-mul-ide2}
      \frac {(2\pi)^{\frac d2}}{t_2-t_1}  \mathbb E \Big[\operatorname{Arg} \Big(\Big|\frac {\widehat u(t_2, \xi)}{\widehat u(t_1, \xi)}\Big|\Big)\Big]& =\sum_{|\alpha|\le p_1, \alpha \; \text{odd}} a_{\alpha}(\bi \xi)^{\alpha}\bi^{-1}.
\end{align}
Notice that for $|t_2-t_1|>0$ sufficiently small, the phase ambiguity in \eqref{dri-mul-ide2} can be removed. 

We can choose the Fourier modes $\xi_k\in \mathcal Q$ with $k=1,\cdots,\widetilde K\ge |\mathcal Q|.$
Then we can rewrite \eqref{dri-mul-ide1} and \eqref{dri-mul-ide2} as \vspace{-2mm}
\begin{align*}
    &\frac {(2\pi)^{\frac d2}}{t_2-t_1} \by_{even}=\bA_{even}\bc_{even},\; \frac {(2\pi)^{\frac d2}}{t_2-t_1} \by_{odd}=\bA_{odd}\bc_{odd}\;,
\end{align*}
where  $\bc_{even}^{\top}=(\bi^{|\alpha|}a_{\alpha})_{|\alpha|\le p_1, even},$
and $\bc_{odd}^{\top}=(\bi^{|\alpha|-1}a_{\alpha})_{|\alpha|\le p_1, odd}.$ Here 
\begin{align*}
(\by_{even})_k=\mathbb E \log\Big(\Big|\frac {\widehat u(t_2, \xi_k)}{\widehat u(t_1, \xi_k)}\Big|\Big),(\bA_{even})_{k\alpha}=\xi_k^{\alpha}, |\alpha|\le p_1 \;\text{and} \; \text{is even}\;,\\
(\by_{odd})_k=\mathbb E Arg\Big(\Big|\frac {\widehat u(t_2, \xi_k)}{ \widehat u(t_1, \xi_k)}\Big|\Big),(\bA_{odd})_{k\alpha}=\xi_k^{\alpha}, |\alpha|\le p_1 \;\text{and} \; \text{is odd}\;.  
\end{align*}
From the assumption on $\mathcal Q$, we have that $\bA_{even}$ and $\bA_{odd}$ are both of full rank. This implies that $a_{\alpha}$ can be uniquely determined.

\end{proof}

\iffalse
It can be seen that,  
to obtain a unique identification of drift terms in  the additive noise, it suffices to let  
$\mathcal Q=\{\xi\in \mathbb Z^d: \mathbb E \widehat u_0(\xi)\neq 0\}$ satisfy the condition in Proposition \ref{prop-drift-identification}, and assume that two time points of the averaged solutions $\mathbb Eu(t_1, \cdot),\mathbb E u(t_2, \cdot)$ are given.
\fi

\subsection{Proof of Theorem \ref{diffusion-identification-covariances} in  additive noise case}\label{sec3-appendix-add-diffusion}

\begin{proof}

In the additive noise case,
according to  Ito's isometry, it holds that
\vspace{-2mm}
{\small \begin{align*}
&\mathbb E \Big[\Big|\widehat u(t_2, \xi)-\widehat u(t_1, \xi) \exp\Big((2\pi)^{-\frac d2}\sum_{|\alpha|=0}^{p_1} a_{\alpha} (\bi \xi)^{\alpha}(t_2-t_1)\Big)\Big|^2\Big]\\
&=\int_{t_1}^{t_2}\exp\Big(2(2\pi)^{-\frac d2}\sum_{|\alpha|\le p_1, \; \text{even} } a_{\alpha} (\bi \xi)^{\alpha}(t_2-s)\Big) |\widehat R(\xi) |^2 ds  \times 
  \Big|(2\pi)^{-\frac d2}\sum_{|\beta|=0}^{p_2}b_{\beta}(\bi \xi)^{\beta}\Big|^2.
\end{align*}}
Notice that by our assumption {\small $$\widetilde F(t_1, \xi,t_2):=\int_{t_1}^{t_2}\exp\Big(2(2\pi)^{-\frac d2}\sum_{|\alpha|\le p_1, \; \text{even} } a_{\alpha} (\bi \xi)^{\alpha}(t_2-s)\Big) |\widehat R(\xi) |^2 ds>0,$$}
we can define
\vspace{-2mm}
{\small\begin{align*}
 F(t_1, \xi,t_2):=\frac {\mathbb E \Big[(2\pi)^d\Big|\widehat u(t_2, \xi)-\widehat u(t_1, \xi) \exp\Big((2\pi)^{-\frac d2}\sum_{|\alpha|=0}^{p_1} a_{\alpha} (\bi \xi)^{\alpha}(t_2-t_1)\Big)\Big|^2\Big]}{\widetilde F(t_1, \xi,t_2)}.
\end{align*}}
Then the identification problem can be reformulated as 
\vspace{-2mm}
{\small\begin{align}\nonumber 
F(t_1, \xi,t_2)&=\Big|\sum_{|\beta|\le p_2}b_{\beta}(\bi \xi)^{\beta}\Big|^2
\\\label{cov-iden}
&=\sum_{|\beta|,|\widetilde \beta|\le p_2,\; \widetilde \beta, \; \beta \; \text{even}}\bi^{|\beta+\widetilde \beta |}b_{\beta}q_{\widetilde \beta}(\xi)^{\beta+\widetilde \beta}+\sum_{|\beta|,|\widetilde \beta|\le p_2,\; \widetilde \beta, \; \beta \; \text{odd}}\bi^{|\beta|-|\widetilde \beta| }b_{\beta}q_{\widetilde \beta}(\xi)^{\beta+\widetilde \beta}.
\end{align}}
Take $\xi_k\in \mathcal Q_1, k\le \widetilde K$ with $\widetilde K\ge |\mathcal Q_1|.$ 
It follows from \eqref{cov-iden} that
\vspace{-2mm}
\begin{align*}
\by_{str}= \bA_{str}\bc_{str},    
\end{align*}
where $(\by_{str})_k=F(t_1, \xi_k,t_2),$
$(\bA_{str})_{k\gamma}=\xi^{\gamma}_k, k\le \widetilde K, |\gamma|\le 2p_2  \;\text{is even},$
and 
\vspace{-2mm}
\begin{align}\label{mult-def}
(\bc_{str})_{\gamma}=\sum_{|\beta|,|\widetilde \beta|\; even} \bi ^{|\beta+\widetilde \beta|}b_{\beta}q_{\widetilde \beta}+\sum_{|\beta|,|\widetilde \beta| \; odd} \bi ^{|\beta|-|\widetilde \beta|}b_{\beta}q_{\widetilde \beta}
\end{align}
with $\beta+\widetilde \beta=\gamma$.
By the assumption on $\mathcal Q_1$, $\bA_{str}$ is of full rank and thus $\bc_{str}$ is uniquely determined. 
Since \eqref{mult-def} is a quadratic function with respect to $b_{\beta}$, there exists at most $2^{\binom{p_2+d}{d}}$ isolated solutions, which are viewed as an equivalent class.  Therefore, $b_{\beta}$ is uniquely determined in the sense of an  equivalent class. 
\end{proof}

\iffalse
It can be seen that, to obtain unique identification in the equivalent class, one can refine the assumption of $\mathcal Q_1$ in Theorem \ref{diffusion-identification-covariances}. In fact, to let Theorem \ref{diffusion-identification-covariances} holds, 
\begin{itemize}
    \item in the additive noise case,  it suffices to impose that Let 
$\mathcal Q_1=\{\xi\in\mathbb Z^d|\widehat{R}(\xi)\neq 0\}$ satisfy 
$$|\mathcal Q_1|\ge \sum_{k=0}^{p_2}\binom{2k+d-1}{d-1},$$ and 
suppose that $\mathcal Q_1$ is not located on an algebraic polynomial hypersurface of degree $\le 2p_2$ consisting of only even-order terms.
    \item in the multiplicative noise case, it suffices to require that Let $\mathcal Q_1=\{\xi\in \mathbb Z^d: \mathbb E \widehat u_0(\xi)\neq 0\}$ satisfy 
 $$|\mathcal Q_1|\ge \max \Bigg(\sum_{k=0}^{p_2} \binom{2k+d-1}{d-1} ,\sum_{k=0}^{p_2-1} \binom{2k+d}{d} \Bigg),$$
 and suppose that $\mathcal Q_1$ is not located on an algebraic polynomial hypersurface of degree $\le 2p_2$ consisting of only even-order terms or odd-order terms.
\end{itemize}
\fi 
Below, we provide one simple example of possible multiple solutions in the equivalent class.

\begin{example}
Let $d=1$, $|\beta|\le p_2=2$. Then from  \eqref{cov-iden}, we are solving the following system:
\begin{align*}
F(t_1, \xi,t_2)
=b_0^2+(-2b_2b_0+b_1^2)\xi^2+
b_2^2 \xi^4.
\end{align*}
Let the number $\widetilde K>0$ of Fourier modes be large enough (larger than $3$ in this example). 
Then there exists a unique vector $(x_1,x_2,x_3)$ such that $b_0^2=x_1,-2b_2b_0+b_1^2=x_2,b_2^2=x_3.$ If the solution $(b_0,b_1,b_2)$ exists for \eqref{cov-iden}, then the number of the isolated solutions is at most $2^3$. 
\end{example}

\subsection{Proof of Theorem \ref{theorem-diffusion-identification} in additive noise case}\label{sec-appendix-add-data}

\begin{proof}
To verify \eqref{add-over-err}, it suffices to
 show that that there exists linear spaces $V$ and $V'$ with dimensions $L=\mathcal O(|\log \epsilon|^2)$ such that 
\begin{align}\label{add-over-err1}
\|e^{\mathcal L t}u_0- P_{V}e^{\mathcal L t}u_0\|_{L^2(\Omega;H)}\lesssim \epsilon(1+\|u_0\|_{L^2(\Omega;H)}),
\end{align}
and that 
\vspace{-1mm}
\begin{align}\label{add-over-err2}
\|(I-P_{V'})u_{sto}\|_{L^2(\Omega;H)}\lesssim \epsilon(1+\|u_0\|_{L^2(\Omega;H)}).
\end{align}
Here we denote $u_{sto}(t)=\sum_{k=1}^{\infty}\int_0^t e^{-(\lambda_k-\mu)(t-s)} q_k R_k \phi_k(x) dW(s)$ for convenience.

Consider the Galerkin approximation of $e^{\mathcal L t}$ with parameter $\widetilde M$,
\vspace{-3mm}
$$u_{ini,\widetilde M}(t,x)=e^{\mu t}\sum_{k=1}^{\widetilde M}c_ke^{-\lambda_k t}\phi_k(x)$$
and that of the stochastic convolution with parameter $M'\in \mathbb N,$
$$u_{sto,M'}(t,x)=\sum_{k=1}^{M'}\int_0^t e^{\mu (t-s)} e^{-\lambda_k (t-s)} q_k R_k \phi_k(x) dW(s).$$ We can take $\widetilde K\in \mathbb N^+$ such that $2\mu-\Re(\lambda_k)<0$ for any $k\ge \widetilde K.$ On the other hand, by the assumption,  there exists $M_{\epsilon}\in \mathbb N^+$ such that $\sum\limits_{k=M_{\epsilon}}^\infty
\frac{1}{\Re(\lambda_k)}
|q_k|^2|R_k|^2\le \epsilon^2.$

By the It\^o's isometry and the fact $u_0$ is $\mathcal F_0$-measurable, we have that for $\widetilde M\ge \widetilde K+\epsilon^{\frac 2{1-2\gamma}},$
\vspace{-2mm}
\begin{align}\label{add-est0}
\mathbb E [\|u_{int,\widetilde M}(t, \cdot)-e^{\mathcal Lt}u_0\|^2]
&\le \theta^2 \frac {\widetilde M^{1-2\gamma}}{2\gamma-1}.
\end{align}
and that for $M'\ge \widetilde K+M_{\epsilon},$
\vspace{-2mm}
\begin{align}\nonumber
\mathbb E [\|u_{sto,M'}(t, \cdot)-u_{sto}(t)\|^2] &=\mathbb E [\int_0^t \sum_{k=M'+1}^\infty
e^{-2(\Re(\lambda_k)-\mu)(t-s)}|q_k|^2|R_k|^2\|\phi_k\|^2 ds]\\\nonumber
&\lesssim \sum_{k=M'+1}^\infty
\frac{1}{2\Re(\lambda_k)}
|q_k|^2 |R_k|^2 \lesssim \epsilon^2.
\end{align}

We first prove \eqref{add-over-err1}.
Define \vspace{-2mm}
\begin{align*}
w_{\epsilon}^1(t)&=\sum_{k=1}^{\widetilde M}c_k\sum_{l=0}^{L_{\epsilon}} (-1)^l \frac {(\lambda_kt-\mu t)^l}{l!}\phi_k(x),
\end{align*}
where $\widetilde M, L_{\epsilon}>0,$ will be determined later. 
Then for each $t$, $w_{\epsilon}^1$ sits in the linear space 
\vspace{-1mm}
$$V_1=\text{span} \Big\{\sum_{k=1}^{\widetilde M}c_k(-1)^l\frac {(\lambda_kt-\mu t)^l}{l!}\phi_k(x), l=0,1\cdots,L_{\epsilon}\Big\}.$$
By the Taylor expansion and Minkowski's inequality, as well as the independent increment of $W(\cdot)$, we have that 
\vspace{-3mm}
{\small\begin{align*}
     \|w_{\epsilon}^1(t, \cdot)-e^{\mu t}\sum_{k=1}^{\widetilde M}c_ke^{-\lambda_k t}\phi_k\|_{L^2(\Omega;H)}^2   
    &\le \Big\|\sum_{k=1}^{\widetilde M}c_k\phi_k(x)\sum_{l=L_{\epsilon}+1}^{\infty}(-1)^l\frac {(\lambda_kt-\mu t)^l}{l!}\Big\|_{L^2(\Omega;H)}^2\\
    &\le \sum_{k=1}^{\widetilde M}\mathbb E [|c_k|^2]
\Big|\sum_{l=L_{\epsilon}+1}^{\infty}(-1)^l\frac {(\lambda_kt-\mu t)^l}{l!}\Big|^2\\
    &\le \sum_{k=1}^{\widetilde M}\mathbb E[|c_k|^2] \frac 1{[(L_{\epsilon}+1)!]^2},
    \end{align*}}
where we require that 
$\sup\limits_{k\le \widetilde M}|\lambda_k-\mu|t\le 1,$ i.e., $t\in [0, \inf\limits_{k\le \widetilde M}\frac 1{|\lambda_{k}-\mu|}]$. It follows that for $t\in [0, \inf\limits_{k\le \widetilde M}\frac 1{|\lambda_{k}-\mu|}]$,
\vspace{-2mm}
\begin{align*}
     \|e^{\mu t}\sum_{k=1}^{\widetilde M}c_ke^{-\lambda_k t}\phi_k-w_{\epsilon}^1(t, \cdot)\|_{L^2(\Omega;H)}^2
     &\le \theta^2 \frac {\widetilde M^{1-2\gamma}}{2\gamma-1} e^{-2L_\epsilon}. 
\end{align*}
Letting $L_{\epsilon}=|\log(\epsilon)|$ and $\widetilde M=\widetilde K+\epsilon^{\frac 2{1-2\gamma}},$ and using \eqref{add-est0}, we have that   for $t\in [0,\inf\limits_{k\le \widetilde M}\frac 1{|\lambda_{k}-\mu|}]$,
\vspace{-2mm}\begin{align}\label{add-err-ini}
     \|e^{\mathcal Lt}u_0-w_{\epsilon}^1(t, \cdot)\|_{L^2(\Omega;H)}\lesssim \epsilon(1+\|u_0\|_{L^2(\Omega;H)}). 
    \end{align}

For $t\in [\inf\limits_{k\le \widetilde M}\frac 1{|\lambda_{k}-\mu|},T],$ by \eqref{exp-decaying} and taking $\kappa=\sup\limits_{k\le \widetilde M}|\log(|\lambda_{k}-\mu|)|/|\log(\epsilon)|$, there exists 
$\mathcal A_L$ which is an approximation of $e^{\mathcal L t}$ such that $\|e^{\mathcal L t
}- \mathcal A_L\|_{H\to H}\lesssim \epsilon$ with $t\in [t_0,T]$ such that $t_0=\epsilon^{\kappa}$, and $L=C_{\mathcal {L}}(\kappa)|\log \epsilon|^2.$
This implies that for any $t\in [\inf\limits_{k\le \widetilde M}\frac 1{|\lambda_{k}-\mu|},T],$ there exists a linear space $V_2$ with dimension $L=C_{ {\mathcal L}}(\kappa)|\log \epsilon|^2$ such that
\vspace{-2mm}
\begin{align*}
\|e^{\mathcal L t}u_0-P_{V_2}e^{\mathcal L t}u_0\|\lesssim \epsilon(1+\|u_0\|_{L^2(\Omega;H)}).
\end{align*}
This, together with \eqref{add-err-ini}, yields that for any $t\in[0,T],$ there exists  a linear space  $V_3$  containing $V_1\cup V_2$ such that 
\begin{align}
\|e^{\mathcal L t}u_0-P_{V_3}e^{\mathcal L t}u_0\|\lesssim \epsilon(1+\|u_0\|_{L^2(\Omega;H)}).
\end{align}

Next, we prove \eqref{add-over-err2}. 
Define $w_{\epsilon}^2$ by
\vspace{-2mm}
$$w_{\epsilon}^2(t):=\sum_{k=1}^{M'}\int_0^t  \sum_{l=0}^{L_{\epsilon}} (-1)^l \frac {((\lambda_k-\mu)(t-s) )^l}{l!} q_k R_k \phi_k(x) dW(s)$$
with $M', L_{\epsilon}$ being determined later.
It can be seen that $w_{\epsilon}^2$ sits in the linear space 
\vspace{-1mm}
{\small $$V_4=\text{span} \Big\{\sum_{k=1}^{M'}(-1)^l\int_0^t \frac {((\lambda_k-\mu)(t-s))^l}{l!}q_k R_k dW(s) \phi_k(x), l=0,1\cdots,L_{\epsilon}\Big\}.$$}
Next, we use the Taylor expansion and It\^o's isometry, 
\begin{align*}
&\Big\|w_\epsilon^2(t, \cdot)-\sum_{k=1}^{M'}\int_0^t e^{-(\lambda_k-\mu)(t-s)} q_k R_k \phi_k(x) dW(s)\Big\|_{L^2(\Omega;H)}^2 \\
&\le  \Big\|\sum_{k=1}^{M'}\int_0^t \sum_{l=L_{\epsilon}+1}^{\infty} (-1)^l \frac {((\lambda_k-\mu)(t-s) )^l}{l!}  q_k R_k\phi_k(x)  dW(s)\Big\|_{L^2(\Omega;H)}^2\\
&\le \sum_{k=1}^{M'}
\int_0^t \Big|\sum_{l=L_{\epsilon}+1}^{\infty}(-1)^l\frac {(\lambda_k-\mu)^l(t-s)^l}{l!}\Big|^2|R_k|^2|q_k|^2 ds.
\end{align*}
By requiring that $\sup\limits_{k\le M'}|\lambda_k-\mu|t\le 1,$ i.e., $t\in [0,\inf\limits_{k\le M'}\frac 1{|\lambda_{k}-\mu|}],$  and taking $L_\epsilon=|\log(\epsilon)|$,  it follows that for $t\in [0,\inf\limits_{k\le M'}\frac 1{|\lambda_{k}-\mu|}],$ 
\begin{align}\nonumber 
&\Big\|w_\epsilon^2(t, \cdot)-\sum_{k=1}^{M'}\int_0^t e^{-(\lambda_k-\mu)(t-s)} q_kR_k \phi_k(x) dW(s)\Big\|_{L^2(\Omega;H)}^2\\\nonumber 
&\le \sum_{k=1}^{M'}
\int_0^t \frac 1{[(L_{\epsilon}+1)!]^2}|q_k|^2|R_k|^2 ds\\\label{err-small-int}
&\lesssim \sum_{k=1}^{M'}|q_k|^2|R_k|^2 e^{-2L_{\epsilon}} t\lesssim \sum_{k=1}^{M'}|q_k|^2|R_k|^2 \inf\limits_{k\le M'}\frac 1{|\lambda_{k}-\mu|} e^{-2L_{\epsilon}}\lesssim \epsilon^2,
\end{align}
where we also used the estimate $\inf\limits_{k\le M'}\frac 1{|\lambda_{k}-\mu|} \le \inf\limits_{k\le M'}\frac 1{|\Re(\lambda_k)-\mu|}  \lesssim \frac 1{\Re(\lambda_k)}$ since $\Re (\lambda_{M'})>2\mu$ for  $M'$ large enough,  and  the fact that $$\sum_{k=1}^{\infty}|q_k|^2|R_k|^2 \frac 1{
\Re(\lambda_{k})}\le \sum_{k=1}^{\infty}|q_k|^2 |R_k|^2 \frac 1{\Re( \lambda_{k})^{1-\theta_1}}<\infty.$$ 

We denote $t_1=\inf_{k\le M'}\frac 1{|\lambda_{k}-\mu|}$ for simplicity. For $t\in [t_1,T],$ to obtain \eqref{add-over-err2},
we will show that there exist  linear spaces $V_5,V_6$ such that 
\begin{align}\label{sto-con-err1}
    \|(I-P_{V_5})\sum_{k=1}^{M'}\int_{t-t_1}^t e^{-(\lambda_k-\mu)(t-s)} q_k R_k \phi_k(x) dW(s)\|_{L^2(\Omega;H)}\lesssim \epsilon,\\\label{sto-con-err2}
     \|(I-P_{V_6})\sum_{k=1}^{M'}\int_{0}^{t-t_1} e^{-(\lambda_k-\mu)(t-s)} q_k R_k \phi_k(x) dW(s)\|_{L^2(\Omega;H)}\lesssim \epsilon.
\end{align}

To verify \eqref{sto-con-err2} in $[t_1,T]$, by \eqref{exp-decaying} with $L=C_{\frac {\mathcal L}2}(\kappa)|\log(\epsilon)|^2$ and $\kappa=\sup\limits_{k\le M'}\log $ $(|\lambda_{k}-\mu|)/|\log(\epsilon)|$, using the fact that $t-s\ge t_1 \ge \epsilon^{\kappa}$ and It\^o's isometry,  it holds that 
\begin{align*}
&\Big\|\sum_{k=1}^{M'}q_k R_k \int_0^{t-t_1} [e^{\frac {\mathcal L} 2(t-s)}-\mathcal A_{\frac L 2} (t-s)]e^{\frac {\mathcal L} 2(t-s)}\phi_k dW(s)\Big\|_{L^2(\Omega;H)}^2\\
&\lesssim \sum_{k=1}^{M'} |q_k|^2 |R_k|^2 \epsilon^2 \frac 1{\Re(\lambda_k)}\lesssim \epsilon^2.
\end{align*}
This implies that we can take the linear space
$V_6$ generated by 
$$\sum_{k=1}^{M'}q_k c_l'
(z_l+\frac 12 \mathcal L)^{-1} \int_0^{t-t_1} e^{-z_l(t-s)}e^{ \frac {\mathcal L}{2}(t-s)}dW(s)\phi_k, |l|\le L,$$ 
where $c_l',z_l\in\mathbb C$,
such that \eqref{sto-con-err2} holds.

To verify \eqref{sto-con-err1} in $[t_1,T]$, using It\^o's isometry, by change of variables and the boundedness of $e^{\mathcal L(\cdot)},$ we obtain that for  $\theta_1\in (0,1),$
\vspace{-2mm}
\begin{align*}
&\Big\|\sum_{k=1}^{M'}q_kR_k\int_{t-t_1}^t [e^{-\frac {\mathcal L} 2(t-s)}-I]e^{-\frac {\mathcal L} 2(t-s)}\phi_k dW(s)\Big\|_{L^2(\Omega;H)}^2
\\
&\lesssim \sum_{k=1}^{M'} \frac {|q_k|^2|R_k|^2}{\Re({\lambda_k})}  (e^{-\Re({\lambda_k})t_1}-e^{-2\Re({\lambda_k})t_1})\lesssim \sum_{k=1}^{M'} \frac {|q_k|^2|R_k|^2}{\Re({\lambda_k})^{1-\theta_1}}   t_1^{\theta}.
\end{align*}
Here we  take 
$M'\ge M_{\epsilon}$ large enough such  that 
$\Re(\lambda_{M'})^{-\theta_1}\le \epsilon^2$, and use the assumption 
that $ \sum_{k=1}^{\infty} \frac {|q_k|^2|R_k|^2}{\Re({\lambda_k})^{1-\theta_1}}<\infty.$
As a consequence, we can take $V_5$ generated by 
$\sum_{k=1}^{M'}q_kR_k \int_{t-t_1}^{t} e^{ \frac {\mathcal L}{2}(t-s)}dW(s)\phi_k$
such that \eqref{sto-con-err1} holds.

Combining \eqref{sto-con-err1}-\eqref{sto-con-err2}, and \eqref{err-small-int} together, we have that \eqref{add-over-err2}  holds for $t\in[0,T]$. 

\end{proof}

\subsection{Proof of Proposition~\ref{lemma_S_drift}}\label{proof_lemma_S_diffuse}
\begin{proof}
We only prove the estimate of The proof for the estimate of $S(\bb|\ba)$ is similar and thus omitted here. 
Since $\mathbb{E}\left[R^2_i(x,\ba^*)\right]=\mathbb{E}\int_{0}^{t_i}\left(\sum_{j=1}^J \bb_{j}^*G_j(s,x)\right)^2\,ds,$
we have that 
\vspace{-2mm}
{\small
\begin{align*}
  &S_{\diffuse}(\bb|\ba)
 = I^{-1}\cdot\sum_{i=1}^I\int_\cD\Bigg|\mathbb{E}\left[R^2_i(x,\widehat{\ba})\right]-\mathbb{E}\left[R^2_i(x,\ba^*)\right]+\mathbb{E}\int_{0}^{t_i}\left(\sum_{j=1}^J \bb_{j}^*G_j(s,x)\right)^2\,ds \\
  &\quad- \mathbb{E}\int_{0}^{t_i}\left(\sum_{j=1}^J\bb_{j}G_j(s,x)\right)^2\,ds\Bigg|^2\,dx.
\end{align*}}
By the Cauchy--Schwarz inequality and the assumptions on the moment boundedness of features, we obtain 
\vspace{-2mm}
\begin{align*}
    &|S_{\diffuse}(\bb|\ba)|\le C(\| \ba \| +\|\ba^*\|)^2(\| \ba-\ba^*\|^2)
    +C(\|\bb\| +\|\bb^* \|)^2\|\bb-\bb^*\|^2.
\end{align*}

\end{proof}

\section{Proof of Theorem~\ref{thm:main}}\label{sec:proof-main}
To prove the stability of the support recovery of QSP
(Algorithm~\ref{QSPalgo}) as stated in Theorem~\ref{thm:main}, we first establish some  lemmas.

\subsection{Preliminary lemmas}

\begin{lemma}[Support identification at initialization]
\label{lem:support_init}
Recall $q_s^{(0)} = \min_{c \in \R}\sum_{i=1}^{I}
\Bigl(\bG_i^{s,s}\,c - \eta_i^{(0)}\Bigr)^2$ for $s=1,\dots,J$ in Algorithm~\ref{QSPalgo}. Under Assumptions~\ref{ass:bounded} and~\ref{ass:signal}, for any $s \in S^*$ and $t \notin S^*$,
\[
    q_s^{(0)} < q_t^{(0)}.
\]
\end{lemma}

\begin{proof}
Note that $\eta_i^{(0)} = \zeta_i = (\starbc)^\top \bG_i \starbc +
\epsilon_i$ and the closed form:
\[
    q_s^{(0)}
    \;=\;
    \sum_{i=1}^{I}(\eta_i^{(0)})^2
    \;-\;
    \frac{\Bigl(\sum_{i=1}^{I}
          \bG_i^{s,s}\,\eta_i^{(0)}\Bigr)^2}
         {\sum_{i=1}^{I}(\bG_i^{s,s})^2}.
\]
Since the term $\sum_i(\eta_i^{(0)})^2$ is common to $q_s^{(0)}$
and $q_t^{(0)}$, it suffices to show:
\[
    \frac{\Bigl(\sum_{i}\bG_i^{s,s}\,\eta_i^{(0)}\Bigr)^2}
         {\sum_i(\bG_i^{s,s})^2}
    \;-\;
    \frac{\Bigl(\sum_{i}\bG_i^{t,t}\,\eta_i^{(0)}\Bigr)^2}
         {\sum_i(\bG_i^{t,t})^2}
    \;>\;0.
\]
We divide the proof into three steps.

\medskip
\noindent\textbf{Step 1. Lower bound for $s \in S^*$.}
Expanding $\eta_i^{(0)} = \sum_{u,v \in S^*}\bG_i^{u,v}
[\starbc]_u[\starbc]_v + \epsilon_i$:
\begin{align*}
    \sum_{i=1}^{I}\bG_i^{s,s}\,\eta_i^{(0)}
    &=
    \sum_{i=1}^{I}(\bG_i^{s,s})^2[\starbc]_s^2
    +
    \sum_{i=1}^{I}\bG_i^{s,s}
    \sum_{\substack{u,v \in S^*\\(u,v)\neq(s,s)}}
    \bG_i^{u,v}[\starbc]_u[\starbc]_v
    +
    \sum_{i=1}^{I}\bG_i^{s,s}\,\epsilon_i.
\end{align*}
By the triangle inequality:
\[
    \left|
    \sum_{i=1}^{I}\bG_i^{s,s}
    \sum_{\substack{u,v \in S^*\\(u,v)\neq(s,s)}}
    \bG_i^{u,v}[\starbc]_u[\starbc]_v
    \right|
    \;\leq\;
    \sum_{\substack{u,v \in S^*\\(u,v)\neq(s,s)}}
    |[\starbc]_u|\,|[\starbc]_v|
    \sum_{i=1}^{I}\bG_i^{s,s}\,|\bG_i^{u,v}|.
\]
For each pair $(u,v)$ with $u \neq v$, we apply the  Cauchy--Schwarz inequality 
over $i$:
\begin{align*}
    \sum_{i=1}^{I}\bG_i^{s,s}\,|\bG_i^{u,v}|
    &=
    \sum_{i=1}^{I}
    \bG_i^{s,s}
    \sqrt{\bG_i^{u,u}\bG_i^{v,v}}
    \cdot
    \sqrt{\frac{|\bG_i^{u,v}|^2}{\bG_i^{u,u}\bG_i^{v,v}}}
    \leq
    \sqrt{
        \sum_{i=1}^{I}
        (\bG_i^{s,s})^2\bG_i^{u,u}\bG_i^{v,v}
    }
    \cdot
    \sqrt{
        \sum_{i=1}^{I}
        \frac{|\bG_i^{u,v}|^2}{\bG_i^{u,u}\bG_i^{v,v}}
    }.
\end{align*}
By Assumption~\ref{ass:bounded}, $\bG_i^{u,u} \leq M$ and
$\bG_i^{v,v} \leq M$, so:
\[
    \sqrt{\sum_{i=1}^{I}(\bG_i^{s,s})^2\bG_i^{u,u}\bG_i^{v,v}}
    \;\leq\;
    M\sqrt{\sum_{i=1}^{I}(\bG_i^{s,s})^2}.
\]
By Definition~\ref{def:coherence}:
\[
    \sqrt{
        \sum_{i=1}^{I}
        \frac{|\bG_i^{u,v}|^2}{\bG_i^{u,u}\bG_i^{v,v}}
    }
    \;\leq\;
    \sqrt{I\,\mu_{\bG}}.
\]
Therefore, summing over all pairs $(u,v) \in S^* \times S^*$
with $(u,v) \neq (s,s)$ and using
$\sum_{u,v \in S^*}|[\starbc]_u||[\starbc]_v| =
\left(\sum_{u \in S^*}|[\starbc]_u|\right)^2 \leq j\norm{\starbc}^2$
by the Cauchy--Schwarz inequality:
\[
    \left|
    \sum_{i=1}^{I}\bG_i^{s,s}
    \sum_{\substack{u,v \in S^*\\(u,v)\neq(s,s)}}
    \bG_i^{u,v}[\starbc]_u[\starbc]_v
    \right|
    \;\leq\;
    |S^*|\,M\,\sqrt{I\,\mu_{\bG}}\,\norm{\starbc}^2\,
    \sqrt{\sum_{i=1}^{I}(\bG_i^{s,s})^2}.
\]
For the noise term, by Assumption~\ref{ass:bounded} and
$|\epsilon_i| \leq \epsilon$:
\[
    \left|\sum_{i=1}^{I}\bG_i^{s,s}\,\epsilon_i\right|
    \;\leq\;
    I\,M\,\epsilon.
\]
Combining and denoting $D_s := \frac{1}{I}\sum_{i=1}^{I}
(\bG_i^{s,s})^2$, and dividing by $I$:
\[
    \frac{1}{I}\sum_{i=1}^{I}\bG_i^{s,s}\,\eta_i^{(0)}
    \;\geq\;
    [\starbc]_s^2\,D_s
    \;-\;
    |S^*|\,M\,\sqrt{\mu_{\bG}\,D_s}\,\norm{\starbc}^2
    \;-\;
    M\,\epsilon.
\]
Assuming that
\[
    \epsilon
    \;\leq\;
    \epsilon_1^*
    \;:=\;
    \min_{s \in S^*}
    \left(
         \frac{[\starbc]_s^2\,D_s}{M}
    \;-\;
    |S^*|\,\sqrt{\mu_{\bG}\,D_s}\,\norm{\starbc}^2
    \right)
    \;>\; 0,
\]
where $\epsilon_1^* > 0$ follows from Assumption~\ref{ass:signal},
we obtain for all $s \in S^*$:
\begin{equation}\label{eq:lb_s}
    \frac{\Bigl(\frac{1}{I}\sum_{i}\bG_i^{s,s}\,
          \eta_i^{(0)}\Bigr)^2}{D_s}
    \;\geq\;
    \left(
        [\starbc]_s^2\,\sqrt{D_s}
        -
        |S^*|\,M\,\sqrt{\mu_{\bG}}\,\norm{\starbc}^2
        -
        \frac{M\,\epsilon}{\sqrt{D_s}}
    \right)^2.
\end{equation}

\medskip
\noindent\textbf{Step 2. Upper bound for $t \notin S^*$.}
By the same argument with $D_t := \frac{1}{I}
\sum_{i=1}^{I}(\bG_i^{t,t})^2$:
\[
    \left|
    \frac{1}{I}\sum_{i=1}^{I}\bG_i^{t,t}\,\eta_i^{(0)}
    \right|
    \;\leq\;
    |S^*|\,M\,\sqrt{\mu_{\bG}}\,\norm{\starbc}^2\,
    \sqrt{D_t}
    \;+\;
    M\,\epsilon.
\]
Squaring and dividing by $D_t$:
\begin{equation}\label{eq:ub_t}
    \frac{\Bigl(\frac{1}{I}\sum_{i}\bG_i^{t,t}\,
          \eta_i^{(0)}\Bigr)^2}{D_t}
    \;\leq\;
    \left(
        |S^*|\,M\,\sqrt{\mu_{\bG}}\,\norm{\starbc}^2
        +
        \frac{M\,\epsilon}{\sqrt{D_t}}
    \right)^2.
\end{equation}

\medskip
\noindent\textbf{Step 3. The gap bound.}
Subtracting \eqref{eq:ub_t} from \eqref{eq:lb_s}:
\begin{align*}
    &\frac{\Bigl(\frac{1}{I}\sum_{i}\bG_i^{s,s}\,
           \eta_i^{(0)}\Bigr)^2}{D_s}
    -
    \frac{\Bigl(\frac{1}{I}\sum_{i}\bG_i^{t,t}\,
           \eta_i^{(0)}\Bigr)^2}{D_t}
    \\
    &\;\geq\;
    \Biggl(
        [\starbc]_s^2\,\sqrt{D_s}
        -
        2\,|S^*|\,M\,\sqrt{\mu_{\bG}}\,\norm{\starbc}^2
        -
        M\,\epsilon\left(\frac{1}{\sqrt{D_s}} +
        \frac{1}{\sqrt{D_t}}\right)
    \Biggr)
    \\
    &\quad\times
    \Biggl(
        [\starbc]_s^2\,\sqrt{D_s}
        +
        M\,\epsilon\left(\frac{1}{\sqrt{D_t}} -
        \frac{1}{\sqrt{D_s}}\right)
    \Biggr)
    \;=:\; f_{s,t}(\epsilon).
\end{align*}
At $\epsilon = 0$:
\[
    f_{s,t}(0)
    \;=\;
    \left(
        [\starbc]_s^2\,\sqrt{D_s}
        -
        2\,|S^*|\,M\,\sqrt{\mu_{\bG}}\,\norm{\starbc}^2
    \right)
    \cdot
    [\starbc]_s^2\,\sqrt{D_s}
    \;>\; 0,
\]
under Assumption~\ref{ass:signal}. Since $f_{s,t}(\epsilon)$
is quadratic in $\epsilon$ with $f_{s,t}(0) > 0$, there
exists $\epsilon^*_{s,t} \in (0,+\infty]$ such that
$f_{s,t}(\epsilon) > 0$ for any $\epsilon \in
[0, \epsilon^*_{s,t})$. Taking
\begin{equation}\label{eq:epsilon2}
    \epsilon_2^*
    \;:=\;
    \min\left\{\epsilon_1^*,\,
        \min_{s \in S^*,\, t \notin S^*}
        \epsilon^*_{s,t},\, R
    \right\}
    \;>\; 0
\end{equation}
for some sufficiently large $R > 0$, and
$\delta^* := \min_{s \in S^*,\, t \notin S^*}
f_{s,t}(\epsilon_2^*) > 0$, we have for all
$\epsilon \leq \epsilon_2^*$:
\[
    \frac{\Bigl(\frac{1}{I}\sum_{i}\bG_i^{s,s}\,
           \eta_i^{(0)}\Bigr)^2}{D_s}
    -
    \frac{\Bigl(\frac{1}{I}\sum_{i}\bG_i^{t,t}\,
           \eta_i^{(0)}\Bigr)^2}{D_t}
    \;\geq\;
    \delta^*
    \;>\; 0
\]
for any $s \in S^*$ and $t \notin S^*$.
\end{proof}

\begin{lemma}[Support stability for $\ell \geq 1$]
\label{lem:support_stability}
Suppose
$S^* \subseteq \tilde{\mathcal{I}}^{\ell+1}$ and
$\norm{\overline{\mathbf{c}}^{(\ell+1)} - \starbc} \leq r$
where $r < \frac{1}{2}\min_{s \in S^*}|[\starbc]_s|$.
Then $S^* \subseteq \mathcal{I}^{\ell+1}$.
\end{lemma}

\begin{proof}
We show that after the shrinking step of Algorithm~\ref{QSPalgo}
applied to $\overline{\mathbf{c}}^{(\ell+1)}$, all indices in $S^*$
are retained in $\mathcal{I}^{\ell+1}$.
Since $\supp(\overline{\mathbf{c}}^{(\ell+1)}) \subseteq
\tilde{\mathcal{I}}^{\ell+1} \supseteq S^*$ and
$\norm{\overline{\mathbf{c}}^{(\ell+1)} - \starbc} \leq r$,
for any $s \in S^*$:
\[
    |[\overline{\mathbf{c}}^{(\ell+1)}]_s|
    \;\geq\;
    |[\starbc]_s|
    - |[\overline{\mathbf{c}}^{(\ell+1)}]_s - [\starbc]_s|
    \;\geq\;
    |[\starbc]_s| - \norm{\overline{\mathbf{c}}^{(\ell+1)}
    - \starbc}_2
    \;\geq\;
    |[\starbc]_s| - r
    \;>\; r.
\]
Meanwhile, for any $t \notin S^*$, since $[\starbc]_t = 0$:
\[
    |[\overline{\mathbf{c}}^{(\ell+1)}]_t|
    \;\leq\;
    \norm{\overline{\mathbf{c}}^{(\ell+1)} - \starbc}_2
    \;\leq\;
    r
    \;<\;
    \min_{s \in S^*}|[\starbc]_s| - r
    \;\leq\;
    |[\overline{\mathbf{c}}^{(\ell+1)}]_s|,
    \quad \forall\, s \in S^*.
\]
The shrinking step of Algorithm~\ref{QSPalgo} keeps the $j$
indices with largest absolute values of
$\overline{\mathbf{c}}^{(\ell+1)}$. Since every $s \in S^*$
satisfies $|[\overline{\mathbf{c}}^{(\ell+1)}]_s| > r$,  $\mathcal{I}^{\ell+1} \supseteq S^*$.
\end{proof}

\subsection{Proof of the stability of support}
\begin{proof}
\noindent\textbf{Part (1).}
By Lemma~\ref{lem:support_init} and $\epsilon 
<\epsilon_2^*$, we have $q_s^{(0)} < q_t^{(0)}$ for
all $s \in S^*$ and $t \notin S^*$. Since $j \geq
|S^*|$, the initialization step of
Algorithm~\ref{QSPalgo} selects the $j$ indices with
smallest $|q_s^{(0)}|$, which includes all of $S^*$.
Therefore $S^* \subseteq \mathcal{I}^0$.

\medskip
\noindent\textbf{Part (2).}
Suppose $S^* \subseteq \mathcal{I}^\ell$ for some
$\ell \geq 0$. The expansion step gives $S^* \subseteq
\mathcal{I}^\ell \subseteq \tilde{\mathcal{I}}^{\ell+1}$.
By \eqref{eq:approx_cond}:
\[
    \norm{\overline{\mathbf{c}}^{(\ell+1)} - \starbc}^2
    \;<\;
    \frac{\min_{s \in S^*}\sqrt{D_s}}
         {4\,\max_{s \in S^*}\sqrt{D_s}}
    \cdot
    \min_{s \in S^*}[\starbc]_s^2.
\]
Taking square roots:
\[
    \norm{\overline{\mathbf{c}}^{(\ell+1)} - \starbc}
    \;<\;
    \sqrt{\frac{\min_{s \in S^*}\sqrt{D_s}}
              {4\,\max_{s \in S^*}\sqrt{D_s}}}
    \cdot
    \min_{s \in S^*}|[\starbc]_s|
    \;\leq\;
    \frac{1}{2}\min_{s \in S^*}|[\starbc]_s|,
\]
 So the hypothesis of
Lemma~\ref{lem:support_stability} is satisfied, and
$S^* \subseteq \mathcal{I}^{\ell+1}$.
\end{proof}

\bibliographystyle{siam}
\bibliography{sample}
\end{document}